\documentclass[12pt]{amsart}

\usepackage{amsmath,amssymb,amsthm,xcolor,hyperref,enumerate}

\usepackage[margin=1in]{geometry}

\newtheorem{theorem}{Theorem}
\newtheorem{lemma}[theorem]{Lemma}
\newtheorem{corollary}[theorem]{Corollary}
\newtheorem{proposition}[theorem]{Proposition}

\theoremstyle{definition} \newtheorem{definition}[theorem]{Definition}
\theoremstyle{remark} \newtheorem*{remark}{Remark}
\theoremstyle{remark} \newtheorem*{example}{Example}
\theoremstyle{remark} \newtheorem*{notation}{Notation}

\numberwithin{equation}{section}
\numberwithin{theorem}{section}

\title[Uniform bounds on extensions]{Uniform exponent bounds on the number of primitive extensions of number fields}

\author{Robert J. Lemke Oliver}
\address{Department of Mathematics, Tufts University, Medford, MA 02155, USA}
\email{robert.lemke\_{}oliver@tufts.edu}

\begin{document}
	\maketitle
	
	\begin{abstract}
	A folklore conjecture asserts the existence of a constant $c_n > 0$ such that $\#\mathcal{F}_n(X) \sim c_n X$ as $X\to \infty$, where $\mathcal{F}_n(X)$ is the set of degree $n$ extensions $K/\mathbb{Q}$ (say, inside a fixed choice of $\overline{\mathbb{Q}}$) with discriminant bounded by $X$.  This conjecture is known if $n \leq 5$, but even the weaker conjecture that there exists an absolute constant $C\geq 1$ such that $\#\mathcal{F}_n(X) \ll_n X^C$ remains unknown and apparently out of reach.
	Here, we make progress on this weaker conjecture (which we term the ``uniform exponent conjecture'') in two ways.  First, we reduce the general problem to that of studying relative extensions of number fields whose Galois group is an almost simple group in a natural primitive permutation representation.  Second, for almost all such groups, we prove the strongest known upper bound on the number of such extensions.  These bounds have the effect of resolving the uniform exponent conjecture for solvable groups, sporadic groups, exceptional groups, and classical groups of bounded rank.
	\end{abstract}

\section{Introduction}
	
	Let $n \geq 2$ and for any $X > 1$, let
		\[
			\mathcal{F}_n(X) := \{ K/\mathbb{Q} : [K:\mathbb{Q}]=n, |\mathrm{Disc}(K)| \leq X\}
		\]
	be the set of degree $n$ field extensions $K/\mathbb{Q}$ (say, inside a fixed choice of algebraic closure $\overline{\mathbb{Q}}$) with absolute discriminant $\mathrm{Disc}(K)$ bounded by $X$.  A folklore conjecture
	asserts the existence of a constant $c_n > 0$ such that $\#\mathcal{F}_n(X) \sim c_n X$ as $X\to \infty$.  This conjecture is known only for $n \leq 5$ \cite{DavenportHeilbronn,Bhargava-Quartic,CohenDiazyDiazOlivier,Bhargava-Quintic}, with the best upper bounds for $n \geq 6$ being found in the papers \cite{BSW,LOThorne}.  In particular, the best known bound for $n \geq 95$ is of the form $\#\mathcal{F}_n(X) \ll_n X^{c (\log n)^2}$ for an explicit constant $c$, and thus even the weaker conjecture that there exists an absolute constant $C\geq 1$ such that $\#\mathcal{F}_n(X) \ll_n X^C$ remains unknown and apparently out of reach.
	
	In this paper, we make progress on this weaker conjecture (which we term the ``uniform exponent conjecture'') in two ways.  First, we reduce the general problem to that of studying relative extensions of number fields whose Galois group is an almost simple group in a natural primitive permutation representation.  (Recall that a finite almost simple group is a group $G$ such that $T \unlhd G \subseteq \mathrm{Aut}(T)$ for some finite nonabelian simple group $T$, referred to as the socle of $G$.)  Second, for almost all such groups, we prove the strongest known upper bound on the number of such extensions.  To describe this work, we begin by fixing some further notation and discussing some motivating examples.
	
	Let $G$ be a transitive permutation group of degree $n$ and let $k$ be a number field.  Let 
		\[
			\mathcal{F}_{n,k}(X;G) := \{ K/k : [K:k] = n, \mathrm{Gal}(\widetilde{K}/k) \simeq G, |\mathrm{Disc}(K)| \leq X\},
		\]
	where $\widetilde{K}$ is the normal closure of $K$ over $k$ and where we regard $\mathrm{Gal}(\widetilde{K}/k) $ as a permutation group by considering its action on the $n$ embeddings $K \hookrightarrow \mathbb{C}$ lifting a fixed embedding of $k$.  Our first theorem effectively resolves the uniform exponent conjecture for finite almost simple groups of Lie type with bounded rank.  
		
	\begin{theorem}\label{thm:lie-type-intro}
		There is an absolute constant $C$ such that if $G$ is a finite almost simple group of Lie type (say, of rank $m$ over $\mathbb{F}_q$) in any of its faithful and transitive permutation representations, then $\#\mathcal{F}_{n,k}(X;G) \ll_{n,[k:\mathbb{Q}]} X^{Cm}$ for any number field $k$ and any $X \geq 1$, where $n = \deg G$.
	\end{theorem}

	The theorem is explicit, but depends slightly on the particular group in question.  For example, it follows from Theorem \ref{thm:classical-bound} below that if $G$ is an almost simple classical group with natural module $\mathbb{F}_q^m$, in its action on totally singular subspaces of $\mathbb{F}_q^m$ with dimension $1$, then $\#\mathcal{F}_{n,k}(X;G) \ll_{n,[k:\mathbb{Q}]} X^{5m+6}$.  In particular, this exponent remains constant as $q$ tends to infinity with $m$ bounded, and it is in this sense that we regard Theorem \ref{thm:lie-type-intro} as resolving the uniform exponent conjecture for groups of bounded rank.  These are the first families of almost simple groups where such a result is known.  (Recall that an almost simple classical group is an almost simple group with socle of the form $\mathrm{PSL}_{m}(\mathbb{F}_q)$, $\mathrm{PSp}_{2m}(\mathbb{F}_q)$, $\mathrm{PSU}_{m}(\mathbb{F}_q)$, $\mathrm{P\Omega}_{2m+1}(\mathbb{F}_q)$, or $\mathrm{P\Omega}^{\pm}_{2m}(\mathbb{F}_q)$.  Stronger bounds are available for most of these socle types.)  
	
	In the alternate regime, where $m$ tends to infinity with $q$ bounded, the exponent in the bound is qualitatively of the form $O(\log n)$, which of course is not bounded.  However, we show that this is essentially the worst case of our method, apart from alternating and symmetric groups.  
	
	\begin{theorem}\label{thm:primitive-bound-improvement}
		There is an absolute constant $C$ so that for any primitive group $G$ of degree $n$, $G \ne S_n,A_n$, any number field $k$, and any $X \geq 1$, there holds $\#\mathcal{F}_{n,k}(X;G) \ll_{n,[k:\mathbb{Q}]} X^{ C \cdot \log n}$.
	\end{theorem}
	
	(For $A_n$ and $S_n$, it is difficult to improve over the general bound on all degree $n$ extensions; see Theorem \ref{thm:degree-n-bound-intro} for the best known such result.)
	
	Taking $k=\mathbb{Q}$, Theorem \ref{thm:primitive-bound-improvement} in particular represents an improvement over the ``trivial'' bound $\#\mathcal{F}_{n,\mathbb{Q}}(X;G) \leq \#\mathcal{F}_n(X) \ll_n X^{c (\log n)^2}$ coming from \cite{LOThorne}.  This trivial bound was the bottleneck for large $n$ in Bhargava's recent proof of van der Waerden's conjecture \cite{Bhargava-vdW} on Galois groups of random polynomials, and indeed, using his machinery, we obtain the following improvement over his bound \cite[Corollary 3]{Bhargava-vdW}.
	
	\begin{corollary}\label{cor:vdw-improvement}
		There is a constant $c>0$ such that for any $n \geq 5$, the number of degree $n$ monic, irreducible polynomials $f \in \mathbb{Z}[x]$ with coefficients bounded by $H$ in absolute value for which $\mathrm{Gal}(f) \ne S_n,A_n$ is $O_n(H^{n - \frac{c n}{\log n}})$ as $H \to \infty$.
	\end{corollary}
	
	As suggested above, Theorem \ref{thm:primitive-bound-improvement} essentially follows from a worst case analysis of Lie-type groups with large rank.  For groups that are rather different from these, our methods are typically much stronger.  As a simple demonstration of this, our next theorem resolves the uniform exponent conjecture for solvable extensions.
	
	\begin{theorem}\label{thm:solvable-intro}
		There is an absolute constant $C$ such that for any solvable transitive permutation group $G$ (say, of degree $n$), we have $\#\mathcal{F}_{n,k}(X;G) \ll_{n,[k:\mathbb{Q}]} X^C$ for any number field $k$ and any $X \geq 1$.  In fact, $C = 14$ is admissible.
	\end{theorem}
	
	Perhaps surprisingly, the proof of Theorem \ref{thm:solvable-intro} does not rely on class field theory in any way.  Indeed, nothing like Theorem \ref{thm:solvable-intro} was known previously even using class field theory, with the best general bounds available in the literature being those of Alberts \cite{Alberts}.  In the special case of nilpotent extensions, rather stronger bounds are available \cite{KlunersMalle,KlunersWang}.
	
	Our last example is less systematic than either Theorem \ref{thm:lie-type-intro} or Theorem \ref{thm:solvable-intro}, but is intended to convey both the explicit nature of our work and the fact that, while our methods ultimately rely on properties of finite groups, these properties are not particularly inaccessible.
	
	\begin{theorem}\label{thm:monster-intro}
		Let $\mathbb{M}$ be the Fischer--Griess Monster group in its minimal degree $n = 97{,}239{,}461{,}142{,}009{,}186{,}000 \approx 9.7 \cdot 10^{19}$ permutation representation.  Then $\#\mathcal{F}_{n,k}(X;\mathbb{M}) \ll_{k} X^{17/2}$ for any number field $k$ and any $X \geq 1$.
	\end{theorem}
	
	By contrast, the previously best bound on $\#\mathcal{F}_{n,k}(X;\mathbb{M})$ follows from the general bound \cite[Theorem 1.2]{LOThorne} on $\#\mathcal{F}_n(X)$, and is of the form $O(X^{1462})$ even in the simplest case when $k=\mathbb{Q}$.  Additionally, we note that the Monster group is known to occur as a Galois group over $\mathbb{Q}$ \cite{Thompson-Galois,Matzat-Galois}, in fact for infinitely many extensions, so this result is not vacuous.
	
	We now turn to stating a general result.  Any transitive permutation group $G$ may be decomposed in terms of primitive groups, and thus one may hope to study $G$-extensions by studying how the fields decompose into the corresponding primitive pieces.  More explicitly, we say that $G$ has a \emph{primitive tower type} $(G_1,\dots,G_m)$ if each $G_i$ is a primitive permutation group and $G$ is a transitive permutation subgroup of the wreath product $G_1 \wr \dots \wr G_m$.  Equivalently, $(G_1,\dots,G_m)$ is a primitive tower type for $G$ if for every $G$-extension $K/k$, there exist extensions $F_0/F_1/\dots/F_m$ with $F_0 = K$ and $F_m=k$ such that $F_{i-1}/F_i$ is a $G_i$-extension.
	
	Finite primitive permutation groups are described by the O'Nan--Scott theorem (see Theorem \ref{thm:onan-scott}), and for a finite primitive permutation group $G$ of degree $n$, we define a quantity $\widetilde{\mathrm{rk}}(G)$ by:
		\begin{itemize}
			\item $\widetilde{\mathrm{rk}}(G) = (\log n)^2$ if $G = A_n$ or $S_n$, in the natural degree $n$ action;
			\item $\widetilde{\mathrm{rk}}(G) = m$ if $G$ is an almost simple classical group of rank $m$;
			\item $\widetilde{\mathrm{rk}}(G) = 1$ if $G$ is an exceptional or sporadic almost simple group, or if $G$ is solvable;
			\item $\widetilde{\mathrm{rk}}(G) = 1 + \frac{m}{p}$ if $G$ is a non-solvable affine group with socle $\mathbb{F}_p^m$; and
			\item $\widetilde{\mathrm{rk}}(G) = \frac{(\log n)^3}{\sqrt{n}}$ if $G$ is any other primitive group.
		\end{itemize}
	For a transitive imprimitive group $G$, we define $\widetilde{\mathrm{rk}}(G)$ to be the minimum over primitive tower types $(G_1,\dots,G_m)$ for $G$ of $\max\{\widetilde{\mathrm{rk}}(G_i)\}$.	For example, if $G$ is solvable, then $G$ has a primitive tower type with every $G_i$ solvable, and thus $\widetilde{\mathrm{rk}}(G)=1$ for solvable groups.  We then have the following.
	
	\begin{theorem}\label{thm:general-bound-intro}
		There is an absolute constant $C$ such that for any transitive permutation group $G$ (say, of degree $n$), we have $\#\mathcal{F}_{n,k}(X;G) \ll_{n,[k:\mathbb{Q}]} X^{ C \cdot \widetilde{\mathrm{rk}}(G)}$ for any number field $k$ and any $X \geq 1$.
	\end{theorem}
	
	We note the following simple corollary to Theorem \ref{thm:general-bound-intro}.
	
	\begin{corollary}
		For any $m \geq 1$, let $\mathcal{G}_m$ be the set of finite groups all of whose alternating composition factors $A_n$ satisfy $(\log n)^2 \leq m$, all of whose classical composition factors have rank at most $m$, and whose cyclic, exceptional, and sporadic composition factors are arbitrary.  There exists an absolute constant $C$ such that for any $m \geq 1$ and any $G \in \mathcal{G}_m$, we have $\#\mathcal{F}_{n,k}(X;G) \ll_{n,[k:\mathbb{Q}]} X^{C m}$ for any number field $k$, any $X \geq 1$, and any permutation representation of $G$, where $n = \deg G$.
	\end{corollary}
	
	As is suggested by the notion of a primitive tower type, we approach the proof of Theorem \ref{thm:general-bound-intro} via induction.  The induction we use is very similar in spirit to one used by Schmidt \cite{Schmidt}, and is partially why we find it convenient to order fields in $\mathcal{F}_{n,k}(X;G)$ by their absolute discriminants rather than their relative discriminants.  This induction allows us to reduce the general problem to that of primitive extensions (see Proposition \ref{prop:induction}), provided we obtain a sufficient saving in terms of the discriminant of $k$ (and which saving we have neglected to include in the statements of Theorems \ref{thm:lie-type-intro}--\ref{thm:general-bound-intro}).  Thus, we work uniformly over $k$, tracking the dependence of our estimates on the discriminant of $k$.  In fact, where it makes sense, we also track the dependence of our estimates on the parameters $n$, $G$, and $[k:\mathbb{Q}]$.  As an example of this, we have the following fully explicit form of the main result from \cite{LOThorne} that applies over an arbitrary number field, and incorporates a minor improvement into the argument.
	
	\begin{theorem}\label{thm:degree-n-bound-intro}
		There are absolute positive constants $c$, $c^\prime$, and $c^{\prime \prime}$ such that for any number field $k$, any $n \geq 3$, and any $X \geq 1$, there holds
			\[
				\#\mathcal{F}_{n,k}(X)
					\leq (2 dn^3)^{c^{\prime\prime} dn (\log n)^2} X^{ c (\log n)^2} |\mathrm{Disc}(k)|^{-c^\prime n \log n},
			\]
		where $\mathcal{F}_{n,k}(X) := \{ K/k : [K:k]=n, |\mathrm{Disc}(K)| \leq X\}$ and $d=[k:\mathbb{Q}]$.  The values $c=1.487$, $c^\prime = 0.159$, and $c^{\prime\prime} = 1.847$ are admissible.
	\end{theorem}
	
	As a corollary, we deduce a completely explicit bound on the set $\mathcal{F}(X) = \cup_{n \geq 2} \mathcal{F}_{n,\mathbb{Q}}(X)$.
	
	\begin{corollary}\label{cor:all-degree-bound}
		For any $X \geq 10^6$, we have $\#\mathcal{F}(X) \leq 2 \cdot X^{9 (\log\log X)^3}$.
	\end{corollary}
	
	In treating Theorem \ref{thm:general-bound-intro} via induction over primitive tower types, it will turn out that the case of primary interest will be almost simple groups $G$ in specific primitive permutation representations that we term ``natural.''  For other primitive representations of these groups, and primitive representations of other kinds of groups, we use ideas similar to those behind \cite[Theorem 23]{Bhargava-vdW} and \cite[Proposition 1.3]{EV} to reduce to the natural, almost simple setting.  The description of the natural representations depends on the socle type of $G$ (see Definition \ref{def:natural}), but for example the representations of almost simple classical groups on totally singular subspaces of dimension $1$ are natural, as are the degree $n$ representations of the symmetric and alternating groups $S_n$ and $A_n$.  For such natural representations, we use an idea originating in work of Ellenberg and Venkatesh \cite{EV} that has also been used in \cite{Dummit,Couveignes,LOThorne}, that the number of $G$-extensions may be controlled by means of the invariant theory of $G$.  The work closest to ours is that of Dummit \cite{Dummit}, who provides bounds on $G$-extensions in terms of the degrees of the primary invariants of $G$.  A key difference in our work is that we instead make use only of a set of $n$ algebraically independent invariants of $G$, without regard to whether they form a set of primary invariants.  There are algorithms to compute sets of primary invariants \cite{DerksenKemper}, but it remains a computational challenge to find primary invariants even for groups of modest size.  By contrast, for almost simple groups in their natural representations, we construct an explicit set of algebraically independent invariants with small degree.  This explicit construction leads to the following.
	
	\begin{theorem}\label{thm:minimal-invariants-intro}
		There exists an absolute constant $C$ such that if $G$ is a finite almost simple group in a natural primitive permutation representation of degree $n$ (see Definition \ref{def:natural}), then there exists a set of $n$ algebraically independent $G$-invariants $\{f_1,\dots,f_n\} \subseteq \mathbb{Z}[x_1,\dots,x_n]^G$ satisfying $\max \{ \deg f_i \} \leq C \frac{ \log |G|}{\log n}$, where $n = \deg G$.
	\end{theorem}
	
	(Our proof of Theorem \ref{thm:minimal-invariants-intro} proceeds by an analysis of the different socle types, and thus makes use of the classification of finite simple groups.)
	
	The construction we use in Theorem \ref{thm:minimal-invariants-intro} is similar in spirit to constructions used for \emph{relative} invariants \cite{Girstmair}, but the problem of determining the smallest possible degrees of a set of $n$ algebraically independent invariants was apparently previously unexplored.  The bound $C \log |G| / \log n$ is natural in this context, as an orbit counting argument suggests that any set of $n$ algebraically independent invariants should satisfy $\max \{\deg f_i\} \gg \log |G| / \log n$.  We thus believe Theorem \ref{thm:minimal-invariants-intro} to be essentially optimal, up to the value of the constant $C$.  We also expect the conclusion of Theorem \ref{thm:minimal-invariants-intro} to hold for any primitive permutation group $G$ regardless of whether it is almost simple or in a natural representation.  This would lead to better bounds on $\#\mathcal{F}_{n,k}(X;G)$ for certain affine groups $G$.
	
	Another problem concerning primitive permutation groups where the ratio $\log |G| / \log n$ arises is Pyber's conjecture \cite{Pyber} on the size of a base for $G$; this conjecture was recently resolved in a series of papers (e.g., \cite{LiebeckShalev-Pyber} and \cite{BurnessSeress-Pyber}, with \cite{DuyanHalasiMaroti-Pyber} completing the proof).  In fact, this connection is more than superficial, and allows us to restrict our attention in the proof of Theorem \ref{thm:minimal-invariants-intro} to classical groups.  We defer further discussion of this connection to Section \ref{sec:invariant-theory} (where we prove Theorem \ref{thm:minimal-invariants-intro}), apart from noting our hope that the ideas used to resolve Pyber's conjecture might be used to give better explicit bounds than we provide, particularly for exceptional and sporadic groups.
	
\subsection{Organization}

	This paper is organized as follows.  
	
	In Section \ref{sec:preliminaries}, we recall and prove some preliminary facts about permutation groups and the geometry of numbers.  
	We also give the inductive argument we use to reduce the general problem of bounding $\mathcal{F}_{n,k}(X;G)$ to the situation when $G$ is primitive; see Proposition \ref{prop:induction}.  This motivates the level of uniformity we pursue in our estimates.  Using this, we also provide an explicit form of a widely used result of Schmidt \cite{Schmidt} mentioned earlier, that in fact extracts a greater saving in terms of the discriminant of the base field.
	
	In Section \ref{sec:general-bounds}, we provide several general bounds on $\mathcal{F}_{n,k}(X;G)$ in terms of the degrees of algebraically independent invariants of $G$, with the invariants being treated as a black box.  These bounds are uniform with respect to the discriminant of $k$, as is required for use in the inductive argument, and are typically explicit.  We have not put much care into optimizing the explicit constants arising, but for future applications we have in mind it is necessary to understand the dependence of these estimates on all parameters at hand.  There is little additional cost to making the estimates explicit, so we have done so.  We also give a version of Bhargava's modification \cite[Theorem 20]{Bhargava-vdW} to Schmidt's theorem that applies over a general number field (see Corollary \ref{cor:bhargava-number-field}).
	
	In Section \ref{sec:invariant-theory}, we introduce a general construction of algebraically independent invariants for finite permutation groups $G$, and we use this construction to prove Theorem \ref{thm:minimal-invariants-intro} in an explicit form.  In particular, we give a fully explicit upper bound on the ``invariant degree'' of every almost simple group in a natural representation, depending on its socle type.  For the Mathieu groups, Lemma \ref{lem:mathieu-invariants} gives a provably minimal set of invariants.
	
	In Section \ref{sec:natural-bounds}, we assemble the results of the previous two sections to provide explicit bounds on $\#\mathcal{F}_{n,k}(X;G)$ when $G$ is an almost simple group in a natural primitive representation.  In particular, we prove Theorem \ref{thm:degree-n-bound-intro} in a slightly more explicit form (see Theorem \ref{thm:alternating-symmetric-bound}) and we provide the strongest known upper bounds on $\#\mathcal{F}_{n,k}(X)$ for every $n\geq 20$ (or every $n \geq 23$ when $k=\mathbb{Q}$, given the improvements of \cite{BSW} over Schmidt's work); see Theorem \ref{thm:small-degree}.
	
	In Section \ref{sec:non-minimal}, we show how bounds on $\#\mathcal{F}_{n,k}(X;G)$ for arbitrary primitive groups $G$ can be inferred from those on almost simple groups in natural representations, which in particular shows that the uniform exponent conjecture can be reduced to studying natural almost simple groups. 
	
	In Section \ref{sec:proofs-intro}, we provide proofs of the remaining theorems from the introduction.  With the previous results in hand, this is essentially a matter of bookkeeping.
	
\section*{Acknowledgements}
	This paper grew out of conversations with Manjul Bhargava surrounding his remarkable proof of van der Waerden's conjecture.  The author wishes to express his gratitude for these (and many other) conversations.  The author would also like to thank Tim Burness, Gergely Harcos, Gunter Malle, Andrew O'Desky, and Evan O'Dorney for comments on an earlier version of this work.
	
	The author was funded by a grant from the National Science Foundation (DMS-2200760) and by a Simons Foundation Fellowship in Mathematics.

\section{Preliminaries}
	\label{sec:preliminaries}

	In this section, we record several preliminary results on primitive groups, geometry of numbers, and bounding number fields.

\subsection{Primitive permutation groups} 
	\label{subsec:permutation-groups}
	
	Recall that a permutation group $G$ is called primitive if it preserves no non-trivial partition of the underlying set.  The O'Nan--Scott theorem describes all finite primitive groups in terms of the \emph{socle} of $G$, that is, the subgroup generated by the minimal normal subgroups of $G$.
	
	\begin{theorem}\label{thm:onan-scott}[O'Nan--Scott]
		Let $G$ be a finite primitive permutation group of degree $n$ with socle $N$.  Then $G$ and $N$ are one of the following types:
			\begin{enumerate}[i)]
				\item \emph{Almost simple:} $N$ is a non-abelian simple group and $G \subseteq \mathrm{Aut}(N)$.
				\item \emph{Affine, of type $\mathbb{F}_p^m$:} $N \simeq (\mathbb{Z}/p\mathbb{Z})^m$ for some prime $p$ and some $m \geq 1$, $n=p^m$, and $G \simeq G_0 \ltimes N$, where $G_0$ is subgroup of $\mathrm{GL}_m(\mathbb{F}_p)$ acting irreducibly on $N$.
				\item \emph{Diagonal:} $N \simeq T^m$ for some nonabelian simple group $T$ and some $m \geq 2$, and $n=|T|^{m-1}$.
				\item \emph{Product:} $N \simeq T^m$ for some nonabelian simple group $T$ and some $m \geq 2$, there exists a proper divisor $d$ of $m$ and a primitive group $U$ of degree $\ell$ with socle $T^d$, $n = \ell^{m/d}$, and $G$ is a subgroup of $U \wr S_{m/d}$.
				\item \emph{Twisted wreath:} $N$ acts regularly and $N \simeq T^m$ for some nonabelian simple group $T$ and some $m \geq 6$.
			\end{enumerate}
	\end{theorem}
	
	See \cite[Chapter 4]{DixonMortimer} for more information on the various types and a greater discussion of this theorem.  For our purposes, we will also make use of a coarser stratification of primitive groups.
	
	\begin{definition} \label{def:basic-elemental}
		A primitive group $G$ is called \emph{basic} if it is of almost simple or diagonal type, or of affine type with $G_0$ preserving no direct sum decomposition of $\mathbb{F}_p^m$.  (See \cite[\S 4.3--\S4.5]{Cameron}.) For any positive integers $m,\ell,r$ satisfying $m \geq 3$, $1 \leq \ell < m/2$, and $\max\{\ell,r\} \geq 2$, a primitive group $G$ of degree $n = \binom{m}{\ell}^r$ is called \emph{non-elemental} of type $(m,\ell,r)$ if it is isomorphic to the action of a subgroup of $S_m \wr S_r$ on $\Omega^r$, where $\Omega$ consists of the subsets of $\{1,\dots,m\}$ of size $\ell$.  A primitive group is called \emph{elemental} if it is not non-elemental of any type.
	\end{definition}
	
	The notion of elemental primitive groups was introduced in \cite{Bhargava-vdW}.  In fact, \cite[Theorem 23]{Bhargava-vdW} provides a strong bound on $\#\mathcal{F}_{n,\mathbb{Q}}(X;G)$ for non-elemental primitive groups $G$.  In Corollary \ref{cor:non-elemental-bound} below, we extend this result to bound $\#\mathcal{F}_{n,k}(X;G)$ for non-elemental $G$ and arbitrary number fields $k$.  The approach used in proving this theorem will also underlie our treatment of groups of diagonal type, and, to a lesser extent, affine groups.  Any non-basic group (e.g., a primitive group of either product or twisted wreath type) is non-elemental of some type with $r \geq 2$, and thus the bulk of our attention will be paid to almost simple groups in their elemental representations.
	
	An important distinction between elemental and non-elemental primitive groups is provided by the following definition.  
	
	\begin{definition} \label{def:index}
		If $G$ is a transitive permutation group of degree $n$, we define the \emph{index} of a non-identity element $g \in G$ by
		\[
			\mathrm{ind}(g) := n - \#\mathrm{Orb}(g),
		\]
		where $\mathrm{Orb}(g)$ is the set of orbits under the cyclic subgroup generated by $g$, and we define the index of $G$, denoted $\mathrm{ind}(G)$, to be the minimum of the index of non-identity elements of $G$, i.e. $\mathrm{ind}(G) := \min_{g \ne \mathrm{id}} \{ \mathrm{ind}(g)\}$.
	\end{definition}
	
	For a primitive group $G$ of degree $n$ not containing $A_n$, we have the lower bound $\mathrm{ind}(G) \geq \lfloor \sqrt{n} \rfloor$ \cite[Theorem 12]{Bhargava-vdW}.  For elemental primitive groups, however, there is a much stronger lower bound:
	
	\begin{lemma}\label{lem:elemental-index}
		Let $G$ be an elemental primitive group of degree $n$, and suppose that $A_n \not \subseteq G$.  Then $\mathrm{ind}(G) \geq {3n}/{14}$.  If moreover $G$ is not almost simple of type $\mathrm{PSU}_4(\mathbb{F}_2)$, $\mathrm{PSp}_{2m}(\mathbb{F}_2)$, $\mathrm{P\Omega}^+_{2m}(\mathbb{F}_2)$, or $\mathrm{P\Omega}^-_{2m}(\mathbb{F}_2)$, then $\mathrm{ind}(G) \geq n/4$.
	\end{lemma}
	\begin{proof}
		The first claim is \cite[Theorem 13]{Bhargava-vdW}, which relies on \cite{GuralnickMagaard}.  The second follows from combining \cite[Theorem 7]{BurnessGuralnick} and \cite[Theorem 7.4]{BurnessGuralnick}.
	\end{proof}
	
	We note in passing the connection between the index of a transitive permutation group $G$ and the discriminant ideal $\mathfrak{D}_{K/k}$ of a $G$-extension $K/k$ of number fields.  In particular, if $\mathfrak{p}$ is a tamely ramified prime with inertia subgroup generated by $g \in G$, then $v_\mathfrak{p}(\mathfrak{D}_{K/k}) = \mathrm{ind}(g)$.  A similar, though more complicated, expression holds for primes of wild ramification; see \eqref{eqn:discriminant} below.  Consequently, the discriminant ideal $\mathfrak{D}_{K/k}$ must be $\mathrm{ind}(G)$-powerful.\footnote{For this reason, we cautiously suggest that referring to $\mathrm{ind}(G)$ instead as the valuation of the group may be more evocative of its meaning.} These considerations suggest that $\#\mathcal{F}_{n,k}(X;G) \ll_{k,G,\epsilon} X^{\frac{1}{\mathrm{ind}(G)} + \epsilon}$, which is now part of the so-called ``weak Malle conjecture.''  See the work of Malle \cite{Malle1,Malle2} for a more careful analysis and a conjecture on the asymptotics of $\mathcal{F}_{n,k}(X;G)$.
	
	Finally, we will make use of the recent resolution of a conjecture about primitive groups known as Pyber's conjecture about the minimal size of a ``base.''  
	
	\begin{definition}\label{def:base}
		If $G$ is a transitive group of degree $n$, a \emph{base} of $G$ is a set of points whose stabilizers intersect trivially.  We define $\mathrm{base}(G)$ to be the smallest order of a base of $G$.  
	\end{definition}
	
	For example, if $G$ is cyclic (or, more generally, if $G$ acts regularly), then $\mathrm{base}(G)=1$, while if $G$ is the full symmetric group $S_n$, then $\mathrm{base}(G) = n-1$.  Since the index of a point stabilizer is $n$, one readily observes that $\mathrm{base}(G) \geq \frac{\log |G|}{\log n}$.  Pyber \cite{Pyber} conjectured that this is close to sharp when $G$ is primitive:
	
	\begin{theorem}[Pyber's conjecture] \label{thm:pyber's-conjecture}
		There is an absolute constant $C>1$ such that for any primitive group $G$ of degree $n$, $\mathrm{base}(G) \leq C \frac{\log |G|}{\log n}$.
	\end{theorem}
	\begin{proof}
		This is the culmination of a series of papers, with \cite{DuyanHalasiMaroti-Pyber} providing the coup de grace.  See also \cite{HalasiLiebeckMaroti}, which provides the explicit bound $\mathrm{base}(G) \leq 2 \frac{\log |G|}{\log n} + 24$.
	\end{proof}
	
	As might be expected, the proof of Theorem \ref{thm:pyber's-conjecture} passes through the description of primitive groups provided by the O'Nan--Scott theorem, Theorem \ref{thm:onan-scott} above.  For our purposes, we will mostly make use of Pyber's conjecture for almost simple groups, in fact mostly for almost simple groups of either exceptional or sporadic type.  For almost all sporadic groups, $\mathrm{base}(G)$ has been determined precisely in the work \cite{Base-Sporadic}, while \cite{Base-Exceptional} shows that $\mathrm{base}(G) \leq 6$ for exceptional groups $G$.  We discuss this work in somewhat more detail when needed in Section \ref{subsec:stabilizer-invariants}.

\subsection{Geometry of numbers} 
	\label{subsec:geometry-of-numbers}

	In order to obtain estimates with our desired level of uniformity, we will make use of several results from the geometry of numbers.  
	
\subsubsection{Bounds on Minkowski minima}
	Given a number field $k$ of degree $d$ and signature $(r_1,r_2)$, let $k_\infty := k \otimes_{\mathbb{Q}} \mathbb{R} \simeq \mathbb{R}^{r_1} \times \mathbb{C}^{r_2} \simeq \mathbb{R}^d$ be the standard Minkowski space associated with $k$, and let $\iota \colon k \hookrightarrow k_\infty$ be the Minkowski embedding.  We equip the Minkowski space $k_\infty$ with the usual measure, which is larger than Lebesgue measure on $\mathbb{R}^d$ by a factor $2^{r_2}$.  The ring of integers $\mathcal{O}_k$ forms a full rank lattice in $k_\infty$ under $\iota$ with covolume $\sqrt{|\mathrm{Disc}(k)|}$, and we let $\lambda_1(k),\dots,\lambda_d(k)$ denote the Minkowski successive minima of this lattice with respect to the function $\|\alpha\| := \max\{ |\alpha|_v \}$, where the maximum runs over the archimedean places of $k$.  That is, $\lambda_i(k)$ is the smallest value $\lambda$ such that set $\{\alpha \in \mathcal{O}_k : \|\alpha\| \leq \lambda\}$ spans a subspace of dimension at least $i$.
	
	Writing $\lambda_i$ for $\lambda_i(k)$, these minima satisfy $\lambda_1 = 1$, $\lambda_1 \leq \lambda_2 \leq \dots \leq \lambda_d$, and 
		\begin{equation}\label{eqn:minkowski-second}
			\frac{2^d}{2^{r_1}(2\pi)^{r_2} d!} \sqrt{|\mathrm{Disc}(k)|} \leq \lambda_1 \dots \lambda_d \leq \frac{2^d}{2^{r_1} (2\pi)^{r_2}} \sqrt{|\mathrm{Disc}(k)|},
		\end{equation}
	these last inequalities following from Minkowski's second theorem.  Notice in particular that \eqref{eqn:minkowski-second} implies that $\lambda_1 \dots \lambda_d \leq \sqrt{|\mathrm{Disc}(k)|}$.
	
	For convenience, we let $\lambda_{\max}(k)$ denote the largest of the Minkowski minima, i.e. $\lambda_{\max}(k) := \lambda_d(k)$.  We will make frequent use of the following bound on $\lambda_{\max}(k)$ for $k \ne \mathbb{Q}$.  (When $k=\mathbb{Q}$, or equivalently when $d=1$, we trivially have $\lambda_{\max}(k)=1$.)
	
	\begin{lemma} \label{lem:largest-minimum}
		For any number field $k$ of degree $d \geq 2$, we have
			\[
				\frac{1}{d} |\mathrm{Disc}(k)|^{\frac{1}{2d-2}} \ll \lambda_{\max}(k) \leq |\mathrm{Disc}(k)|^{\frac{1}{d}}.
			\]	
		The implied constant in the first inequality above is absolute, and may be taken to be $1$ if $d \geq 4$.
	\end{lemma}
	
	\begin{proof}
		The lower bound follows directly from \eqref{eqn:minkowski-second} using that $\lambda_1 =1 $ and $\lambda_i \leq \lambda_d$ for all $2 \leq i \leq d-1$.  The upper bound essentially follows from \cite[Theorem 3.1]{BSTTTZ}.  Strictly, that statement is instead a bound for the largest element of a Minkowski reduced basis for $\mathcal{O}_k$, but if we let $v_0 = 1$ and $v_1,\dots,v_{d-1}$ be the linearly independent elements of $\mathcal{O}_k$ with $\|v_i\| = \lambda_{i+1}$, the proof given there shows that $\|v_{d-1}\| \leq \|v_i\| \cdot \|v_{\pi(i)}\|$ for some permutation $\pi$ of $\{1,\dots,d-2\}$ and each $1 \leq i \leq d-2$.  Following the reasoning of that proof, we then find $\lambda_d^d \leq (\lambda_1 \dots \lambda_d)^2$, which then yields the stated upper bound upon appealing to \eqref{eqn:minkowski-second}.  
	\end{proof}
	
	Any fractional ideal $\mathfrak{C}$ of a number field $k$ with degree $d$ may also be viewed as a rank $d$ lattice in Minkowski space.  Consequently, it too has associated Minkowski minima $\lambda_1(\mathfrak{C}),\dots, \lambda_d(\mathfrak{C})$, and we have:
	
	\begin{lemma}\label{lem:minkowski-ideal}
		Let $k$ be a number field of degree $d$, and let $\mathfrak{C}$ be a fractional ideal in $k$.  Let $\lambda_1(\mathfrak{C}),\dots,\lambda_d(\mathfrak{C})$ be the successive minima of $\mathfrak{C}$ with respect to the gauge function $\|\cdot\|$ when viewed as a lattice inside Minkowski space.  Then:
			\begin{itemize}
				\item $\lambda_1(\mathfrak{C}) \leq |\mathfrak{C}|^{\frac{1}{d}} |\mathrm{Disc}(k)|^{\frac{1}{2d}}$, where $|\mathfrak{C}|$ is the norm of $\mathfrak{C}$;
				\item $\lambda_i(\mathfrak{C}) \leq \lambda_1(\mathfrak{C}) \lambda_i(k)$; and in particular
				\item $\lambda_d(\mathfrak{C}) \leq |\mathfrak{C}|^{\frac{1}{d}} |\mathrm{Disc}(k)|^{\frac{3}{2d}}$.
			\end{itemize}
	\end{lemma}
	\begin{proof}
		Analogous to \eqref{eqn:minkowski-second}, by Minkowski's second inequality, we find 
			\[
				\lambda_1(\mathfrak{C}) \dots \lambda_d(\mathfrak{C}) 
					\leq \frac{2^d}{2^{r_1}(2\pi)^{r_2}} \mathrm{covol}(\mathfrak{C}) \leq \mathrm{covol}(\mathfrak{C}).
			\]
		Since $\mathrm{covol}(\mathfrak{C}) = |\mathfrak{C}| \sqrt{|\mathrm{Disc}(k)|}$ and $\lambda_i(\mathfrak{C}) \geq \lambda_1(\mathfrak{C})$ for each $i \geq 1$, the first claim follows.  For the second, let $x \in \mathfrak{C}$ be the element with $\|x\| = \lambda_1(\mathfrak{C})$, and let $v_1,\dots,v_d \in \mathcal{O}_k$ be the linearly independent elements with $\| v_i\| = \lambda_i(k)$.  Then the elements $xv_1,\dots,xv_d$ are linearly independent elements of $\mathfrak{C}$, so for each $i$ we find that $\lambda_i(\mathfrak{C}) \leq \max\{ \|xv_1\|, \dots, \| xv_i \|\} \leq \lambda_1(\mathfrak{C}) \lambda_i(k)$.  The third claim simply follows by combining the first two claims with Lemma \ref{lem:largest-minimum}.
	\end{proof}
	
	While we prefer the explicit bound $\lambda_d(\mathfrak{C}) \leq |\mathfrak{C}|^{\frac{1}{d}} |\mathrm{Disc}(k)|^{\frac{3}{2d}}$ for most of our purposes, work of Fraczyk, Harcos, and Maga \cite[Corollary 3]{FraczykHarcosMaga} provides an often stronger bound with an inexplicit dependence on the degree $d$.
	
	\begin{lemma}[Fraczyk--Harcos--Maga] \label{lem:minkowski-ideal-asymptotic}
		Let $k$ be a number field of degree $d$ and let $\mathfrak{C}$ be a fractional ideal in $k$.  Then $\lambda_d(\mathfrak{C}) \ll_d |\mathfrak{C}|^{\frac{1}{d}} |\mathrm{Disc}(k)|^{\frac{1}{d}}$.
	\end{lemma}
	\begin{proof}
		This essentially follows from \cite[Corollary 3]{FraczykHarcosMaga}, as noted in the paragraph following its statement.  However, the gauge function used to define the successive minima in \cite{FraczykHarcosMaga} is inherited from the standard inner product on $k_\infty$ (essentially the $\ell^2$-norm), rather than the choice $\| \cdot \|$ above (essentially $\ell^\infty$).  Letting $\lambda_{i,2}(\mathfrak{C})$ denote the successive minima with respect to the gauge function of \cite{FraczykHarcosMaga}, one therefore finds $\lambda_i(\mathfrak{C}) \leq \lambda_{i,2}(\mathfrak{C}) \leq \sqrt{d} \lambda_i(\mathfrak{C})$ for each $1 \leq i \leq d$.  This, together with \cite[Corollary 3]{FraczykHarcosMaga}, implies the claim.
	\end{proof}
	
	\begin{remark}
		Though we have stated only an upper bound on $\lambda_d(\mathfrak{C})$ in Lemma \ref{lem:minkowski-ideal-asymptotic}, the work of Fraczyk--Harcos--Maga actually gives two-sided bounds on each successive minimum $\lambda_i(\mathfrak{C})$; see \cite[Corollary 4]{FraczykHarcosMaga}.  Additionally, Gergely Harcos has informed us that Lemma \ref{lem:wild-bound} below, together with the arguments of \cite{FraczykHarcosMaga}, implies the explicit bound $\lambda_{d,2}(\mathfrak{C}) \leq \pi^{-\frac{d}{2}} \Gamma(\frac{d}{2}+1)2^{d^2}|\mathfrak{C}|^{\frac{1}{d}} |\mathrm{Disc}(k)|^{\frac{1}{d}}$ (which therefore also holds for $\lambda_d(\mathfrak{C})$ by the proof of Lemma \ref{lem:minkowski-ideal-asymptotic}).
	\end{remark}
		
	 We will invoke Lemma \ref{lem:minkowski-ideal-asymptotic} in Sections \ref{subsec:power-sums} and \ref{subsec:skew-bounds} where we do not pursue bounds with an explicit dependence on $d$.

	Suppose now that $K/k$ is an extension of degree $n$.  In carrying out our arguments bounding relative extensions, it will be necessary to know not only how large $\lambda_{\max}(K)$ may be (for which Lemma \ref{lem:largest-minimum} provides a satisfactory answer), but also how small it may be, both relative to $\lambda_{\max}(k)$ and the discriminants of $K$ and $k$.  For this, we have the following two simple lemmas.
	
	\begin{lemma}\label{lem:lambda-extension}
		Let $K/k$ be an extension of number fields of degree $n$.  Then $\lambda_{\max}(k) \leq n \lambda_{\max}(K)$.
	\end{lemma}
	\begin{proof}
		Observe that $n\mathcal{O}_K$ is a sublattice of $\mathcal{O}_K^0 \oplus \mathcal{O}_k$, where $\mathcal{O}_K^0 := \{ \alpha \in \mathcal{O}_K : \mathrm{Tr}_{K/k} \alpha = 0\}$.  Hence, $\lambda_{\max}(\mathcal{O}_K^0 \oplus \mathcal{O}_k) \leq \lambda_{\max}(n\mathcal{O}_K) = n \lambda_{\max}(K)$.  On the other hand, letting $\mathrm{Proj}\colon \mathcal{O}_K^0 \oplus \mathcal{O}_k \to \mathcal{O}_k$ be the natural projection, then we also have that $\| \mathrm{Proj}(\alpha) \| \leq \|\alpha\|$ by our choice of gauge function.  From this, it follows that $\lambda_{\max}(\mathcal{O}_K^0 \oplus \mathcal{O}_k) \geq \lambda_{\max}(\mathcal{O}_k) = \lambda_{\max}(k)$.  (Observe that the values of the successive minima of $\mathcal{O}_k$ are independent of whether it is viewed as a lattice inside $k_\infty$ or $K_\infty$ by the choice of our gauge function $\| \cdot \|$.)  The lemma follows.
	\end{proof}
	
	\begin{remark}
		The stronger inequality $\lambda_{\max}(k) \leq \lambda_{\max}(K)$ does not hold in general.  For example, suppose that $k = \mathbb{Q}(\sqrt{D})$ for some positive squarefree $D \equiv 3 \pmod{4}$, and $K = k(i) = \mathbb{Q}(\sqrt{-D},i)$.  Then $\mathcal{O}_K$ contains the linearly independent elements $1$, $i$, $\frac{1+\sqrt{-D}}{2}$, and $\frac{i + \sqrt{D}}{2}$, so $\lambda_{\max}(K) \leq \sqrt{\frac{D+1}{4}}$.  On the other hand, it is easy to see that $\lambda_{\max}(k) = \sqrt{D}$, which shows that Lemma \ref{lem:lambda-extension} is asymptotically optimal if $n=2$.
	\end{remark}

	\begin{lemma}\label{lem:lambda-lower-bound}
		Let $k$ be a number field of degree $d$ and $K/k$ an extension of degree $n \geq 2$.  Then $\lambda_{\max}(K) \geq  \frac{2}{\pi} \frac{1}{(nd)^2}\left(\frac{|\mathrm{Disc}(K)|}{|\mathrm{Disc}(k)|}\right)^{\frac{1}{2d(n-1)}}$.
	\end{lemma}
	\begin{proof}
		Since $\mathcal{O}_k$ is a sublattice of $\mathcal{O}_K$, notice that $\lambda_i(K) \leq \lambda_i(k)$ for all $1 \leq i \leq d$.  Consequently, we find
			\[
				\lambda_1(K) \dots \lambda_{nd}(K) \leq \lambda_1(k) \dots \lambda_d(k) \lambda_{\max}(K)^{nd-d} \leq \lambda_{\max}(K)^{nd-d} \sqrt{|\mathrm{Disc}(k)|},
			\]
		upon appealing to the upper bound of \eqref{eqn:minkowski-second}.  Applying the lower bound of \eqref{eqn:minkowski-second} to the left-hand side above and simplifying, we thus find
			\[
				\lambda_{\max}(K)
					\geq \left(\frac{2^{nd}}{(2\pi)^{nd/2} (nd)!} \right)^{\frac{1}{d(n-1)}} \left(\frac{|\mathrm{Disc}(K)|}{|\mathrm{Disc}(k)|}\right)^{\frac{1}{2d(n-1)}}.
			\]
		The result then follows from the trivial upper bound $(nd)! \leq (nd)^{nd}$ and simplifying.
	\end{proof}
	
\subsubsection{Lattice point enumeration}
	The bulk of our bounds on number fields will rely on estimates for (in fact, only upper bounds on) the number of lattice points inside boxes of bounded height inside Minkowski space, i.e. those lattice points $\alpha$ satisfying $\| \alpha \| \leq H$ for some $H$.
	This follows from classical work of Blichfeldt \cite{Blichfeldt}, a proof of which may be found in \cite[Section 3]{FraczykHarcosMaga}.
	
	\begin{lemma}[Blichfeldt] \label{lem:lattice-points}
		Let $k$ be a number field of degree $d$ and let $0\ne \mathfrak{C}$ be an integral ideal of the ring of integers $\mathcal{O}_k$, viewed as a lattice inside $k_\infty$.  Then for any $H \geq \lambda_d(\mathfrak{C})$, we have
			\[
				\#\{ \alpha \in \mathfrak{C} : \| \alpha \| \leq H\}
					\leq 2^{r_1(k)} (2\pi)^{r_2(k)} (d+1)! \frac{H^d}{|\mathfrak{C}| \sqrt{|\mathrm{Disc}(k)|}},
			\]
		where $|\mathfrak{C}|:=|\mathcal{O}_k/\mathfrak{C}|$ is the norm of $\mathfrak{C}$.
	\end{lemma}
	\begin{proof}
		This follows from \cite[Theorem 2]{FraczykHarcosMaga} upon recalling the definition of $\lambda_d(\mathfrak{C})$.  Note that Blichfeldt in fact proves something stronger; see \cite[Theorem 7]{FraczykHarcosMaga}.
	\end{proof}
	
	We now specialize this to the two cases of primary interest to us, the first of which simply takes the ideal $\mathfrak{C}$ to be $\mathcal{O}_k$ itself.
	
	\begin{corollary} \label{cor:integers-bounded-height}
		Let $k$ be a number field of degree $d$, and let $\lambda_{\max}(k)$ denote the largest Minkowski successive minimum of $\mathcal{O}_k$.  Then for any $H \geq \lambda_{\max}(k)$, we have
			\[
				\# \{ \alpha \in \mathcal{O}_k : \|\alpha\| \leq H \}
					\leq (2\pi)^{d/2} (d+1)! \frac{H^d}{\sqrt{|\mathrm{Disc}(k)|}}.
			\]
	\end{corollary}
	\begin{proof}
		This follows from Lemma \ref{lem:lattice-points} and the simple observation $2^{r_1(k)} (2\pi)^{r_2(k)} \leq (2\pi)^{d/2}$.
	\end{proof}
	
	\begin{lemma}\label{lem:ideal-bounded-height}
		Let $k$ be a number field of degree $d$, and let $\mathfrak{C} \subseteq \mathcal{O}_k$ be an integral ideal.  Then for any $\beta \in \mathcal{O}_k$ and any $H \geq \frac{d}{2} \lambda_d(\mathfrak{C})$, there holds	
			\[
				\#\{ \alpha \in \mathcal{O}_k : \alpha \equiv \beta \pmod{\mathfrak{C}}, \|\alpha\| \leq H\}
					\leq (8\pi)^{d/2} (d+1)! \frac{H^d}{|\mathfrak{C}| \sqrt{|\mathrm{Disc}(k)|}},
			\]
		where $|\mathfrak{C}|:=|\mathcal{O}_k/\mathfrak{C}|$ is the norm of $\mathfrak{C}$.
	\end{lemma}
	\begin{proof}
		Let $v_1,\dots,v_d \in \mathfrak{C}$ be the linearly independent elements of $\mathfrak{C}$ satisfying $\|v_i\| = \lambda_i(\mathfrak{C})$.  Let $\Lambda$ be the sublattice of $\mathfrak{C}$ spanned by $v_1,\dots,v_d$.  Then the region $\Omega \subseteq k_\infty$ formed by vectors of the form $\epsilon_1 v_1 + \dots + \epsilon_d v_d$ with $-\frac{1}{2} < \epsilon_i \leq \frac{1}{2}$ is a fundamental domain for $k_\infty/\Lambda$; it is thus also $[\mathfrak{C}:\Lambda]$ copies of a fundamental domain for $k_\infty/\mathfrak{C}$.  Consequently, there is some $\beta^\prime \in \Omega$ such that $\beta^\prime \equiv \beta \pmod{\mathfrak{C}}$.
		
		Now, by the construction of $\Omega$, we have $\| \beta^\prime\| \leq \frac{d}{2} \lambda_d(\mathfrak{C})$, which in turn is bounded by $H$.  In particular, the shifted box $\{ \alpha - \beta^\prime : \|\alpha\| \leq H\}$ inside $k_\infty$ 
		is contained entirely within the (central) box $\{ \gamma \in k_\infty : \|\gamma\| \leq 2H\}$.  Thus, the number of $\alpha \in \mathcal{O}_k$ with $\|\alpha\| \leq H$ congruent to $\beta \pmod{\mathfrak{C}}$ is bounded above by the number of elements $\gamma \in \mathfrak{C}$ with $\|\gamma\| \leq 2H$.  The result then follows from Lemma \ref{lem:lattice-points}.
	\end{proof}

\subsubsection{The trace zero sublattice}
	
	Lastly, some of our results will make use of the properties of a distinguished sublattice of $\mathcal{O}_K$, namely the trace zero sublattice $\mathcal{O}_K^0$ that appeared in the proof of Lemma \ref{lem:lambda-extension}.  Explicitly, let $k$ be a number field of degree $d$, let $K/k$ be an extension of degree $n$, and let $\mathcal{O}_K^0 := \{ \alpha \in \mathcal{O}_K : \mathrm{Tr}_{K/k} \alpha = 0\}$.  So doing, $\mathcal{O}_K^0$ is a rank $d(n-1)$ lattice inside the Minkowski space $K_\infty$. Additionally, observe that there is a natural embedding $k_\infty \hookrightarrow K_\infty$, whose image we denote by $W$.  The trace zero sublattice $\mathcal{O}_K^0$ then is a full-rank sublattice of the orthogonal complement $W^\perp$.
	
	The Minkowski measure on $K_\infty$ gives rise to natural measures on $W$ and $W^\perp$ that differ from the standard Lebesgue measure by factors of $2^{r_2(k)+\frac{r_1^-(K/k)}{2}}$ and $2^{r_2(K)-r_2(k)-\frac{r_1^-(K/k)}{2}}$, respectively, where $r_1^-(K/k)$ denotes the number of real places of $k$ admitting a strictly complex extension to $K$.  Additionally, the measure on $W$ differs from that on $k_\infty$ by a factor of $2^{\frac{r_1^-(K/k)}{2}} n^{d/2}$.  From these considerations, we then find:
	
	\begin{lemma} \label{lem:trace-zero-covolume}
		Let $k$ be a number field of degree $d$, $K/k$ an extension of degree $n \geq 2$, and $\mathcal{O}_{K}^0$ the trace zero sublattice.  Let $W^\perp$ be as above, equipped with its natural measure.  Then 
			\[
				\mathrm{covol}_{W^\perp}(\mathcal{O}_K^0)
					= \frac{ [\mathcal{O}_K : \mathcal{O}_K^0 \oplus \mathcal{O}_k] }{2^{\frac{r_1^-(K/k)}{2}} n^{\frac{d}{2}} } \sqrt{\frac{|\mathrm{Disc}(K)|}{|\mathrm{Disc}(k)|}},
			\]
		so also $\mu_{\mathrm{Leb}}(W^\perp/\mathcal{O}_K^0) = 2^{r_2(k)-r_2(K)} n^{-\frac{d}{2}} [\mathcal{O}_K : \mathcal{O}_K^0 \oplus \mathcal{O}_k] \sqrt{\frac{|\mathrm{Disc}(K)|}{|\mathrm{Disc}(k)|}}$.
	\end{lemma}
	\begin{proof}
		Since $\mathcal{O}_K^0 \oplus \mathcal{O}_k$ is a sublattice of $\mathcal{O}_K$, we have on the one hand $\mathrm{covol}(\mathcal{O}_K^0 \oplus \mathcal{O}_k) = [O_K : \mathcal{O}_K^0 \oplus \mathcal{O}_k] \mathrm{covol}(\mathcal{O}_K) = [\mathcal{O}_K : \mathcal{O}_K^0 \oplus \mathcal{O}_k] \sqrt{|\mathrm{Disc}(K)|} $.  On the other, since $\mathcal{O}_K^0$ and $\mathcal{O}_k$ lie in the orthogonal spaces $W^\perp$ and $W$, we have $\mathrm{covol}(\mathcal{O}_K^0 \oplus \mathcal{O}_k) = \mathrm{covol}_{W^\perp}(\mathcal{O}_K^0) \mathrm{covol}_{W}(\mathcal{O}_k)$.  Since $\mathrm{covol}_W(\mathcal{O}_k) = 2^{\frac{r_1^-(K/k)}{2}} n^{\frac{d}{2}} \mathrm{covol}_{k_\infty}(\mathcal{O}_k) = 2^{\frac{r_1^-(K/k)}{2}} n^{\frac{d}{2}}\sqrt{|\mathrm{Disc}(k)|}$, the lemma follows.
	\end{proof}
	
	As a simple consequence of Lemma \ref{lem:trace-zero-covolume} and Minkowski's first theorem, we then find:
	
	\begin{lemma}\label{lem:trace-zero-element}
		Let $K/k$ be an extension of degree $n \geq 2$, where $k$ is a number field of degree $d$.  Then there exists a non-zero $\alpha \in \mathcal{O}_K$ with $\mathrm{Tr}_{K/k}\alpha = 0$ such that $\| \alpha\| \leq n^{\frac{1}{2(n-1)}} \left( \frac{|\mathrm{Disc}(K)|}{|\mathrm{Disc}(k)|} \right)^{\frac{1}{2d(n-1)}}$.
	\end{lemma}
	\begin{proof}
		Observe that $\mathcal{O}_K^0 \oplus \mathcal{O}_k$ is the kernel of the trace map $\mathcal{O}_K \to \mathcal{O}_k / n\mathcal{O}_k$.  Hence, the index $[\mathcal{O}_K : \mathcal{O}_K^0 \oplus \mathcal{O}_k]$ is at most $n^d$.  The claim then follows from Lemma \ref{lem:trace-zero-covolume} and Minkowski's first theorem.
	\end{proof}
	
	Lastly, if we let $\lambda_{\max}(\mathcal{O}_K^0) = \lambda_{d(n-1)}(\mathcal{O}_K^0)$ be the largest of the successive minima for $\mathcal{O}_K^0$, we find it convenient to note that for most $K$, $\lambda_{\max}(\mathcal{O}_K^0)$ cannot be smaller than $\lambda_{\max}(K)$.
	
	\begin{lemma}\label{lem:trace-zero-minimum}
		Let $K/k$ be an extension of number fields, and suppose that $\lambda_{\max}(K) > \lambda_{\max}(k)$.  Then $\lambda_{\max}(\mathcal{O}_K^0) \geq \lambda_{\max}(K)$.
	\end{lemma}
	\begin{proof}
		Since $\mathcal{O}_K^0 \oplus \mathcal{O}_k$ is a sublattice of $\mathcal{O}_K$, we have $\lambda_{\max}(K) \leq \lambda_{\max}(\mathcal{O}_K^0 \oplus \mathcal{O}_k)$.  By our choice of gauge function $\| \cdot \|$, we have $\lambda_{\max}(\mathcal{O}_K^0 \oplus \mathcal{O}_k) \leq \max\{ \lambda_{\max}(\mathcal{O}_K^0), \lambda_{\max}(k) \}$.  We claim this maximum must be $\lambda_{\max}(\mathcal{O}_K^0)$, for if not, we would conclude that $\lambda_{\max}(K) \leq \lambda_{\max}(k)$, contradicting our assumption that $\lambda_{\max}(K) > \lambda_{\max}(k)$.  We therefore have $\lambda_{\max}(K) \leq \lambda_{\max}(\mathcal{O}_K^0)$, as desired.
	\end{proof}

\subsection{A simple induction and a preliminary bound}  
	\label{subsec:induction}

	To motivate the kind of uniformity we will pursue in our bounds on $\mathcal{F}_{n,k}(X;G)$, we now turn to presenting the inductive argument we use to reduce the general problem to that of studying primitive extensions.  
	
	To this end, let $G$ be a transitive permutation group acting on a set $\Omega$ and let $\Omega =: \Delta_0,\Delta_1,\dots,\Delta_m = \{ \Omega\}$ be a sequence of systems of blocks that is nested in the sense that each block of $\Delta_i$ has the same size and is the union of at least two blocks from $\Delta_{i-1}$.  For each $i \geq 1$, if $B \in \Delta_i$ is a block, then by assumption $B = \cup_{j=1}^{n_i} B_j$ for some blocks $B_j \in \Delta_{i-1}$.  Since $G$ is transitive, the stabilizer of $B$ acts transitively on the set $\{B_1,\dots,B_{n_i}\}$, and we let $G_i$ denote the associated permutation group of degree $n_i$.  Up to isomorphism, the group $G_i$ is independent of the choice of block $B \in \Delta_i$, and we refer to the sequence $(G_1,\dots,G_m)$ as a \emph{tower type} for $G$.  Equivalently, $(G_1,\dots,G_m)$ is a tower type for $G$ if and only if for every $G$-extension $K/k$, there exist extensions $F_0/F_1/\dots/F_m$ with $F_0 =K$ and $F_m = k$ such that $F_{i-1}/F_i$ is a $G_i$-extension.  This notion corresponds with that of a primitive tower type from the introduction provided that each $G_i$ is a primitive permutation group.
	
	Our inductive argument is then the following.
	
	\begin{proposition}\label{prop:induction}
		Suppose $G$ is transitive of degree $n$, and let $(G_1,\dots,G_m)$ be a tower type for $G$.  Let $n_i$ be the degree of $G_i$, and suppose for each $i$ there exist constants $a_i,b_i>0$, and a function $C_i := C_i(d)$ such that for every number field $k$ of degree $d$ and every $X \geq 1$, we have
			\[
				\#\mathcal{F}_{n_i,k}(X;G_i)
					\leq C_i(d) X^{a_i} |\mathrm{Disc}(k)|^{-b_i}.
			\]
		Suppose also that $b_i \geq a_i+\frac{1}{2}$ for each $i \leq m-1$, and that $b_i \geq a_i + \frac{1}{2}+\frac{1}{2n_i}$ for each $i \leq m-1$ such that $a_i < a_{i+1} < a_i + \frac{1}{2n_i}$.  Then we have 
			\[
				\#\mathcal{F}_{n,k}(X;G) \leq 2^{m-1}\prod_{i=1}^{m-1} b_i \cdot \prod_{i=1}^m C_i(d_i) \cdot X^{\max\{a_1,\dots,a_m\}} |\mathrm{Disc}(k)|^{-b_m - \beta}
			\]
		for every number field $k$ of degree $d$ and every $X\geq 1$, where $d_i = d \prod_{j=i+1}^m n_j$ and
			\[
				\beta := \frac{1}{2} \sum_{i=0}^{m-2} \prod_{j=m-i}^{m} n_j = \frac{n_m + n_{m-1}n_m + \dots + n_2 \dots n_m}{2}.
			\]
	\end{proposition}
	\begin{proof}
		By induction, it suffices to consider the case $m=2$.  Since in this case $(G_1,G_2)$ is a tower type for $G$, we have
			\begin{align*}
				\#\mathcal{F}_{n,k}(X;G) 
					&\leq \sum_{F_1 \in \mathcal{F}_{n_2,k}(X^{1/n_1};G_2)} \#\mathcal{F}_{n_1,F_1}(X;G_1) \\
					&\leq C_1(dn_1) \sum_{F_1 \in \mathcal{F}_{n_2,k}(X^{1/n_1};G_2)} \frac{X^{a_1}}{|\mathrm{Disc}(F_1)|^{b_1}}.
			\end{align*}
		Suppose first that $a_2 \leq b_1 - \frac{1}{2}$.  Then the summation above converges as $X\to \infty$.  Noting that the $F_1$ appearing satisfy $|\mathrm{Disc}(F_1)| \geq |\mathrm{Disc}(k)|^{n_2}$, by partial summation, we find
			\[
				\sum_{F_1 \in \mathcal{F}_{n_2,k}(\infty;G_2)} \frac{1}{|\mathrm{Disc}(F_1)|^{b_1}}
					\leq \frac{b_1}{b_1-a_2} \frac{C_2(d)}{|\mathrm{Disc}(k)|^{b_2 + (b_1-a_2)n_2}}  \leq \frac{2 b_1 C_2(d)}{|\mathrm{Disc}(k)|^{b_2+\frac{n_2}{2}}},
			\]
		so that we find $\#\mathcal{F}_{n,k}(X;G) \leq 2b_1 \cdot C_1(dn_1) C_2(d) \cdot X^{a_1} |\mathrm{Disc}(k)|^{-b_2-\frac{n_2}{2}}$, which is sufficient.
		
		Suppose now that $a_2 > b_1 - \frac{1}{2}$, which implies that $a_2 > a_1$, and hence that $a_2 \geq a_1 + \frac{1}{2n_1}$.  (Our assumptions ensure that if $a_1 < a_2 < a_1 + \frac{1}{2n_1}$, then $a_2 \leq b_1 - \frac{1}{2}$.)  Set $\alpha_2 = n_1(a_2-a_1)+b_1$, and observe that $\alpha_2 \geq \max\{a_2,b_1\}$.  Using the (weaker) bound $\#\mathcal{F}_{n_2,k}(t;G_2) \leq C_2(d) t^{\alpha_2} |\mathrm{Disc}(k)|^{-b_2-n_2(\alpha_2-a_2)}$, valid for all $t \geq 1$ since the set is empty for $t < |\mathrm{Disc}(k)|^{n_2}$, we find
			\[
				\sum_{F_1 \in \mathcal{F}_{n_2,k}(X^{1/n_1};G_2)} \frac{X^{a_1}}{|\mathrm{Disc}(F_1)|^{b_1}}
					\leq \frac{b_1 C_2(d) }{\alpha_2-b_1}\frac{X^{a_1 + \frac{\alpha_2-b_1}{n_1}}}{|\mathrm{Disc}(k)|^{b_2+n_2(\alpha_2-a_2)}}  
					\leq 2b_1 C_2(d) \frac{X^{a_2}}{|\mathrm{Disc}(k)|^{b_2+n_2}},
			\]
		which again is sufficient, and completes the proof of the proposition.
	\end{proof}

	\begin{remark}
		As the proof shows, the exponent $\max\{a_1,\dots,a_m\}$ in the conclusion of Proposition \ref{prop:induction} is not optimal if $\max\{a_1,\dots,a_m\} \ne a_1$.  Similarly, the exponent of $|\mathrm{Disc}(k)|$ is also typically not optimal.  We prefer the clean statement for the presentation of many of our theorems, but of course improvements to those theorems follow by incorporating less simplified versions of the inductive argument.
	\end{remark}

	We now use Proposition \ref{prop:induction} and our treatment of the geometry of numbers to prove a version of Schmidt's theorem \cite{Schmidt} that obtains a greater savings in terms of the discriminant of $k$ than Schmidt's original argument and is completely explicit.  This will be useful to us in our handling of certain small degree extensions in our inductive argument, and also provides a baseline uniform bound.
	
	\begin{theorem}\label{thm:uniform-schmidt}
		Let $G$ be a transitive group of degree $n \geq 2$ and let $k$ be a number field of degree $d$.  Then
			\[
				\#\mathcal{F}_{n,k}(X;G)
					\leq (2 \pi)^{d(n-1)/2} (d+1)!^{n-1} n^{\frac{d(5n-2)}{4} +1} X^{\frac{n+2}{4}} |\mathrm{Disc}(k)|^{-3n/4}.
			\]
	\end{theorem}
	\begin{proof}
		Note that we may assume that $X \geq |\mathrm{Disc}(k)|^n$, since otherwise the set $\mathcal{F}_{n,k}(X;G)$ is empty and the conclusion is trivially true.  We suppose first that the group $G$ is primitive.  Let $K \in \mathcal{F}_{n,k}(X;G)$.  By Lemma \ref{lem:trace-zero-element}, there is a non-zero element $\alpha \in \mathcal{O}_K$ with $\mathrm{Tr}_{K/k}\alpha = 0$ such that $\| \alpha\| \leq H$, where $H = n^{\frac{1}{2(n-1)}} \left(\frac{X}{|\mathrm{Disc}(k)|}\right)^{\frac{1}{2d(n-1)}}$.  Since $G$ is primitive, we must have $K = k(\alpha)$, so the extension $K$ can be recovered by determining $\alpha$.  If we let $P_i := \mathrm{Tr}_{K/k}\alpha^i$ for each $2 \leq i \leq n$, then each $P_i$ is an element of $\mathcal{O}_k$ satisfying $\| P_i \| \leq n \|\alpha\|^i \leq n H^i$.  Moreover, since $\mathrm{Tr}_{K/k} \alpha = 0$, the values $P_2,\dots,P_n$ determine the minimal polynomial of $\alpha$ over $k$, and hence determine $\alpha$ and $K$, up to conjugation.  Thus, $\#\mathcal{F}_{n,k}(X;G)$ may be bounded by $n$ times the number of possible choices for $P_2,\dots,P_n$.
		
		Since $X \geq |\mathrm{Disc}(k)|^n$, we find $H^2 \geq n^{\frac{1}{n-1}} |\mathrm{Disc}(k)|^{\frac{1}{d}} > \lambda_{\max}(k)$ by Lemma \ref{lem:largest-minimum}.  Hence, we also have $H^i \geq \lambda_{\max}(k)$ for every $i \geq 2$.  By Corollary \ref{cor:integers-bounded-height}, it follows that the number of choices for $P_i$ is at most $(2\pi)^{d/2} n^d (d+1)! H^{id} |\mathrm{Disc}(k)|^{-\frac{1}{2}}$.  Thus, the number of choices for the $n-1$ different $P_2,\dots,P_n$ is at most
			\begin{align*}
				& (2\pi)^{d(n-1)/2} n^{d(n-1)} (d+1)!^{n-1} H^{d\frac{(n-1)(n+2)}{2}} |\mathrm{Disc}(k)|^{-\frac{n-1}{2}} \\
					& \quad\quad\quad = (2\pi)^{d(n-1)/2} (d+1)!^{n-1} n^{\frac{d(5n-2)}{4}} X^{\frac{n+2}{4}} |\mathrm{Disc}(k)|^{-3n/4}.
			\end{align*}
		Accounting for the $n$ conjugates $\alpha$ associated with fixed $P_2, \dots, P_n$, we obtain the claim if $G$ is primitive.
		
		We now consider the case that $G$ is imprimitive, with the goal of showing that the same bound holds.  We invoke Proposition \ref{prop:induction}, and, by way of induction and using the notation from Proposition \ref{prop:induction}, we may reduce to the case that $m=2$, $G_1$ is primitive, and where we assume for the group $G_2$ that
			\[
				\#\mathcal{F}_{n_2,k}(X;G_2)
					\leq (2 \pi)^{d(n_2-1)/2} (d+1)!^{n_2-1} n_2^{\frac{d(5n_2-2)}{4}+1} X^{\frac{n_2+2}{4}} |\mathrm{Disc}(k)|^{-3n_2/4}
			\]
		for every number field $k$ of degree $d$ and every $X \geq 1$.  Applying Proposition \ref{prop:induction}, we then find
			\begin{equation} \label{eqn:schmidt-induction}
				\#\mathcal{F}_{n,k}(X;G)
					\leq C \cdot X^{\frac{\max\{n_1,n_2\}+2}{4}} |\mathrm{Disc}(k)|^{-\frac{5n_2}{4}},
			\end{equation}
		where
			\[
				C := \frac{3}{2} (2\pi)^{d(n_1n_2-1)/2} (dn_2+1)!^{n_1-1} (d+1)!^{n_2-1} n_1^{\frac{dn_2(5n_1-2)}{4}+2} n_2^{\frac{d(5n_2-2)}{4}+1}.
			\]
		A straightforward calculation reveals that $C \leq (2\pi)^{d(n_1n_2-1)/2} (d+1)!^{n_1n_2-1} (n_1n_2)^{\frac{d(5n_1n_2-2)}{4} + 1}$ for any $d\geq 1$ and any $n_1,n_2 \geq 2$.  Hence, the claim holds for the ``implied constant'' in the estimate, and it remains to verify that it holds also for the powers of $X$ and $|\mathrm{Disc}(k)|$.  Returning to \eqref{eqn:schmidt-induction}, we may assume that $X \geq |\mathrm{Disc}(k)|^{n_1n_2}$, which implies that
			\[
				X^{\frac{\max\{n_1,n_2\}+2}{4}} |\mathrm{Disc}(k)|^{-\frac{5n_2}{4}}
					\leq X^{\frac{n_1n_2+2}{4}} |\mathrm{Disc}(k)|^{-\frac{5n_2}{4}-\frac{(n_1n_2)^2}{4}+\frac{n_1n_2\max\{n_1,n_2\}}{4}}.
			\]
		The inequality $\frac{5n_2}{4}+\frac{(n_1n_2)^2}{4}-\frac{n_1n_2\max\{n_1,n_2\}}{4} > \frac{3n_1n_2}{4}$ holds for every $n_1,n_2 \geq 2$, which completes the proof.
	\end{proof}
	
	\begin{corollary}\label{cor:schmidt-total}
		For any number field $k$, any $n \geq 2$, and any $X \geq 1$, let $\mathcal{F}_{n,k}(X) := \{ K/k : [K:k] = n, |\mathrm{Disc}(K)| \leq X\}$ and $\mathcal{F}_{n,k}^\mathrm{prim}(X) := \{ K \in \mathcal{F}_{n,k}(X) : K/k \text{ is primitive}\}$.  Then
			\[
				\#\mathcal{F}_{n,k}^{\mathrm{prim}}(X)
					\leq (2\pi)^{d(n-1)/2}(d+1)!^{n-1}n^{\frac{d(5n-2)}{4}+1} X^{\frac{n+2}{4}} |\mathrm{Disc}(k)|^{-3n/4}
			\]
		and
			\[
				\#\mathcal{F}_{n,k}(X)
					\leq (2\pi)^{d(n-1)/2}(d+1)!^{n-1}n^{n+\frac{d(5n-2)}{4}+1} X^{\frac{n+2}{4}} |\mathrm{Disc}(k)|^{-3n/4},
			\]
		where $d = [k:\mathbb{Q}]$.
	\end{corollary}
	\begin{proof}
		Although Theorem \ref{thm:uniform-schmidt} is stated for individual permutation groups $G$, the proof shows that the same bound holds for unions $\bigcup \mathcal{F}_{n,k}(X;G)$ over groups $G$ with primitive tower types of the same degree.  More specifically, for any $m \geq 1$ and any $n_1,\dots,n_m \geq 2$ with $n_1 \dots n_m = n$, the bound from Theorem \ref{thm:uniform-schmidt} holds for the union over groups with a primitive tower type of the form $(G_1,\dots,G_m)$, where each $G_i$ is primitive with degree $n_i$.  Taking $m=1$ yields the claim regarding $\mathcal{F}_{n,k}^\mathrm{prim}(X)$.  For the second claim, we note that this observation implies
			\[
				\#\mathcal{F}_{n,k}(X)
					\leq T(n) \cdot (2\pi)^{d(n-1)/2}(d+1)!^{n-1}n^{\frac{d(5n-2)}{4}} X^{\frac{n+2}{4}} |\mathrm{Disc}(k)|^{-3n/4},
			\]
		where $T(n)$ is the number of ways of writing $n$ as an ordered product of integers $\geq 2$.  We then use the trivial chain of inequalities $T(n) \leq \Omega(n)^{\Omega(n)} \leq n^n$ to conclude, where $\Omega(n)$ denotes the number of prime factors of $n$, counted with multiplicity.  (This bound on $T(n)$ is of course far from optimal, but since the shape of the constant is already qualitatively of the form $(dn)^{O(dn)}$, we don't expect it to substantially impact the utility of the result.)
	\end{proof}
	
	\begin{remark}
		For the sake of comparison, Schmidt \cite{Schmidt} proves only that $\#\mathcal{F}_{n,k}(X) \ll_{n,d} X^{\frac{n+2}{4}} |\mathrm{Disc}(k)|^{-\frac{n+2}{4}-\frac{1}{2d}}$.  He also uses an inductive argument, analogous to our Proposition \ref{prop:induction}, to obtain improved bounds for imprimitive extensions.  His version of this argument provides the same power of $X$ as Proposition \ref{prop:induction}, but with a slightly larger saving in $|\mathrm{Disc}(k)|$ in some particular cases, at least compared to the stated version of Proposition \ref{prop:induction}: Schmidt's bound is essentially of the form $\ll_{n,d} (X/|\mathrm{Disc}(k)|)^{\frac{\max\{n_1,\dots,n_m\}+2}{4}} |\mathrm{Disc}(k)|^{-\frac{1}{2d}}$, whereas ours is $\ll_{n,d} X^{\frac{\max\{n_1,\dots,n_m\}+2}{4}} |\mathrm{Disc}(k)|^{-\frac{3n_m}{4}-\beta}$.  However, as noted there, we proved Proposition \ref{prop:induction} in a form that is both convenient and sufficient for our purposes, but not in a way that provides optimal results.  For example, if $\max\{n_1,\dots,n_m\} = n_1$ (which is required for Schmidt's bound to be superior), then the proof of Proposition \ref{prop:induction} in fact easily yields a bound stronger than $\ll_{n,d} X^{\frac{n_1+2}{4}} |\mathrm{Disc}(k)|^{-\frac{n}{2}}$, which is superior to the bound provided by Schmidt in this case.  In particular, the arguments here fully subsume those of Schmidt, and are in addition completely explicit.
	\end{remark}

\section{Bounds on number fields via invariant theory}
	\label{sec:general-bounds}
	
	In this section, following ideas of Ellenberg--Venkatesh \cite{EV} and Lemke Oliver--Thorne \cite{LOThorne}, but incorporating some additional improvements (including some from \cite{Bhargava-vdW}), we prove many general results affording much stronger bounds than Theorem \ref{thm:uniform-schmidt}, provided one has non-trivial information about the invariant theory of the Galois group.  
	More specifically, suppose $G$ is a transitive permutation group of degree $n$.  
	Such groups $G$ act naturally on $\mathbb{A}^n$ by permutation of coordinates, and thus also on $(\mathbb{A}^n)^r$ for any $r \geq 1$ by a diagonal extension of this action.  
	The bounds we produce will depend on sets $\mathcal{I}$ of $nr$ algebraically independent invariants in this action on $(\mathbb{A}^n)^r$, which we will always assume to have coefficients in $\mathbb{Z}$.  
	We provide multiple bounds in this section; which is the strongest will depend on properties of the group $G$, the set $\mathcal{I}$, and the desired level of uniformity in the associated parameters.
	We highlight at the outset two results we anticipate will have the greatest utility:
		
	\begin{theorem}\label{thm:invariant-theory-soft-general}
		Let $G$ be a transitive permutation group of degree $n$ and let $\mathcal{I}=\{f_1,\dots,f_{nr}\}$ be a set of $nr$ algebraically independent invariants of $G$ in its action on $(\mathbb{A}^n)^r$ and with coefficients in $\mathbb{Z}$.  Then for any number field $k$ and any $X \geq 1$, there holds
			\[
				\#\mathcal{F}_{n,k}(X;G)
					\ll_{n,[k:\mathbb{Q}],r,\mathcal{I}} X^{\frac{1}{n} \deg \mathcal{I} - \frac{r}{2}} |\mathrm{Disc}(k)|^{-\frac{nr}{2}},
			\]
		where $\deg \mathcal{I} := \sum_{i=1}^{nr} \deg f_i$.
	\end{theorem}
	
	For groups $G$ with large index (e.g. those in Lemma \ref{lem:elemental-index}), we obtain a stronger bound.
	
	\begin{theorem}\label{thm:invariant-theory-soft-congruence}
		Let $G$ be a transitive permutation group of degree $n$ and let $\mathcal{I}=\{f_1,\dots,f_{nr}\}$ be a set of $nr$ algebraically independent invariants of $G$ in its action on $(\mathbb{A}^n)^r$ and with coefficients in $\mathbb{Z}$.  Let $\mathrm{ind}(G)$ denote the index of $G$.  Then for any number field $k$, any $X \geq 1$, and any $\epsilon>0$, we have
			\[
				\#\mathcal{F}_{n,k}(X;G)
					\ll_{n,[k:\mathbb{Q}],r,\mathcal{I},\epsilon}  
					X^{\frac{1}{n} \deg \mathcal{I} -\frac{3r}{2}+\frac{1}{\mathrm{ind}(G)} + \epsilon} |\mathrm{Disc}(k)|^{nr/2} \!\!\!\!\!\prod_{\substack{i \leq nr:\\ \deg f_i < \frac{n}{\mathrm{ind}(G)}}} \!\!\!\max\left\{\frac{X^{\frac{1}{\mathrm{ind}(G)} - \frac{\deg f_i}{n}}}{ |\mathrm{Disc}(k)|^{\frac{n}{\mathrm{ind}(G)}-\frac{3}{2}}},1\right\},
			\]
		where $\deg \mathcal{I} := \sum_{i=1}^{nr} \deg f_i$.
	\end{theorem}
	
	We prove Theorems \ref{thm:invariant-theory-soft-general} and \ref{thm:invariant-theory-soft-congruence} in slightly more refined forms; see Theorems \ref{thm:invariant-theory-multiplicity} and \ref{thm:invariant-theory-full} below.  Theorem \ref{thm:invariant-theory-soft-congruence} is typically stronger in $X$ than Theorem \ref{thm:invariant-theory-soft-general} by a factor roughly $O(X^{r-\frac{1}{\mathrm{ind}(G)}})$, but there are situations in which Theorem \ref{thm:invariant-theory-soft-general} yields a stronger result.  Notably, this is the case when $G$ is either the alternating or symmetric group of degree $n$.
	
	Additionally, as mentioned in the introduction and as with Theorem \ref{thm:uniform-schmidt}, the theorems leading to Theorems \ref{thm:invariant-theory-soft-general} and \ref{thm:invariant-theory-soft-congruence} (namely, Theorems \ref{thm:invariant-theory-multiplicity} and \ref{thm:invariant-theory-full}) give completely explicit bounds on $\#\mathcal{F}_{n,k}(X;G)$.  In particular the dependence of the implied constants on the parameters $[k:\mathbb{Q}]$, $n$, $r$, $\mathcal{I}$, and $\epsilon$ in Theorems \ref{thm:invariant-theory-soft-general} and \ref{thm:invariant-theory-soft-congruence} is spelled out therein.  Additionally, in Theorem \ref{thm:invariant-theory-power-sums}, we give a slight improvement to Theorem \ref{thm:invariant-theory-soft-congruence} that simplifies the product on the right-hand side.
	
	Before turning to the proofs of these results, we provide a simple example demonstrating the utility of these bounds over Theorem \ref{thm:uniform-schmidt}.
	
	\begin{example}
		For $q$ an odd prime power and $m \geq 2$ an integer, let $G = \mathrm{PGL}_m(\mathbb{F}_q)$ act in its primitive degree $n:=(q^m-1)/(q-1)$ representation on $\mathbb{P}^{m-1}(\mathbb{F}_q)$.  Then $\mathrm{ind}(G) = \frac{q^{m-1}-1}{2}$.  In the next section, we will show that there is a set of $n$ algebraically independent invariants for $G$ acting on $\mathbb{A}^n$ with degree bounded solely in terms of $m$; in fact, with $O_m(1)$ exceptions, these invariants all have degree $m+4$.  It then follows from Theorem \ref{thm:invariant-theory-soft-congruence} that, for any $k$, we have
			\[
				\#\mathcal{F}_{n,k}(X;G) \ll_{n,k} X^{m + \frac{5}{2} + \frac{c}{n}}
			\]
		for some constant $c>0$ depending only on $m$.  We make this explicit in Theorem \ref{thm:linear-bound-large} below.  By contrast, Theorem \ref{thm:uniform-schmidt} only implies that $\#\mathcal{F}_{n,k}(X;G) \ll_{n,k} X^{\frac{q^m+2q-3}{4(q-1)}}$, and \cite[Theorem 1]{LOThorne} only implies that $\#\mathcal{F}_{n,\mathbb{Q}}(X;G) \ll_n X^{O((m \log q)^2)}$.
	\end{example}

	Finally, we mention again the work of Dummit \cite{Dummit} on bounding $G$-extensions of $\mathbb{Q}$ using invariant theory.  An important distinction between our work and Dummit's is that Dummit requires the set of invariants to be a set of primary invariants, whereas we only require our invariants to be algebraically independent.  While there are some advantages to working with primary invariants, as we shall see, Theorem \ref{thm:invariant-theory-soft-general} provides both a substantial theoretical improvement (in that the degrees of algebraically independent invariants may, in general, be much smaller than those of primary invariants) and a substantial computational improvement (as the problem of computing a set of primary invariants for groups even of moderate size is frequently intractable).  In particular, it is not apparent to us how to use Dummit's work to obtain a meaningful bound on $\#\mathcal{F}_{n,k}(X;G)$ for groups like $\mathrm{PGL}_m(\mathbb{F}_q)$.

\subsection{An initial bound via invariant theory}
	\label{subsec:initial-invariant-bound}
	
	In this subsection, we prove something weaker than Theorems \ref{thm:invariant-theory-soft-general} and \ref{thm:invariant-theory-soft-congruence} that makes clear our general approach; the specific improvements leading to Theorems \ref{thm:invariant-theory-soft-general} and and \ref{thm:invariant-theory-soft-congruence} will be carried out in the next two subsections.  
	
	As in the statements of these theorems, we are concerned with sets $\mathcal{I}=\{f_1,\dots,f_{nr}\}$ of $nr$ algebraically independent invariants of $G$ in this action.  The key characterization of algebraic independence we shall use (sometimes implicitly) is the following classical criterion.
	
	\begin{lemma}\label{lem:independence-criterion}
		With notation as above, the set $\mathcal{I}$ is algebraically independent if and only if the determinant of the $nr\times nr$ matrix $\mathbf{D}(\mathcal{I}) = \{ \frac{\partial f_i}{\partial x_j}\}_{1 \leq i,j \leq nr}$ is a non-zero polynomial.
	\end{lemma}
	\begin{proof}
		This is known as the Jacobian criterion for algebraic independence, a proof of which may be found in \cite[Theorem 2.2]{EhrenborgRota}, for example.
	\end{proof}
	
	We next have the following simple proposition that essentially records \cite[Lemma 2.4]{EV} as it is applied there and in \cite{LOThorne}.  
	
	\begin{proposition}\label{prop:integers-avoid-hypersurface}
		Let $k$ be a number field of degree $d$ and let $K/k$ be an extension of degree $n$.  Let $v\colon k \to \mathbb{C}$ be a fixed embedding (either real or complex).  Let $\iota \colon K^r \to (\mathbb{C}^n)^r$ be formed by the product of the $n$ embeddings of $K$ extending $v$.  Finally, let $F\colon \mathbb{C}^{nr} \to \mathbb{C}$ be any non-zero polynomial.
		
		Then there is an element $\alpha \in \mathcal{O}_K^r$ with $\|\alpha\| \leq \frac{dn(\deg F + 1)}{2} \lambda_{\max}(K)$ for which $F(\iota(\alpha)) \ne 0$, where $\| \cdot \|$ denotes the natural extension of the gauge function $\| \cdot \|$ on $\mathcal{O}_K$ to $\mathcal{O}_K^r$.
	\end{proposition}
	\begin{proof}
		Let $\lambda_1,\dots,\lambda_{dn}$ be the successive minima of $\mathcal{O}_K$, and let $v_1,\dots,v_{dn}$ be the linearly independent elements of $\mathcal{O}_K$ with $\|v_i\| = \lambda_i$.  Then there is a map $\mathbb{Q}^{dnr} \to K^r$ given by choosing coordinates in terms of the $v_i$.  Composing this map with $\iota$ and $F$, we obtain a non-zero polynomial $\mathbb{Q}^{dnr} \to \mathbb{C}$ with complex coefficients, which in particular we may also regard as a polynomial $\mathbb{C}^{dnr} \to \mathbb{C}$.  The result then follows from \cite[Lemma 2.4]{EV}.
	\end{proof}
	
	With these preliminaries handled, we are now able to present the core result on which the next subsections shall improve.  This is essentially a more general version of the argument used in \cite{LOThorne}.  In stating this result uniformly, we find it useful to define the following quantities.
	
	\begin{notation}
		Suppose that $G$ is a transitive permutation group of degree $n$, and $\mathcal{I} = \{ f_1,\dots,f_{nr} \}$ is a set of $nr$ distinct $G$-invariant polynomials in its action on $(\mathbb{A}^n)^r$.  Then:
			\begin{itemize}
				\item for a polynomial $f_i \in \mathcal{I}$, we let $|f_i|$ denote the polynomial obtained from $f_i$ by replacing all non-zero coefficients with their absolute values;
				\item we let $\mathcal{I}(1) := \prod_{i=1}^{nr} |f_i|(1)$; and
				\item we let $\deg \mathcal{I} = \sum_{i=1}^{nr} \deg f_i $.
			\end{itemize}
	\end{notation}
	\noindent (To the extent possible, we recap this notation in the statement of our theorems.)
	
	\begin{theorem}\label{thm:invariant-theory-simple}
		Let $G$ be a transitive permutation group of degree $n$ and let $\mathcal{I}=\{f_1,\dots,f_{nr}\}$ be a set of $nr$ algebraically independent invariants of $G$ in its action on $(\mathbb{A}^n)^r$.  Then for any number field $k$ of degree $d$ and any $X \geq 1$,
			\[
				\#\mathcal{F}_{n,k}(X;G)
					\leq (2\pi)^{dnr/2} (d+1)!^{nr} \mathcal{I}(1)^d \left(dn\left(\deg \mathcal{I} + \binom{n}{2}\right)\right)^{d\cdot \deg \mathcal{I}}  \cdot X^{\frac{1}{n} \deg \mathcal{I}} / |\mathrm{Disc}(k)|^{nr/2},
			\]
		where $\deg \mathcal{I} = \sum_{i=1}^{nr} \deg f_i$ and $\mathcal{I}(1) = \prod_{i=1}^{nr} |f_i|(1)$.
	\end{theorem}
	\begin{proof}
		We first consider the case that the basefield $k$ is equal to $\mathbb{Q}$.  Let $\mathbf{D}(\mathcal{I})$ be the matrix of partial derivatives associated with $\mathcal{I}$, and let $D\colon \mathbb{A}^{nr} \to \mathbb{A}^1$ be its determinant, which is non-zero by Lemma \ref{lem:independence-criterion}.  Let $\mathrm{Disc}_1 \colon (\mathbb{A}^n)^r \to \mathbb{A}^1$ be the squareroot of the discriminant in the first copy of $\mathbb{A}^n$, i.e. 
			\[
				\mathrm{Disc}_1	:= \prod_{i < j \leq n} (x_i-x_j).
			\]
		Hence, $\deg( D(\mathbf{x}) \mathrm{Disc}_1(\mathbf{x})) = \deg \mathcal{I} + \binom{n}{2} - nr \leq \deg \mathcal{I} + \binom{n}{2} -1 $.  
		
		Now, let $K \in \mathcal{F}_{n,\mathbb{Q}}(X;G)$.    Combining Lemma \ref{lem:largest-minimum} and Proposition \ref{prop:integers-avoid-hypersurface}, there is an element $\alpha \in \mathcal{O}_K^r$ with 
		$\|\alpha\| \leq \frac{n}{2} (\deg \mathcal{I} + \binom{n}{2}) X^{\frac{1}{n}}$ 
		such that $D(\iota(\alpha)) \mathrm{Disc}_1(\iota(\alpha)) \ne 0$.  Since $\mathrm{Disc}_1(\iota(\alpha)) \ne 0$, it follows that $K = \mathbb{Q}(\alpha)$.  On the other hand, since $D(\iota(\alpha)) \ne 0$, it follows that $\iota(\alpha)$ corresponds to a $0$-dimensional affine component of the intersection of the $nr$ hypersurfaces $f_i(\mathbf{x}) = f_i(\iota(\alpha))$ for $i \leq nr$.  By B\'ezout's theorem, there are a most $\prod_{i=1}^{nr} \deg f_i$ such affine components, and thus also at most that many fields $K^\prime \in \mathcal{F}_{n,\mathbb{Q}}(X;G)$ possessing an element $\alpha^\prime \in \mathcal{O}_{K^\prime}^r$ for which $D(\iota(\alpha^\prime))\neq 0$ and $f_i(\alpha^\prime) = f_i(\alpha)$ for all $i$.  Since $\|\alpha\| \leq \frac{n}{2}(\deg \mathcal{I} + \binom{n}{2}) X^{\frac{1}{n}}$ and each polynomial $f_i$ is $G$-invariant, it follows that each $f_i(\iota(\alpha))$ is an integer of absolute value at most $|f_i|(1) \cdot (\frac{n}{2}(\deg \mathcal{I} + \binom{n}{2}))^{\deg f_i} X^\frac{\deg f_i}{n}$.  There are therefore at most
			\begin{align*}
				\prod_{i=1}^{nr}\left[ 2|f_i|(1) \cdot \left(\frac{n}{2}\left(\deg \mathcal{I} + \binom{n}{2}\right)\right)^{\deg f_i} X^\frac{\deg f_i}{n}+1\right] \\
					\leq 3^{nr} \mathcal{I}(1) \left(\frac{n}{2}\left(\deg \mathcal{I} + \binom{n}{2}\right)\right)^{\deg \mathcal{I}}  \cdot X^{\frac{1}{n} \deg \mathcal{I}}
			\end{align*}
		collective choices for these $nr$ integers.  Accounting for the at most $\prod_{i=1}^{nr} \deg f_i \leq 2^{\deg \mathcal{I} - nr}$ different $K \in \mathcal{F}_{n,\mathbb{Q}}(X;G)$ with the same set of invariants yields a stronger statement than claimed in the case $k= \mathbb{Q}$.
		
		For the general case, where $k$ is a number field of degree $d \geq 2$, we again apply Lemma \ref{lem:largest-minimum} and Proposition \ref{prop:integers-avoid-hypersurface} to find an $\alpha \in \mathcal{O}_K^r$ for which $\|\alpha\| \leq \frac{dn}{2}(\deg\mathcal{I} + \binom{n}{2}) X^{\frac{1}{dn}}$ and for which $D(\iota(\alpha)) \mathrm{Disc}_1(\iota(\alpha)) \ne 0$.  As with the case $k=\mathbb{Q}$, the extension $K/k$ is then determined, up to at most $\prod_{i=1}^{nr} \deg f_i$ choices, by the values $f_i(\iota(\alpha))$.  These are integers in $k$ with height bounded by $|f_i|(1)\left(\frac{dn}{2}(\deg\mathcal{I} + \binom{n}{2})\right)^{\deg f_i} X^{\frac{\deg f_i}{dn}}$.  We next note that, for the set $\mathcal{F}_{n,k}(X;G)$ to be non-empty, we may assume that $X \geq |\mathrm{Disc}(k)|^n$, since we have defined the set $\mathcal{F}_{n,k}(X;G)$ with respect to the absolute discriminant of $K$.  Since $\deg f_i \geq 1$ for each $i$, this bound on the height of $f_i(\iota(\alpha))$ is greater than $n\cdot|\mathrm{Disc}(k)|^{\frac{1}{d}} \geq \lambda_{\max}(k)$, the latter inequality following from Lemma \ref{lem:largest-minimum}.  Thus, by Corollary \ref{cor:integers-bounded-height}, the number of choices for each $f_i(\iota(\alpha)) \in \mathcal{O}_k$ is at most $|f_i|(1)^d(2\pi)^{d/2} (d+1)! \left(\frac{dn}{2}(\deg\mathcal{I} + \binom{n}{2})\right)^{d\cdot \deg f_i} X^{\frac{\deg f_i}{n}} / \sqrt{|\mathrm{Disc}(k)|}$.  The rest of the claim then follows as in the case that $k=\mathbb{Q}$.
	\end{proof}
	
	To go further, we incorporate two improvements into this argument.  The first is easier, and simply refines Proposition \ref{prop:integers-avoid-hypersurface} to account for the number of elements of $\mathcal{O}_K^r$ for which the determinant does not vanish.  The second is somewhat subtler, and relies on showing that the specialized invariants $f_i(\iota(\alpha)) \in \mathcal{O}_k$ are subject to fairly restrictive congruence conditions at the primes ramified in $K/k$.  Additionally, in both of these arguments, we keep track of the dependence on $\lambda_{\max}(K)$, as in small degrees it is possible to use an alternate argument when $\lambda_{\max}(K)$ is large to obtain a further slight improvement; see Section \ref{subsec:skew-bounds}.

\subsection{Accounting for multiplicity: Integral points avoiding hypersurfaces}
	\label{subsec:invariant-bound-multiplicity}
		
	We begin with a simple lemma.	
		
	\begin{lemma}\label{lem:hypersurface-points}
		Let $F\colon \mathbb{C}^N \to \mathbb{C}$ be any non-zero polynomial.  Then for any $N$ sets of integers $B_1,\dots,B_N$ with each $|B_i| \geq 2 \deg F$, there holds
		\[
			\#\{(x_1,\dots,x_N) \in B_1 \times \dots B_N : F(x_1,\dots,x_N) \ne 0\}
				\geq \frac{1}{2^N} \prod_{i=1}^N |B_i|.
		\]
	\end{lemma}
	\begin{proof}
		By induction.  If $N=1$, then the statement is clear, as any single-variable polynomial has at most degree-many zeros.  If $N > 1$, then similarly the assumption that $|B_N| \geq 2 \deg F$ implies that there are at least $|B_N|/2$ values of $x_N \in B_N$ for which the resulting specialized polynomial $\mathbb{C}^{N-1} \to \mathbb{C}$ is not identically $0$.  The claim then follows from the inductive hypothesis.
	\end{proof}
	
	This then yields the following easy refinement to Proposition \ref{prop:integers-avoid-hypersurface}.
	
	\begin{corollary}\label{cor:hypersurface-integers}
		Let $k$ be a number field of degree $d$ and let $K/k$ be an extension of degree $n$.  Let $v\colon k \to \mathbb{C}$ be an embedding of $k$ (either real or complex), and for any integer $r\geq 1$, let $\iota\colon K^r \to (\mathbb{C}^n)^r$ be given by the product of the $n$ embeddings of $K$ extending $v$.
		
		Then for any non-zero polynomial $F\colon \mathbb{C}^{nr} \to \mathbb{C}$ and any $\lambda \geq \lambda_{\max}(K)$, there holds
			\[
				\#\{ \alpha \in \mathcal{O}_K^r : \| \alpha \| \leq dn\deg F \cdot \lambda, F(\iota(\alpha)) \ne 0\}
					\geq \left(\frac{ \deg F}{2} \right)^{dnr} \frac{\lambda^{dnr}}{ |\mathrm{Disc}(K)|^{r/2}}.
			\]
	\end{corollary}
	\begin{proof}
		We proceed much as in the proof of Proposition \ref{prop:integers-avoid-hypersurface}.
		As there, let $\lambda_1,\dots,\lambda_{dn}$ be the successive minima of $\mathcal{O}_K$, and let $v_1,\dots,v_{dn}$ be the linearly independent elements of $\mathcal{O}_K$ satisfying $\| v_i\| = \lambda_i$.  Let $\phi\colon \mathbb{Q}^{dnr} \to K^r$ be the natural map given by choosing coordinates in terms of the $v_i$, and we let $\tilde{F}\colon \mathbb{Q}^{dnr} \to \mathbb{C}$ be the composite $F \circ \iota \circ \phi$.  Since $F$ is a non-zero polynomial over $\mathbb{C}$, it follows that we may regard $\tilde{F}$ as a polynomial $\mathbb{C}^{dnr}\to\mathbb{C}$ as well.
		
		Now, for each $1 \leq i \leq dn$, let $B_i = \{ x \in \mathbb{Z} : |x| \leq \frac{\lambda \deg F}{\lambda_i}\}$.  For any $\mathbf{x} \in (B_1 \times \dots \times B_{nd})^r$, we have by construction that $\phi(\mathbf{x}) \in \mathcal{O}_K^r$ and $\|\phi(\mathbf{x})\| \leq dn \deg F \cdot \lambda$.  On the other hand, by Lemma \ref{lem:hypersurface-points}, the number of such $\mathbf{x}$ for which $\tilde{F}(\mathbf{x}) \ne 0$ is 
			\[
				\geq \frac{1}{2^{dnr}} \prod_{i=1}^{dn} |B_i|^r \geq \frac{1}{2^{dnr}} \prod_{i=1}^{dn} \left[ (2 \deg F -1) \frac{\lambda}{\lambda_i}\right] \geq \left(\frac{\deg F}{2}\right)^{dnr} \frac{\lambda^{dnr}}{\mathrm{Disc}(K)^{r/2}}.
			\]
		Combining these two considerations yields the claim.
	\end{proof}
	
	We now present the first improvement to Theorem \ref{thm:invariant-theory-simple}.
	
	\begin{theorem} \label{thm:invariant-theory-multiplicity}
		Let $G$ be a transitive permutation group of degree $n$ and let $\mathcal{I}=\{f_1,\dots,f_{nr}\}$ be a set of $nr$ algebraically independent invariants of $G$ in its action on $(\mathbb{A}^n)^r$.  Then for any number field $k$ of degree $d$, and any $X,\lambda \geq 1$, we have
			\begin{align*}
				\#\{K \in \mathcal{F}_{n,k}(X;G) &: \lambda_{\max}(K) \leq \lambda\} \\
					&\leq \left(2\pi\right)^{dnr/2} (d+1)!^{nr} \mathcal{I}(1)^d\left(2dn\cdot \left(\deg \mathcal{I} + \binom{n}{2}\right)\right)^{d\cdot \deg \mathcal{I}} \frac{\lambda^{d\cdot\deg \mathcal{I} - dnr} X^{\frac{r}{2}}}{|\mathrm{Disc}(k)|^{nr/2}}.
			\end{align*}
		In particular, 
			\[
				\#\mathcal{F}_{n,k}(X;G) 
					\leq \left(2\pi\right)^{dnr/2} (d+1)!^{nr} \mathcal{I}(1)^d \left(2dn\cdot \left(\deg \mathcal{I} + \binom{n}{2}\right)\right)^{d\cdot \deg \mathcal{I}} \frac{X^{\frac{1}{n}\deg \mathcal{I} - \frac{r}{2} }}{ |\mathrm{Disc}(k)|^{rn/2}}.
			\]
	\end{theorem}
	\begin{proof}
		We follow the proof of Theorem \ref{thm:invariant-theory-simple}, noting only the necessary changes.
		Let $D$ and $\mathrm{Disc}_1$ be as in Theorem \ref{thm:invariant-theory-simple}, and let $F\colon \mathbb{C}^{nr} \to \mathbb{C}$ be their product.
		We work dyadically, i.e. with fields $K$ whose discriminant satisfies $X/2 \leq |\mathrm{Disc}(K)| \leq X$.  
		By Corollary \ref{cor:hypersurface-integers}, for each such field, there are at least $\left(\frac{\deg F}{2}\right)^{dnr} \lambda^{dnr} / X^{r/2} \geq \lambda^{dnr}/X^{r/2}$ elements $\alpha \in \mathcal{O}_K^r$ for which $\|\alpha\| \leq dn\deg F \cdot \lambda$ and $F(\iota(\alpha)) = D(\iota(\alpha)) \mathrm{Disc}_1(\iota(\alpha)) \ne 0$.  
		
		Each invariant $f_i(\iota(\alpha)) \in \mathcal{O}_k$ has height bounded by $|f_i|(1) \cdot (dn \deg F)^{\deg f_i} \lambda^{\deg f_i}$.  By the assumption that $\lambda_{\max}(K) \leq \lambda$ and by Lemma \ref{lem:lambda-extension}, we conclude that this height bound is larger than $\lambda_{\max}(k)$.
		Consequently, by Corollary \ref{cor:integers-bounded-height}, the number of choices in $\mathcal{O}_k$ for $f_i(\iota(\alpha))$ is at most
			\[
				(2\pi)^{d/2} (d+1)! |f_i|(1)^d (dn \cdot \deg F)^{d\cdot \deg f_i} \frac{\lambda^{d \cdot \deg f_i}}{|\mathrm{Disc}(k)|^{1/2}},
			\]
		and thus there are at most 
			\[
				(2\pi)^{dnr/2} (d+1)!^{nr} \mathcal{I}(1)^{d} (dn\cdot\deg F)^{d \cdot \deg \mathcal{I}} \frac{\lambda^{d\cdot\deg \mathcal{I}}}{|\mathrm{Disc}(k)|^{nr/2}}
			\]
		choices for all the $f_i(\iota(\alpha))$'s.  As in the proof of Theorem \ref{thm:invariant-theory-simple}, there are fewer than $\prod_{i=1}^{nr} \deg f_i \leq 2^{\deg \mathcal{I} - nr}$ possible extensions $K/k$ with a given set of invariants.  Accounting for the $\geq \lambda^{dnr} / X^{r/2}$ different $\alpha$ inside a given field, we find that the number of desired $K$ with $X/2 < |\mathrm{Disc}(K)| \leq X$ is at most
			\[
				2^{-nr} (2\pi)^{dnr/2} (d+1)!^{nr} \mathcal{I}(1)^{d} (2dn\cdot\deg F)^{d \cdot \deg \mathcal{I}} \frac{\lambda^{d\cdot\deg \mathcal{I} - dnr}X^{r/2}}{|\mathrm{Disc}(k)|^{nr/2}}.
			\]
		After accounting for the dyadic summation and exploiting the fact that $2^{-nr} \leq 1/4$, the first claim follows.  The second follows upon taking $\lambda = X^{1/dn}$ and invoking Lemma \ref{lem:largest-minimum}.	
	\end{proof}
	
	We now turn to the second improvement over Theorem \ref{thm:invariant-theory-simple}.

\subsection{Congruence restrictions at ramified primes}
	\label{subsec:invariant-bound-congruence}
	
	In the proofs of Theorems \ref{thm:invariant-theory-simple} and \ref{thm:invariant-theory-multiplicity}, we used that the values of the invariants are integers of bounded height.  In this section, we show that these values are subject to sometimes substantial congruence restrictions at ramified primes.  This will result in an improvement in the exponent of $X$ in the resulting bound on $G$-extensions at the expense of the dependence on the other parameters.  Thus, while we anticipate the results of this section will be better for many applications, there are instances where the results of the previous will be better.
	
	We recall a few standard notions.  Let $k$ be a number field and $\mathfrak{p}$ a prime ideal of $k$.  If $K/k$ is a $G$-extension with normal closure $\widetilde{K}/k$ and $\mathfrak{P}$ is a prime of $\widetilde{K}$ lying over $\mathfrak{p}$, we let $D_\mathfrak{P}$ and $I_\mathfrak{P}$ be the associated decomposition and inertia subgroups.  For each $i \geq 0$, we also let $G_i=G_{\mathfrak{P},i}$ be the associated (lower) ramification groups, i.e. $G_i = \{ \sigma \in D_\mathfrak{P} : \sigma(x) \equiv x \pmod{\mathfrak{P}^{i+1}} \,\forall x \in \mathcal{O}_{\widetilde{K}} \}$.  In particular, $G_0 = I_\mathfrak{P}$, which notation we will use interchangeably.  If $\mathfrak{D}_{K/k}$ is the relative discriminant ideal of $K/k$, then by computing the Artin conductor as in \cite[\S VII.11]{Neukirch}, we find
		\begin{equation} \label{eqn:discriminant}
			\mathfrak{D}_{K/k} 
				= \prod_{\mathfrak{p}} \mathfrak{p}^{\sum_{i=0}^\infty \frac{|G_i|}{|G_0|} (n - \#\mathrm{Orb}(G_i))},
		\end{equation}
	where for any $H \subseteq G$, $\mathrm{Orb}(H)$ is the set of orbits under $H$ in the permutation action of $G$.  If we let $p := \mathrm{char}(\mathcal{O}_k/\mathfrak{p})$ be the rational prime lying below $\mathfrak{p}$, then as is well known, the groups $G_i$ are $p$-groups for each $i \geq 1$, and the quotient $G_0 / G_1$ is cyclic of order prime to $p$.  In particular, the groups $G_i$ for $i \geq 1$ may be non-trivial only if $p \mid |G|$, and we define
		\begin{equation}\label{eqn:tame-wild}
			\mathfrak{D}_{K/k}^\mathrm{tame} = \prod_\mathfrak{p} \mathfrak{p}^{n-\#\mathrm{Orb}(G_0)} 
				, \quad
			\mathfrak{D}_{K/k}^\mathrm{wild} = \prod_{\mathfrak{p}} \mathfrak{p}^{\sum_{i=1}^\infty \frac{|G_i|}{|G_0|}(n-\#\mathrm{Orb}(G_i))},
		\end{equation}
	to be the tame and wild parts of the discriminant.  
	
	Our first lemma shows that the values of $G$-invariants evaluated at integral elements of $K$ are naturally subject to congruence conditions with relative density $|\mathfrak{D}_{K/k}^\mathrm{tame}|^{-1}$, where $| \cdot |$ denotes the ideal norm.
	
	\begin{lemma} \label{lem:invariant-congruences-primes}
		Let $G$ be a transitive permutation group of degree $n$ and let $\mathcal{I}=\{f_1,\dots,f_{nr}\}$ be any set of $G$-invariants in its action on $(\mathbb{A}^n)^r$ with coefficients in $\mathbb{Z}$.  Let $k$ be a number field and let $\mathfrak{p}$ be a prime ideal of $k$.  Let $K/k$ be a $G$-extension, let $\mathfrak{P}$ be a prime of the normal closure $\widetilde{K}/k$ lying over $\mathfrak{p}$, and let $I_\mathfrak{P}$ be the associated inertia subgroup of $G$.  Then there is a set $\mathcal{R}_\mathfrak{p}(\mathcal{I}) \subseteq (\mathcal{O}_k/\mathfrak{p})^{nr}$ of size at most $|\mathfrak{p}|^{r\#\mathrm{Orb}(I_\mathfrak{P})}$ and depending only on $\mathcal{I}$, $\mathfrak{p}$, and the $G$-conjugacy class of $I_\mathfrak{P}$, such that for any $\alpha \in \mathcal{O}_K^r$, the tuple $( f_1(\alpha), \dots, f_{nr}(\alpha) ) \pmod{\mathfrak{p}} $ lies in $\mathcal{R}_\mathfrak{p}(\mathcal{I})$.
	\end{lemma}
	\begin{proof}
		We prove the claim for $r=1$, the case $r\geq 2$ being materially the same.  Given the setup of the lemma, let $\iota_1,\dots,\iota_n$ be the $n$ embeddings of $K$ into $\widetilde{K}$ fixing $k$, and for $\alpha \in \mathcal{O}_K$, let $\iota_\mathfrak{P}(\alpha) := (\iota_1(\alpha),\dots,\iota_n(\alpha)) \pmod{\mathfrak{P}}$.  We claim that there are at most $|\mathfrak{p}|^{\#\mathrm{Orb}(I_\mathfrak{P})}$ possible values of $\iota_\mathfrak{P}(\alpha)$ that occur.  If so, since for any $G$-invariant with integral coefficients the value of $\iota_\mathfrak{P}(\alpha)$ determines the value of $f(\alpha) \pmod{\mathfrak{p}}$, then any set of integral $G$-invariants (regardless of size) must lie in at most $|\mathfrak{p}|^{\#\mathrm{Orb}(I_\mathfrak{P})}$ classes $\pmod{\mathfrak{p}}$.
		
		To see the claim, by the definition of the inertia subgroup, for any $\alpha \in \mathcal{O}_K$, the $n$-tuple $\iota_\mathfrak{P}(\alpha)$ must be invariant under the permutation action of $I_\mathfrak{P}$.  In particular, there may be at most $\#\mathrm{Orb}(I_\mathfrak{P})$ distinct entries in $\iota_\mathfrak{P}(\alpha)$.  Let $F_\mathfrak{p} = \mathcal{O}_k/\mathfrak{p}$, $F_\mathfrak{P} = \mathcal{O}_{\widetilde{K}}/\mathfrak{P}$, and $f_\mathfrak{p} := [F_\mathfrak{P} : F_\mathfrak{p}]$, and recall that $D_\mathfrak{P}/I_\mathfrak{P} \simeq \mathrm{Gal}(F_\mathfrak{P} / F_\mathfrak{p})$.  In particular, for any $\sigma \in D_\mathfrak{P}$, there is some $\phi \in \mathrm{Gal}(F_\mathfrak{P} / F_\mathfrak{p})$ such that $\iota_\mathfrak{P}(\alpha)^\sigma = \iota_\mathfrak{P}(\alpha)^\phi$ (where $\sigma$ acts via permutation and $\phi$ acts via Galois).  Let $\Omega$ be a $D_\mathfrak{P}/I_\mathfrak{P}$ orbit on the set of $I_\mathfrak{P}$-orbits $\iota_\mathfrak{P}(\alpha)/I_\mathfrak{P}$.  If $\Omega$ has size $d$, say, then we must have $d \mid f_\mathfrak{p}$ and that the entries in $\Omega$ are conjugate elements in the unique degree $d$ extension of $F_\mathfrak{p}$ inside $F_\mathfrak{P}$.  In particular, there are at most $|\mathfrak{p}|^d$ choices for the entries in this orbit, and thus $|\mathfrak{p}|^{\#\mathrm{Orb}(I_\mathfrak{P})}$ choices overall for $\iota_\mathfrak{P}(\alpha)$, as claimed.
		
		Finally, observe that if $\mathfrak{P}^\prime = \mathfrak{P}^g$ is a prime conjugate to $\mathfrak{P}$, then $\iota_{\mathfrak{P}^\prime}(\alpha) = \iota_{\mathfrak{P}}(\alpha)^g$ for any $\alpha \in \mathcal{O}_K$ when we regard $F_\mathfrak{P}$ and $F_{\mathfrak{P}^\prime}$ as lying inside a fixed choice of algebraic closure of $F_\mathfrak{p}$.  Since $I_{\mathfrak{P}^\prime} = I_{\mathfrak{P}}^g$, it follows that the set $\mathcal{R}_\mathfrak{p}(\mathcal{I})$ depends only on $\mathfrak{p}$, $\mathcal{I}$, and the conjugacy class of the inertia subgroup $I_\mathfrak{P}$.
	\end{proof}
	
	\begin{remark}
		In the case of wild ramification, it is possible to adapt the proof of Lemma \ref{lem:invariant-congruences-primes} to exploit the action of the higher ramification groups $G_i$, $i \geq 1$, to obtain nontrivial congruence restrictions on the values of invariants modulo larger powers of the prime $\mathfrak{p}$.  These restrictions are somewhat less clean than those from the tame case, and the improvement to bounds on $G$-extensions arising from incorporating these greater restrictions is limited to the dependence of the bounds on $G$, $n$, and $[k:\mathbb{Q}]$, and not in the powers of $X$ or $|\mathrm{Disc}(k)|$.  Thus, we anticipate that for most applications, the restriction to the tame case will be sufficient.  
	\end{remark}
	
	By applying the Chinese remainder theorem, Lemma \ref{lem:invariant-congruences-primes} shows that the $G$-invariants associated with any $G$-extension $K/k$ are subject to congruence conditions with relative density $|\mathfrak{D}_{K/k}^\mathrm{tame}|^{-1}$.  However, in order to meaningfully apply this, we need to understand both how close $|\mathfrak{D}_{K/k}^\mathrm{tame}|$ is to $|\mathfrak{D}_{K/k}|$ (equivalently, to understand how large $|\mathfrak{D}_{K/k}^\mathrm{wild}|$ may be), and to be able to account for the number of choices of the inertia subgroups at the different primes dividing $\mathfrak{D}_{K/k}$.  
	Lemma \ref{lem:wild-bound} accomplishes the first of these tasks and Lemma \ref{lem:inertia-choices} accomplishes the second.
	
	To bound the wild part of the discriminant, we employ an argument we learned of from \cite[Discussion around (14)]{FraczykHarcosMaga}.
	
	\begin{lemma}\label{lem:wild-bound}
		Let $k$ be a number field of degree $d$ and let $K/k$ be an extension of degree $n$.  If $n=2$, then $|\mathfrak{D}_{K/k}^{\mathrm{wild}}| \leq 4^d$; if $n = 3$, then $|\mathfrak{D}_{K/k}^\mathrm{wild}| \leq 216^d$; and if $n \geq 4$, then $|\mathfrak{D}_{K/k}^\mathrm{wild}| \leq 4^{ dn^2}$.
	\end{lemma}
	\begin{proof}
		We first rewrite the expression for $\mathfrak{D}_{K/k}^\mathrm{tame}$.  Let $\mathfrak{p}$ be a prime of $k$, and, for a prime $\mathfrak{P}$ of $K$ lying over $\mathfrak{p}$, let $f_{\mathfrak{P}/\mathfrak{p}}$ and $e_{\mathfrak{P}/\mathfrak{p}}$ be the associated inertial degree and ramification index.  Each such prime $\mathfrak{P}$ corresponds to an orbit of (the conjugacy class of) the decomposition group over $\mathfrak{p}$, and each such orbit further breaks up as a union of orbits under the inertia subgroup $G_0$.  In fact, the orbit corresponding to the prime $\mathfrak{P}$ breaks up as a union of $f_{\mathfrak{P}/\mathfrak{p}}$ $G_0$-orbits of size $e_{\mathfrak{P}/\mathfrak{p}}$.  Hence, we see that $\#\mathrm{Orb}(G_0) = \sum_{\mathfrak{P} \mid \mathfrak{p}} f_{\mathfrak{P}/\mathfrak{p}}$, so that
			\begin{equation}\label{eqn:tame-rewrite}
				\mathfrak{D}_{K/k}^{\mathrm{tame}}
					= \prod_{\mathfrak{p}} \mathfrak{p}^{n-\sum_{\mathfrak{P} \mid \mathfrak{p}} f_{\mathfrak{P}/\mathfrak{p}}}
					= \prod_{\mathfrak{p}} \mathfrak{p}^{\sum_{\mathfrak{P} \mid \mathfrak{p}} f_{\mathfrak{P}/\mathfrak{p}}(e_{\mathfrak{P}/\mathfrak{p}}-1)}.
			\end{equation}
		
		Now, let $\mathfrak{d}_{K/k}$ be the different ideal of $K/k$.  By \cite[Theorem III.2.6]{Neukirch}, for any prime $\mathfrak{P}$ of $K$ lying over a prime $\mathfrak{p}$ of $k$, we have $v_{\mathfrak{P}}(\mathfrak{d}_{K/k}) = e_{\mathfrak{P}/\mathfrak{p}} - 1$ if $\mathfrak{P}$ is tamely ramified and $e_{\mathfrak{P}/\mathfrak{p}} \leq v_{\mathfrak{P}}(\mathfrak{d}_{K/k}) \leq e_{\mathfrak{P}/\mathfrak{p}} -1 + v_{\mathfrak{P}}(e_{\mathfrak{P}/\mathfrak{p}})$ if $\mathfrak{P}$ is wildly ramified.  Since $\mathrm{Norm}_{K/k} \mathfrak{P} = \mathfrak{p}^{f_{\mathfrak{P}/\mathfrak{p}}}$, it follows from this discussion and \eqref{eqn:tame-rewrite} that for any prime $\mathfrak{p}$ of $k$ wildly ramifying in $K$ that
			\[
				v_\mathfrak{p}(\mathfrak{D}_{K/k}^\mathrm{wild})
					\leq \sum_{\mathfrak{P} \mid \mathfrak{p}} f_{\mathfrak{P}/\mathfrak{p}} v_\mathfrak{P}(e_{\mathfrak{P}/\mathfrak{p}})
					= \sum_{\mathfrak{P} \mid \mathfrak{p}} f_{\mathfrak{P}/\mathfrak{p}} e_{\mathfrak{P}/\mathfrak{p}} v_\mathfrak{p}(e_{\mathfrak{P}/\mathfrak{p}})
					\leq n \cdot \max\{ v_\mathfrak{p}(e_{\mathfrak{P}/\mathfrak{p}}) : \mathfrak{P} \mid \mathfrak{p} \}.
			\]
		Taking norms and observing that any wildly ramified prime must divide $n!$, we see that
			\[
				|\mathfrak{D}_{K/k}^\mathrm{wild}|
					\leq \prod_{p \leq n} p^{ n \sum_{\mathfrak{p} \mid p} f_{\mathfrak{p}/p} \max\{ v_\mathfrak{p}(e_{\mathfrak{P}/\mathfrak{p}}) : \mathfrak{P} \mid \mathfrak{p} \} }
					\leq \prod_{p \leq n} p^{ dn \max\{v_p(e_{\mathfrak{P}/\mathfrak{p}}) : \mathfrak{P} \mid \mathfrak{p} \mid p\}}.
			\]
		Since any ramification index $e_{\mathfrak{P}/\mathfrak{p}}$ is at most $n$, we have that $v_p( e_{\mathfrak{P}/\mathfrak{p}}) \leq \lfloor \frac{\log n}{ \log p} \rfloor$ for any prime $p$.  Hence, we conclude that
			\[
				|\mathfrak{D}_{K/k}^\mathrm{wild}|
					\leq \prod_{p \leq n} p^{ \frac{ dn \log n}{\log p}}
					= \prod_{p \leq n} e^{ dn \log n}
					< 4^{dn^2},
			\]
		where the last expression is obtained from the inequality
			\[
				\#\{p \leq n\} < \frac{n \log 4}{\log n}
			\]
		valid for all $n \geq 2$ (which follows straightforwardly from work of Rosser and Schoenfeld \cite[(3.6)]{RosserSchoenfeld}, for example).  The claims for $n=2$ and $n=3$ follow exactly as above, but without appealing to Rosser and Schoenfeld's work.
	\end{proof}
	
	\begin{lemma}\label{lem:inertia-choices}
		Let $G$ be a transitive group of degree $n$.  For each prime $p$, let $\mathcal{I}_p(G)$ be the set of conjugacy classes of subgroups $G_0 \subseteq G$ with a unique (hence normal) Sylow $p$-subgroup $G_1$ such that $G_0/G_1$ is cyclic.  (In particular, if $p \nmid |G|$, then $\mathcal{I}_p(G)$ is the set of conjugacy classes of cyclic subgroups.)  For any integer $v$, $0 \leq v \leq n-1$, let $b(p,v)$ be the number of classes of $G_0 \in \mathcal{I}_p(G)$ with $\#\mathrm{Orb}(G_0) = n-v$.  Finally, for any number field $k$, define a multiplicative function $\tau_G$ on the set of integral ideals of $k$ by setting $\tau_G(\mathfrak{p}^v) = b(p,v)$ for each prime ideal $\mathfrak{p}$, where $p$ is the rational prime lying below $\mathfrak{p}$.
		
		Then for any $X \geq e$ (where $e$ denotes the base of the natural logarithm), there holds
			\[
				\sum_{\substack{\mathfrak{a} \subseteq \mathcal{O}_k:\\ |\mathfrak{a}| \leq X}} \tau_G(\mathfrak{a})
					\leq (4n)^{d|G|} X^{\frac{1}{\mathrm{ind}(G)}} \left(\frac{2}{\mathrm{ind}(G)} \log X\right)^{db-1},
			\]
		where $\mathrm{ind}(G)$ is the minimum value of $v \geq 1$ such that $b(p,v) \ne 0$ for any prime $p$, and where $b = b(p,\mathrm{ind}(G))$ for any prime $p \nmid |G|$.
	\end{lemma}
	\begin{proof}
		We begin by considering the case $k=\mathbb{Q}$, identifying ideals of $\mathbb{Z}$ with positive integers $m$.  First, let $\mathcal{I}^\mathrm{cyc}(G)$ consist of the conjugacy classes of cyclic subgroups of $G$, and let $b^\mathrm{cyc}(v)$ be the number of classes of $G_0 \in \mathcal{I}^\mathrm{cyc}(G)$ for which $\#\mathrm{Orb}(G_0) = n-v$.  Thus, $b(p,v) = b^\mathrm{cyc}(v)$ for any prime $p$ not dividing $|G|$.  If we let $S(G)$ denote the number of subgroups of $G$, then at primes $p \mid |G|$, we have the trivial inequality $b(p,v) \leq S(G) b^\mathrm{cyc}(v)$.  In particular, if we define a multiplicative function $\tau_G^\mathrm{cyc}$ by setting $\tau_G^\mathrm{cyc}(p^v) = b^\mathrm{cyc}(v)$, then we have the equally trivial inequality $\tau_G(m) \leq S(G)^{\omega(|G|)} \tau_G^\mathrm{cyc}(m)$ for any integer $m$, where $\omega(|G|)$ denotes the number of primes dividing $|G|$.  Thus, to bound the sum of $\tau_G(m)$, it suffices to bound the sum of $\tau_G^\mathrm{cyc}(m)$ and to bound the expression $S(G)^{\omega(|G|)}$.  
		
		We split $\tau_G^\mathrm{cyc}$ into two further pieces.  Let $a=\mathrm{ind}(G)$ be the smallest value of $v$ such that $b^\mathrm{cyc}(v)$ is nonzero, and define a multiplicative function $\tau_1$ supported on $a$-th powers by setting $\tau_1(m^a) = d_b(m)$, where $b = b^\mathrm{cyc}(a)$ and $d_b(m)$ denotes the standard $b$-fold divisor function (i.e., the number of ways of expressing $m$ as a product of $b$ ordered integers).  Also define $\tau_2$ by setting $\tau_2(p^a)=0$ and $\tau_2(p^v) = b^\mathrm{cyc}(v)$ for every $v\ne a$.  Thus, for any $m$, $\tau_G^\mathrm{cyc}(m) \leq (\tau_1 * \tau_2)(m)$, where $*$ denotes Dirichlet convolution.  From \cite[Theorem 4.1]{Bordelles}, we have for any $X \geq 1$ that
			\[
				\sum_{m \leq X} d_b(m) 
					\leq X \left(\log X + \gamma + \frac{1}{X}\right)^{b-1},
			\]
		where $\gamma = 0.577\dots$ is the Euler--Mascheroni constant.  We therefore obtain
			\begin{align*}
				\sum_{m \leq X} \tau_G^\mathrm{cyc}(m)
					&\leq \sum_{q \leq X} \tau_2(q) \sum_{m \leq (X/q)^{1/a}} d_b(m) \\
					&\leq X^{\frac{1}{a}}\left(\frac{\log X}{a} + \gamma + \frac{1}{X^{1/a}}\right)^{b-1} \sum_{q \geq 1} \frac{\tau_2(q)}{q^{1/a}},
			\end{align*}
		the sum over $q$ converging since $\tau_2(q)$ is supported on integers that are at least $a+1$ powerful.  We now observe that
			\[
				\sum_{q \geq 1} \frac{\tau_2(q)}{q^{1/a}} = \prod_p \left( 1 + \sum_{k=a+1}^{n-1} \frac{b^\mathrm{cyc}(k)}{p^{k/a}} \right) \leq \zeta\left(1 + \frac{1}{a}\right)^{|G|} \leq n^{|G|},
			\]
		with the last inequality following since $\zeta(1+x) \leq 1 + \frac{1}{x}$ for any $x > 0$ and the fact that $a \leq n-1$.  For any $X \geq 2^a$, we note that $\log X^{1/a} + \gamma + \frac{1}{X^{1/a}} \leq 2 \log X^{\frac{1}{a}}$ and conclude that
			\[
				\sum_{m \leq X} \tau_G^\mathrm{cyc}(m)
					\leq n^{|G|} X^{\frac{1}{a}} \left(\frac{2}{a} \log X\right)^{b-1}.
			\]
		In fact, the same conclusion holds for any $e \leq X < 2^a$, since in this case, the left-hand side is equal to $1$.
		
		We now consider the contribution of the factor $S(G)^{\omega(|G|)}$.  An easy induction argument based on Sylow subgroups shows that any subgroup of $G$ is generated by at most $\Omega(|G|)$ elements, where $\Omega(|G|)$ denotes the number of prime factors of $|G|$ counted with multiplicity.  Hence, $S(G) \leq |G|^{\Omega(|G|)}$, and we find $S(G)^{\omega(|G|)} \leq \exp\left( \frac{(\log |G|)^3}{(\log 2)^2} \right) \leq 4^{|G|}$.  This yields the conclusion if $k = \mathbb{Q}$.
		
		If $k \ne \mathbb{Q}$, then define the multiplicative function $\widetilde{\tau}_G$ on integers $m$ by $\widetilde{\tau}_G(m) = \sum_{|\mathfrak{a}| = m} \tau_G(\mathfrak{a})$, so that
			\[
				\sum_{|\mathfrak{a}| \leq X} \tau_G(\mathfrak{a}) = \sum_{m \leq X} \widetilde{\tau}_G(m).
			\]
		Modeled on the case $k=\mathbb{Q}$, we then observe that $\widetilde{\tau}_G(m) \leq S(G)^{d \omega(|G|)} (\widetilde{\tau}_1 * \widetilde{\tau}_2)(m)$, where $\widetilde{\tau}_1$ is a multiplicative function supported on $a$-th powers satisfying $\widetilde{\tau}_1(m^a) = d_{[k:\mathbb{Q}]b}(m^a)$, and where $\widetilde{\tau}_2$ is supported on integers that are $(a+1)$-powerful and satisfies $\sum_{i=a+1}^\infty \widetilde{\tau}_2(p^i) \leq d |G|$ for every prime $p$.  Proceeding exactly as in the case $k=\mathbb{Q}$ then yields the claim.
	\end{proof}
	
	We're finally ready to present the second improvement to Theorem \ref{thm:invariant-theory-simple}.
	
	\begin{theorem}\label{thm:invariant-theory-full}
		Let $G$ be a transitive permutation group of degree $n\geq 3$ and let $\mathcal{I}=\{f_1,\dots,f_{nr}\}$ be a set of $nr$ algebraically independent invariants of $G$ in its action on $(\mathbb{A}^n)^r$.  Let $\mathrm{ind}(G)$ denote the index of $G$.  Then for any number field $k$ of degree $d$, any $X \geq e$, and any $\lambda \geq 1$, we have
			\begin{align*}
				&\#\{ K \in \mathcal{F}_{n,k}(X;G) : \lambda_{\max}(K) \leq \lambda\} \\
					& \quad\leq 2^{-r(n-1)} (\sqrt{8\pi}\cdot  8^{|G|})^{dnr} (d+1)!^{nr} \mathcal{I}(1)^d \cdot \left(2dn \left(\deg \mathcal{I} + \binom{n}{2}\right)\right)^{d \cdot \deg \mathcal{I}}  \cdot \\
					& \quad\quad\cdot  \frac{\lambda^{d \cdot \deg \mathcal{I}-dnr} |\mathrm{Disc}(k)|^{rn/2} }{X^{\frac{r}{2}-\frac{1}{\mathrm{ind}(G)}}} \left(\frac{2\log X }{\mathrm{ind}(G)} \right)^{db-1} \prod_{i \leq nr} \max\left\{ \frac{X^{\frac{1}{\mathrm{ind}(G)}} |\mathrm{Disc}(k)|^{\frac{3}{2}-\frac{n}{\mathrm{ind}(G)}}}{ \lambda^{d\cdot \deg f_i}}, 1\right\},
			\end{align*}
		and in particular, we have
			\begin{align*}
				& \#\mathcal{F}_{n,k}(X;G) \\
					& \quad \leq 2^{-r(n-1)} (\sqrt{8\pi}\cdot  8^{|G|})^{dnr} (d+1)!^{nr} \mathcal{I}(1)^d \cdot \left(2dn \left(\deg \mathcal{I} + \binom{n}{2}\right)\right)^{d \cdot \deg \mathcal{I}} \cdot |\mathrm{Disc}(k)|^{rn/2} \cdot \\
					& \quad\quad  X^{\frac{1}{n}\deg \mathcal{I} - \frac{3r}{2} + \frac{1}{\mathrm{ind}(G)}} \left(\frac{2 \log X}{\mathrm{ind}(G)}\right)^{db-1} \!\!\!\!\!\!\prod_{\deg f_i \leq \frac{n}{\mathrm{ind}(G)}}\!\!\!\! \max\left\{ X^{\frac{1}{\mathrm{ind}(G)}-\frac{\deg f_i}{n}} |\mathrm{Disc}(k)|^{\frac{3}{2}-\frac{n}{\mathrm{ind}(G)}}, 1 \right\}.
			\end{align*}
	\end{theorem}
	\begin{proof}
		We proceed as in the proof of Theorem \ref{thm:invariant-theory-multiplicity}, except incorporating the congruence restrictions on the invariants coming from Lemma \ref{lem:invariant-congruences-primes}.  Let $K \in \mathcal{F}_{n,k}(X;G)$ be such that $\lambda_{\max}(K) \leq \lambda$ and $|\mathrm{Disc}(K)| \geq X/2$, and let $\mathfrak{C} := \prod_{\mathfrak{p} \mid \mathfrak{D}_{K/k}} \mathfrak{p}$.  Since the congruence conditions from Lemma \ref{lem:invariant-congruences-primes} depend only on the conjugacy class of inertia at each ramified prime, it follows that the invariants $(f(\alpha) : f \in \mathcal{I})$ for any $\alpha \in \mathcal{O}_K^r$ must lie in one of at most $\tau_G(\mathfrak{D}_{K/k}^\mathrm{tame}) |\mathfrak{C}|^{rn} |\mathfrak{D}_{K/k}^\mathrm{tame}|^{-r}$ classes $\pmod{\mathfrak{C}}$, where $\tau_G$ is as in Lemma \ref{lem:inertia-choices}.  
		We observe that $|\mathfrak{C}| \leq |\mathfrak{D}_{K/k}|^{1/\mathrm{ind}(G)} \leq X^{1/\mathrm{ind}(G)} |\mathrm{Disc}(k)|^{-n/\mathrm{ind}(G)}$.  Hence, by Lemma \ref{lem:minkowski-ideal}, $\lambda_d(\mathfrak{C}) \leq X^{\frac{1}{d\cdot\mathrm{ind}(G)}} |\mathrm{Disc}(k)|^{\frac{3}{2d}-\frac{n}{d\cdot\mathrm{ind}(G)}}$.  
		
		Now, as in the proof of Theorem \ref{thm:invariant-theory-multiplicity}, the invariant $f_i(\alpha)$ for any $\alpha \in \mathcal{O}_K^r$ has height bounded by $|f_i|(1)\cdot(dn \deg F)^{\deg f_i} \lambda^{\deg f_i}$.  Hence, if $\lambda^{\deg f_i} \geq X^{\frac{1}{d\cdot\mathrm{ind}(G)}} |\mathrm{Disc}(k)|^{\frac{3}{2d}-\frac{n}{d\cdot\mathrm{ind}(G)}}$, then by Lemma \ref{lem:ideal-bounded-height}, the number of values of $f_i(\alpha)$ satisfying a given congruence condition $\pmod{\mathfrak{C}}$ is at most
			\begin{equation}\label{eqn:invariants-satisfying-congruence}
				(8\pi)^{d/2} (d+1)! |f_i|(1)^d \cdot (dn \deg F)^{d\cdot \deg f_i} \frac{\lambda^{d\cdot \deg f_i}}{|\mathfrak{C}| \sqrt{|\mathrm{Disc}(k)|}}.
			\end{equation}
		If $\lambda^{\deg f_i} < X^{\frac{1}{d\cdot\mathrm{ind}(G)}} |\mathrm{Disc}(k)|^{\frac{3}{2d}-\frac{n}{d\cdot\mathrm{ind}(G)}}$, then the number may be instead bounded by 
			\[
				X^{\frac{1}{\mathrm{ind}(G)}} |\mathrm{Disc}(k)|^{\frac{3}{2}-\frac{n}{\mathrm{ind}(G)}} \lambda^{-d\cdot \deg f_i}
			\]
		times the quantity \eqref{eqn:invariants-satisfying-congruence}.  All told, the number of values of the invariants $f_i(\alpha)$ satisfying the congruence conditions from Lemma \ref{lem:invariant-congruences-primes} is at most
			\[
				\tau_G(\mathfrak{D}_{K/k}^\mathrm{tame}) (8 \pi)^{dnr/2} (d+1)!^{nr} \mathcal{I}(1)^d \cdot (dn \deg F)^{d \cdot \deg \mathcal{I}} \frac{\lambda^{d \cdot \deg \mathcal{I}}}{|\mathrm{Disc}(k)|^{rn/2} |\mathfrak{D}_{K/k}^\mathrm{tame}|^r} \cdot \mathcal{E},
			\]
		where 
			\[
				\mathcal{E}
					:= \prod_{i\leq nr} \max\left\{ \frac{X^{\frac{1}{\mathrm{ind}(G)}} |\mathrm{Disc}(k)|^{\frac{3}{2}-\frac{n}{\mathrm{ind}(G)}}}{ \lambda^{d\cdot \deg f_i}}, 1\right\}.
			\]
		By the conductor--discriminant formula and Lemma \ref{lem:wild-bound}, we observe that
			\[
				|\mathfrak{D}_{K/k}^\mathrm{tame}|^{-r}
					= \frac{|\mathfrak{D}_{K/k}^\mathrm{wild}|^r |\mathrm{Disc}(k)|^{rn}}{|\mathrm{Disc}(K)|^r}
					\leq 4^{dn^2r} \frac{2^r|\mathrm{Disc}(k)|^{rn}}{X^r}.
			\]
		Hence, the number of choices for the invariants is at most
			\[
				\tau_G(\mathfrak{D}_{K/k}^\mathrm{tame}) \cdot (8 \pi)^{dnr/2} (d+1)!^{nr} \mathcal{I}(1)^d \cdot (dn \deg F)^{d \cdot \deg \mathcal{I}} \cdot 4^{dn^2r} \cdot \frac{2^r \lambda^{d \cdot \deg \mathcal{I}} |\mathrm{Disc}(k)|^{rn/2} }{X^r} \cdot \mathcal{E}.
			\]
		As in the proof of Theorem \ref{thm:invariant-theory-multiplicity}, each field is counted with multiplicity at least $\lambda^{dnr}/X^{r/2}$, and there are at most $2^{\deg \mathcal{I} - nr}$ possible extensions given each set of invariants.  Since $|\mathfrak{D}_{K/k}^\mathrm{tame}| \leq X$, all told, it follows from Lemma \ref{lem:inertia-choices} that the number of $G$-extensions $K/k$ with $X/2 < |\mathrm{Disc}(K)| \leq X$ and $\lambda_{\max}(K) \leq \lambda$ is at most
			\[
				2^{-r(n-1)} (\sqrt{8 \pi}\cdot  8^{|G|})^{dnr} (d+1)!^{nr} \mathcal{I}(1)^d \cdot (2dn \deg F)^{d \cdot \deg \mathcal{I}} \frac{\lambda^{d \cdot \deg \mathcal{I}-dnr} |\mathrm{Disc}(k)|^{rn/2} }{X^{\frac{r}{2}-\frac{1}{\mathrm{ind}(G)}}} \cdot \left(\frac{2\log X }{\mathrm{ind}(G)} \right)^{db-1} \!\!\!\! \cdot \mathcal{E},
			\]
		provided that $n \geq 4$; if $n=2$ or $n=3$, then replacing the factor $4^{dn^2r}$ above by $4^{dr}$ and $216^{dr}$, respectively, reveals that this conclusion holds in these cases too.
		After accounting for the dyadic summation, and exploiting that $2^{-r(n-1)} \leq 1/4$ by our assumption that $n\geq 3$, the first claim follows.  
		
		For the second claim, we take $\lambda = X^{\frac{1}{dn}}$, which suffices by Lemma \ref{lem:largest-minimum}.  So doing, we observe that the inequality $\lambda^{d\cdot \deg f_i} \geq X^{\frac{1}{\mathrm{ind}(G)}} |\mathrm{Disc}(k)|^{\frac{3}{2}-\frac{n}{\mathrm{ind}(G)}}$ holds whenever $X^{\frac{\deg f_i}{n} - \frac{1}{\mathrm{ind}(G)}} \geq |\mathrm{Disc}(k)|^{\frac{3}{2} - \frac{n}{\mathrm{ind}(G)}}$.  If $\deg f_i \geq \frac{n}{\mathrm{ind}(G)}$ (which implies that $\deg f_i \geq 2$), then this inequality must hold, since we may assume that $X \geq |\mathrm{Disc}(k)|^n$.  On the other hand, if $\deg f_i < \frac{n}{\mathrm{ind}(G)}$, then this inequality fails for sufficiently large $X$, and thus the condition cannot be ignored for such invariants.  
	\end{proof}
	
	\begin{remark}
		In comparing the dependence of Theorem \ref{thm:invariant-theory-full} on the degrees $d$ and $n$ with that in Theorem \ref{thm:invariant-theory-multiplicity}, we observe that Theorem \ref{thm:invariant-theory-full} loses a factor of $8^{|G|dnr}$ while gaining at most a factor of $X^r$.  Thus, Theorem \ref{thm:invariant-theory-full} is beneficial only if $|G|dn \leq C \log X$ for some constant $C$.  Since $G$ is transitive of degree $n$, we must have $|G| \geq n$, and thus Theorem \ref{thm:invariant-theory-full} can be useful only if $dn^2 \ll \log X$.  However, the Minkowski-style bound $|\mathrm{Disc}(K)| \geq c^{dn}$ for a constant $c>1$ implies that one might have to consider degrees $d$ and $n$ satisfying $dn \ll \log X$.  In particular, if the degree $d$ is fixed (as would be the case if the field $k$ is fixed) while the degree $n$ varies, then Theorem \ref{thm:invariant-theory-full} is useful only in a thin portion of the Minkowski range.  On the other hand, if the parameter $n$ is fixed, then Theorem \ref{thm:invariant-theory-full} is beneficial in a healthier portion of the Minkowski range $d \ll_n \log X$.  It is conceivable that by incorporating refinements from higher ramification groups in Lemma \ref{lem:invariant-congruences-primes} and by being more careful in suitable analogues of Lemmas \ref{lem:wild-bound} and \ref{lem:inertia-choices} that the range of applicability may be extended, but we anticipate that for most applications in which the full Minkowski range is relevant that Theorem \ref{thm:invariant-theory-multiplicity} will be superior.
	\end{remark}

\subsection{Further improvements from independence of power sums}
	\label{subsec:power-sums}
	
	We will principally make use of Theorem \ref{thm:invariant-theory-full} for elemental primitive groups not containing $A_n$.  For such groups, it follows from Lemma \ref{lem:elemental-index} that the ratio $\frac{n}{\mathrm{ind}(G)}$ is at most $\frac{14}{3}$.  Thus, for such groups, the product over the invariants $f_i$ such that $\deg f_i \leq \frac{n}{\mathrm{ind}(G)}$ is limited, at worst, to those invariants of degree at most $4$.  The extra cost of these small invariants is therefore mostly inconsequential for large groups, but for specific cases and the aesthetics of our inductive argument, we find it convenient to note that one may often avoid even this small cost by using ideas of Bhargava \cite{Bhargava-vdW} and incorporating \emph{power sums} into the set of invariants.  
		
	A power sum associated with $(\mathbb{A}^n)^r$ is any polynomial $p_{i,j}$ of the form $x_{j,1}^i + \dots + x_{j,n}^i$ with $1 \leq j \leq r$, where we take the coordinate ring of $(\mathbb{A}^n)^r$ to be $\mathbb{Z}[x_{1,1},\dots,x_{1,n},\dots,x_{r,1},\dots,x_{r,n}]$.  Any power sum is evidently a $G$-invariant for any permutation group $G$ of degree $n$.  We first show that one may iteratively replace elements of any algebraically independent set $\mathcal{I}$ of invariants by power sums.  In practice, this will let us assume that all small invariants in $\mathcal{I}$ are power sums.
	
	\begin{lemma}\label{lem:replace-by-power-sums}
		Let $G$ be a transitive permutaiton group of degree $n$, and let $\mathcal{I} = \{f_1,\dots,f_{nr}\}$ be a set of $nr$ algebraically independent $G$-invariants in its action on $(\mathbb{A}^n)^r$, and suppose that each $f \in \mathcal{I}$ has degree at most $n$.  Let $p$ be a power sum on $(\mathbb{A}^n)^r$ of degree at most $n$, and suppose that $p \not\in \mathcal{I}$.  Then there is some $f_i \in \mathcal{I}$ not equal to a power sum such that the set $\{p\} \cup \mathcal{I} \setminus \{f_i\}$ is algebraically independent.
	\end{lemma}
	\begin{proof}
		Since $p \not\in \mathcal{I}$, the set $\mathcal{I} \cup \{p\}$ consists of $nr+1$ elements, and thus cannot be algebraically independent.  There is therefore a nontrivial algebraic relation satisfied by its elements.  Explicitly, there must be some nontrivial polynomial $R \in \mathbb{Z}[P,F_1,\dots,F_{nr}]$ such that $R(p,f_1,\dots,f_{nr}) = 0$, where we regard $P,F_1,\dots,F_{nr}$ as formal variables.  This relation $R$ must have degree at least one in $P$, since the set $\mathcal{I}$ is algebraically independent.  We may also assume that, as a polynomial in $P$, the coefficients of $R$ are relatively prime.  Additionally, the relation cannot involve only power sums of degree at most $n$ in $\mathcal{I}$, since the set $\{p_{i,j}\}_{1 \leq i \leq n, 1 \leq j \leq r}$ is algebraically independent.  Thus, there is some non-power sum $f_i\in \mathcal{I}$ involved in the relation (i.e., such that the polynomial $P$ is not constant in $F_i$), and this $f_i$ satisfies the conclusion of the lemma.
	\end{proof}
	
	By iteratively applying Lemma \ref{lem:replace-by-power-sums}, we may assume that the set $\mathcal{I}$ contains all power sums of degree $1$, all powers sums of degree $2$, and so on, until it is no longer advantageous to make the replacement (e.g. if the added power sum is of larger degree than the invariant it is replacing).  Thus, we now consider sets $\mathcal{I}$ where there is some $m \geq 1$ such that the first $mr$ invariants $f_i \in \mathcal{I}$ are the power sums of degree $\leq m$.  The advantage of this is that we will show that the congruence conditions imposed by Lemma \ref{lem:invariant-congruences-primes} are approximately equidistributed in fibers of the projection to $f_1,\dots,f_{mr}$.
		
	We begin by recalling the following lemma, proved in an equivalent form in \cite[Proposition 21]{Bhargava-vdW}.
	
	\begin{lemma}\label{lem:power-sum-congruence}
		Let $k$ be a number field, let $n\geq 2$, let $\mathfrak{p}$ be a prime ideal of $k$ not dividing $n!$, and let $v$ be an integer satisfying $1 \leq v \leq n-1$.   Then for any $c_1,\dots,c_{n-v} \in \mathcal{O}_k$, the values $c_{n-v+1},\dots,c_n \in \mathcal{O}_k$ such that the polynomial $f(x) = x^n + c_1 x^{n-1} + \dots + c_n$ has $v_\mathfrak{p}(\mathrm{Disc}(f)) \geq v$ lie in at most $O_n(1)$ congruence classes $\pmod{\mathfrak{p}}$.
	\end{lemma}
	
	\begin{remark} 
		Notice that we have not made the $O_n(1)$ term in Lemma \ref{lem:power-sum-congruence} explicit.  In a break from the previous sections, we do not pursue fully explicit, or even fully uniform, bounds in this section.  This is for two reasons.  First, since the improvement to the power of $X$ is minimal (typically $O(1/n)$), we expect this improvement to be principally interesting when $n$ is small, so that the issue of dependence on $n$ is irrelevant.  Second, should one work it out, the dependence on $n$ would be substantially worse even than Theorem \ref{thm:invariant-theory-full}; see the remark following the proof of Theorem \ref{thm:invariant-theory-power-sums}.  
	\end{remark}
	
	Using this, we find:
	
	\begin{lemma}\label{lem:some-power-sums}
		Let $G$ be a transitive permutation group of degree $n$ and let $\mathcal{I}=\{f_1,\dots,f_{nr}\}$ be a set of algebraically independent invariants of $G$ in its action on $(\mathbb{A}^n)^r$ with coefficients in $\mathbb{Z}$, and suppose there is some $m \geq 1$ such that each $f_i$ with $i \leq mr$ is a power sum of degree at most $m$.  For a number field $k$, let $\mathfrak{D}$ be an integral ideal of $k$ that occurs as the relative discriminant of at least one $G$-extension $K/k$, and let $\mathfrak{p}$ be a prime ideal dividing $\mathfrak{D}$ that does not divide $n!$.  Set $v = v_\mathfrak{p}(\mathfrak{D})$.  Then:
		
		If $v > n-m$, then for any $r(n-v)$-tuple $( a_{i,j} )  \in (\mathcal{O}_k^{n-v})^r$, there is a set $\mathcal{R} \subseteq (\mathcal{O}_k/\mathfrak{p})^{rv}$ of size $O_{n,r}(1)$ such that for any $G$-extension $K/k$ with discriminant $\mathfrak{D}$ and any primitive $\alpha \in \mathcal{O}_K^r$ satisfying $p_{i,j}(\iota(\alpha))=a_{i,j}$ for each $i \leq n-v$ and each $j \leq r$, then $(f_{r(n-v)+1}(\iota(\alpha)),\dots,f_{nr}(\iota(\alpha))) \in \mathcal{R}$.
		
		If $v \leq n-m$, then for any $mr$-tuple $(a_{i,j}) \in (\mathcal{O}_k^m)^r$, there is a set $\mathcal{R} \subseteq (\mathcal{O}_k/\mathfrak{p})^{r(n-m)}$ of size $O_{n,r}(|\mathfrak{p}|^{r(n-m-v)})$ such that for any $G$-extension $K/k$ with discriminant $\mathfrak{D}$ and any primitive $\alpha \in \mathcal{O}_K^r$ satisfying $p_{i,j}(\iota(\alpha))=a_{i,j}$ for each $i \leq m$ and $j \leq r$, then $(f_{mr+1}(\iota(\alpha)),\dots,f_{nr}(\iota(\alpha))) \in \mathcal{R}$.
	\end{lemma}
	\begin{proof}
		We prove the claim for $r=1$, the case $r \geq 2$ being materially the same.  Let $K/k$ be a $G$-extension with discriminant $\mathfrak{D}$, and let $\alpha \in \mathcal{O}_K$ be such that $K = k(\alpha)$.  If $f_\alpha(x)$ denotes the minimal polynomial of $\alpha$ over $k$, then we have $v_\mathfrak{p}(\mathrm{Disc}(f_\alpha)) \geq v_\mathfrak{p}(\mathfrak{D}) = v$.  Write $f_\alpha(x) = x^n + c_1 x^{n-1} + \dots + c_n$.  If $v > n-m$, then the values $a_1,\dots,a_{n-v}$ determine $c_1,\dots,c_{n-v}$ by the assumption that each $f_i$ for $i \leq m$ is the power sum of degree $i$.  By Lemma \ref{lem:power-sum-congruence}, there are $O_n(1)$ choices for the remaining coefficients $c_{n-v+1},\dots,c_n \pmod{\mathfrak{p}}$, and thus also $O_n(1)$ choices for $\alpha$ over $\overline{k_\mathfrak{p}}$.  It thus follows that there are $O_n(1)$ choices $\pmod{\mathfrak{p}}$ for the remaining invariants $f_{n-v+1}(\iota(\alpha)),\dots,f_n(\iota(\alpha))$, as claimed.
		
		Similarly, if $v\leq n-m$, then the values $a_1,\dots,a_m$ determine $c_1,\dots,c_m$ as before.  There are $|\mathfrak{p}|^{n-m-v}$ choices for the coefficients $c_{m+1},\dots,c_{n-v} \pmod{\mathfrak{p}}$, which then determine up to $O_n(1)$ choices the remaining coefficients $c_{n-v+1},\dots,c_n \pmod{\mathfrak{p}}$.  It thus follows there are $O_n(|\mathfrak{p}|^{n-m-v})$ choices for $\alpha$ over $\overline{k_\mathfrak{p}}$, and the claim follows as before.
	\end{proof}
	
	This leads to the following refinement of Theorem \ref{thm:invariant-theory-full}.
	
	\begin{theorem} \label{thm:invariant-theory-power-sums}
		Let $G$ be a transitive permutation group of degree $n \geq 4$ and let $\mathcal{I}=\{f_1,\dots,f_{nr}\}$ be a set of $nr$ algebraically independent invariants of $G$ in its action on $(\mathbb{A}^n)^r$.  Assume for some $1 \leq m \leq n-1$ that each $f_i$ for $i \leq mr$ is a power sum of degree at most $m$.  Then for any number field $k$ of degree $d$, any $X$, and any $\lambda$, we have
			\begin{align*}
				&\#\{ K \in \mathcal{F}_{n,k}(X;G) : \lambda_{\max}(K) \leq \lambda\} \\
					&\quad \ll_{n,d,\mathcal{I},\epsilon} \lambda^{d \cdot \deg \mathcal{I} - dnr} X^{-\frac{r}{2}+\frac{1}{\mathrm{ind}(G)} + \epsilon} |\mathrm{Disc}(k)|^{\frac{rn}{2}} \cdot  \prod_{mr < i \leq nr} \max\left\{ \frac{X^{\frac{1}{\mathrm{ind}(G)}} |\mathrm{Disc}(k)|^{1-\frac{n}{\mathrm{ind}(G)}}}{ \lambda^{d\cdot \deg f_i}}, 1\right\},
			\end{align*}
		and in particular
			\begin{align*}
				& \#\mathcal{F}_{n,k}(X;G)
					\ll_{n,d,\mathcal{I},\epsilon} \\
					& \quad X^{\frac{1}{n} \deg\mathcal{I} -\frac{3r}{2}+\frac{1}{\mathrm{ind}(G)} + \epsilon} |\mathrm{Disc}(k)|^{\frac{rn}{2}} \prod_{\substack{ i>mr: \\ \deg f_i < \frac{n}{\mathrm{ind}(G)}}} \max\left\{X^{\frac{1}{\mathrm{ind}(G)} - \frac{\deg f_i}{n}}|\mathrm{Disc}(k)|^{1-\frac{n}{\mathrm{ind}(G)}},1\right\}.
			\end{align*}
	\end{theorem}
	\begin{proof}
		We proceed as in the proof of Theorem \ref{thm:invariant-theory-full}, except in our treatment of the congruence conditions on those $f_i$ with $\lambda^{\deg f_i} < X^{\frac{1}{d\cdot \mathrm{ind}(G)}} |\mathrm{Disc}(k)|^{\frac{3}{2d}-\frac{n}{d\cdot\mathrm{ind}(G)}}$ (which range will change slightly below, as we will now invoke Lemma \ref{lem:minkowski-ideal-asymptotic} instead of Lemma \ref{lem:minkowski-ideal}).  As in the proof of Theorem \ref{thm:invariant-theory-full}, let $\mathfrak{C} = \prod_{p \mid \mathfrak{D}_{K/k}} \mathfrak{p}$.  Let $a_1,\dots,a_{nr} \in \mathcal{O}_k$ denote putative values of $f_1(\iota(\alpha)), \dots, f_{nr}(\iota(\alpha))$.  We now consider more carefully the congruence conditions to which each $a_i$ is subject.  In particular, we wish to show that the density $\tau_G(\mathfrak{D}_{K/k}^\mathrm{tame}) |\mathfrak{D}_{K/k}^\mathrm{tame}|^{-1} = O_{n,d,\epsilon}(|\mathfrak{D}_{K/k}|^{-1+\epsilon})$ coming from Lemma \ref{lem:invariant-congruences-primes} and the Chinese remainder theorem may be realized upon fibering over the values $a_1,\dots,a_{mr}$.  Thus, let $\mathfrak{p} \mid \mathfrak{D}_{K/k}$, $\mathfrak{p} \nmid n!$, and let $v = v_p(\mathfrak{D}_{K/k})$.  We also assume that the invariants $f_1,\dots,f_r$ are the power sums of degree $1$, the invariants of degree $f_{r+1},\dots,f_{2r}$ are the power sums of degree $2$, and so on, up to assuming that $f_{(m-1)r+1},\dots,f_{mr}$ are the power sums of degree $m$.
		
		By Lemma \ref{lem:some-power-sums}, if $v \leq n-m$, then for any $a_1,\dots,a_{mr}$, there are $O_{n}(|\mathfrak{p}|^{r(n-m-v)})$ choices for the tuple $(a_{mr+1},\dots,a_{nr}) \pmod{\mathfrak{p}}$.  These conditions have relative density $O_n(|\mathfrak{p}|^{-rv})$, and thus in this case, we impose no congruence conditions $\pmod{\mathfrak{p}}$ on $a_1,\dots,a_{mr}$ and congruence conditions $\pmod{\mathfrak{p}}$ on $a_{mr+1},\dots,a_{nr}$.
		
		Next, we consider the case that $v > n-m$.  In this case, given any choices for $a_1,\dots,a_{r(n-v)}$, there are $O_{n}(1)$ choices $\pmod{\mathfrak{p}}$ for every other $a_i$ with $i \geq r(n-v)+1$ by Lemma \ref{lem:some-power-sums}.  
		
		Pulling together the different primes $\mathfrak{p}$, we see that each $a_i$ with $i > mr$ is subject to congruence conditions $\pmod{\mathfrak{C}}$ that depend on the $a_j$ with $j < i$; we treat these $a_i$ as in the proof of Theorem \ref{thm:invariant-theory-full}.  We also see that each $a_i$ with $(w-1)r+1 \leq i \leq wr$, $1 \leq w \leq m$, is subject to congruence conditions $\pmod{\mathfrak{D}_w}$ that depend on the $a_j$ with $j \leq (w-1)r$, where
			\[
				\mathfrak{D}_w := \prod_{\substack{ \mathfrak{p} \mid \mathfrak{D} \\ v_\mathfrak{p}(\mathfrak{D}) \geq n+1-w}}\mathfrak{p}.
			\]
		Note that the value of the first $r$ invariants (i.e., $a_1,\dots,a_r$) are subject to congruence constraints at most at primes dividing $n!$.  Since $\mathfrak{D}_w^{n+1-w} \mid \mathfrak{D}_{K/k}$ by definition, we find $|\mathfrak{D}_w| \leq |\mathfrak{D}_{K/k}|^{\frac{1}{n+1-w}} = |\mathrm{Disc}(K)|^{\frac{1}{n+1-w}} |\mathrm{Disc}(k)|^{-\frac{n}{n+1-w}}$.  
		Consequently, by Lemma \ref{lem:minkowski-ideal-asymptotic}, there holds $\lambda_{\max}(\mathfrak{D}_w) \ll_d |\mathrm{Disc}(K)|^{\frac{1}{d(n+1-w)}} |\mathrm{Disc}(k)|^{\frac{1}{d}-\frac{n}{d(n+1-w)}}$.  Next, by Lemma \ref{lem:lambda-lower-bound}, we may assume that $\lambda \gg_{n,d} |\mathrm{Disc}(K)|^{\frac{1}{2d(n-1)}} |\mathrm{Disc}(k)|^{-\frac{1}{2d(n-1)}}$.  A computation then reveals that for such $\lambda$, the inequality $\lambda^w \gg_{n,d} \lambda_{\max}(\mathfrak{D}_w)$ holds for each $2 \leq w \leq n-2$, and also for $w=n-1$ by virtue of our assumption that $n \geq 4$.  
		
		We now proceed essentially as in the proof of Theorem \ref{thm:invariant-theory-full}, but with the following modifications.  We impose no congruence constraints on the values $a_1,\dots,a_r$ of the first $r$ invariants; the number of choices for these in total is therefore $O_{n,d}(\lambda^{dr} / |\mathrm{Disc}(k)|^{r/2})$ by Corollary \ref{cor:integers-bounded-height} and Lemma \ref{lem:lambda-extension}.  For the next $r$ invariants $a_{r+1},\dots,a_{2r}$, we impose the requisite congruence conditions $\pmod{\mathfrak{D}_2}$ coming from Lemma \ref{lem:some-power-sums}.  
		Taking $H$ to be a sufficiently large multiple of $\lambda^2$ in terms of $d$ and $n$, by Lemma \ref{lem:ideal-bounded-height}, we see that the number of choices for $a_{r+1},\dots,a_{2r}$ is at most
			\[
				\ll_{n,d,\epsilon} \frac{\lambda^{2dr} X^\epsilon}{|\mathfrak{D}_2|^r |\mathrm{Disc}(k)|^{r/2}},
			\]
		the factor of $X^\epsilon$ arising from applying the Chinese remainder theorem to Lemma \ref{lem:some-power-sums} and a divisor bound.
		Similarly, iteratively considering each $3 \leq w \leq m$, the previous paragraph shows that the density of the congruence constraints on the values $a_{(w-1)r+1},\dots,a_{wr}$ $\pmod{\mathfrak{D}_w}$ may be realized by taking $H$ to be a sufficiently large multiple of $\lambda^w$ in Lemma \ref{lem:ideal-bounded-height}.  In particular, there are at most
			\[
				\ll_{n,d,\epsilon} \frac{\lambda^{wdr} X^\epsilon}{|\mathfrak{D}_w|^r |\mathrm{Disc}(k)|^{r/2}}
			\]
		choices for $a_{(w-1)r+1},\dots,a_{wr}$.  Finally, for the values $a_{mr+1},\dots,a_{nr}$, we impose the congruence conditions $\pmod{\mathfrak{C}}$ coming from Lemma \ref{lem:some-power-sums}, estimating the number of integers satisfying these congruences as in the proof of Theorem \ref{thm:invariant-theory-full} but using Lemma \ref{lem:minkowski-ideal-asymptotic} instead of Lemma \ref{lem:minkowski-ideal}.  This yields the theorem.
	\end{proof}
	
	\begin{remark}
		It is possible to obtain a completely explicit (or fully uniform) version of Theorem \ref{thm:invariant-theory-power-sums} by exploiting the ideas of the previous section.  In so doing, one would need an analogue of Lemma \ref{lem:inertia-choices} for the divisor-type function accounting for the $O_n(1)$ choices in Lemma \ref{lem:power-sum-congruence}.  The expression $O_n(1)$ is made explicit in \cite[Proposition 21]{Bhargava-vdW}; namely, it is bounded by $q(v,n-v)\cdot (n-v)!$, where for any integers $k,\ell$, $q(k,\ell)$ denotes the number of partitions of $k$ into at most $\ell$ parts.  For the values of $v$ for which we wish to apply Lemma \ref{lem:power-sum-congruence}, this expression is at least $n^{Cn}$ for some constant $C>0$.  Consequently, the $X^\epsilon$ term present in Theorem \ref{thm:invariant-theory-power-sums} could in principle be replaced by an expression of the order $(\log X)^{n^{O(n)}}$ using these ideas.  
	\end{remark}
	
\subsection{Skew lattices and a return to Schmidt's theorem}
	\label{subsec:skew-bounds}
	
	Finally, we present two arguments, based on Schmidt's theorem (or Theorem \ref{thm:uniform-schmidt}) and a theorem of Bhargava \cite[Theorem 20]{Bhargava-vdW}, that allow us to do somewhat better than Theorem \ref{thm:uniform-schmidt} if the ring of integers $\mathcal{O}_K$ is \emph{skew}, i.e. if $\lambda_{\max}(K)$ is rather larger than the other successive minima.  In particular, in contrast to Theorems \ref{thm:invariant-theory-multiplicity}, \ref{thm:invariant-theory-full}, and \ref{thm:invariant-theory-power-sums}, all of which consider fields $K$ for which $\lambda_{\max}(K) \leq \lambda$ for some $\lambda \geq 1$, here we are concerned with fields satisfying the complementary condition $\lambda_{\max}(K) > \lambda$.
	This will allow us to blend between the two regimes, providing stronger bounds in some cases.
	The results we prove here will be useful only for small $n$ (in fact, only bounded $n$, though these results will be instrumental in obtaining the best bound on extensions of degrees between $20$ and $85$; see Theorem \ref{thm:small-degree}), so as in the previous section, we do not make our estimates explicit, and we do not track the dependence on $n$ or the degree $d$.
	
	We begin with the following variant of Lemma \ref{lem:trace-zero-element}.
	
	\begin{lemma} \label{lem:skew-trace-zero}
		Let $k$ be a number field of degree $d$, let $K/k$ be an extension of degree $n$, and suppose that $\lambda_{\max}(K) > \lambda_{\max}(k)$.  Then there is a non-zero element $\alpha \in \mathcal{O}_K$ with $\mathrm{Tr}_{K/k}\alpha =0$ such that $\|\alpha \| \ll_{n,d} |\mathrm{Disc}(K)|^{\frac{1}{2(d(n-1)-1)}} |\mathrm{Disc}(k)|^{-\frac{1}{2(d(n-1)-1)}} \lambda_{\max}(K)^{-\frac{1}{d(n-1)-1}}$.
	\end{lemma}
	\begin{proof}
		Let $\mathcal{O}_K^0$ be the sublattice of $\mathcal{O}_K$ consisting of elements with trace $0$ to $k$.  By Lemma \ref{lem:trace-zero-minimum}, we have $\lambda_{\max}(\mathcal{O}_K^0) \geq \lambda_{\max}(\mathcal{O}_K)$.  By Minkowski's second theorem, we have
			\[
				\lambda_1(\mathcal{O}_K^0) \dots \lambda_{d(n-1)}(\mathcal{O}_K^0) \ll_{n,d} \mathrm{covol}_{W^\perp}(\mathcal{O}_K^0) \ll_{n,d} \sqrt{\frac{|\mathrm{Disc}(K)|}{|\mathrm{Disc}(k)|}},
			\]
		where the last inequality follows from Lemma \ref{lem:trace-zero-covolume}.  The left-hand side above is bounded below by $\lambda_1(\mathcal{O}_K^0)^{d(n-1)-1} \lambda_{\max}(K)$, so from this we find that
			\[
				\lambda_1(\mathcal{O}_K^0) \ll_{n,d} |\mathrm{Disc}(K)|^{\frac{1}{2(d(n-1)-1)}} |\mathrm{Disc}(k)|^{-\frac{1}{2(d(n-1)-1)}} \lambda_{\max}(K)^{-\frac{1}{d(n-1)-1}}.
			\]
		The element $\alpha \in \mathcal{O}_K^0$ with $\| \alpha \| = \lambda_1(\mathcal{O}_K^0)$ thus satisfies the conclusion of the lemma.
	\end{proof}
	
	We then have a simple modification of Theorem \ref{thm:uniform-schmidt}.
	
	\begin{lemma}\label{lem:skew-schmidt}
		Let $k$ be a number field of degree $d$ and let $G$ be a primitive permutation group of degree $n \geq 6$.  Then for any $X \geq |\mathrm{Disc}(k)|^n$ and any $\lambda \geq \lambda_{\max}(k)$, there holds
			\begin{align*}
				&\#\{ K \in \mathcal{F}_{n,k}(X;G) : \lambda_{\max}(K) > \lambda\} \\
				&	\quad \ll_{n,d} \left(\frac{X}{\lambda^2}\right)^{\frac{n+2}{4}\left(1+\frac{1}{d(n-1)-1}\right)} |\mathrm{Disc}(k)|^{-\frac{3n}{4}-\frac{n+2}{4(d(n-1)-1)}} \cdot \max\left\{ 1, |\mathrm{Disc}(k)| \left(\frac{\lambda^2|\mathrm{Disc}(k)|}{X}\right)^{\frac{d}{d(n-1)-1}} \right\}.
			\end{align*}
	\end{lemma}
	\begin{proof}
		We mimic the proof of Theorem \ref{thm:uniform-schmidt}, except for the element $\alpha$ provided by Lemma \ref{lem:skew-trace-zero}.  Let $H = \left(\frac{X}{\lambda^2 |\mathrm{Disc}(k)|}\right)^{\frac{1}{2(d(n-1)-1)}}$.  Since we have assumed that $X \geq |\mathrm{Disc}(k)|^n$, and we may assume that $\lambda \leq X^{\frac{1}{dn}}$ by Lemma \ref{lem:largest-minimum}, we find $H^i \geq \lambda_{\max}(k)$ for each $i \geq 3$.  When $i=2$, $\lambda = X^{\frac{1}{dn}}$, and $\lambda_{\max}(k) = |\mathrm{Disc}(k)|^{\frac{1}{d}}$, we only find that $H^2 \geq \lambda_{\max}(k) |\mathrm{Disc}(k)|^{-\frac{1}{d(dn-d-1)}}$.  Hence, increasing the side length of the box for $P_2 = \mathrm{Tr}_{K/k} \alpha^2$ by a factor $\max\{ 1, \frac{|\mathrm{Disc}(k)|^{\frac{1}{d}}}{H^2} \}$ so that we may apply Corollary \ref{cor:integers-bounded-height} but otherwise proceeding as in the proof of Theorem \ref{thm:uniform-schmidt}, we find
			\[
				\#\{ K \in \mathcal{F}_{n,k}(X;G) : \lambda_{\max}(K) > \lambda\}
					\ll_{n,d} H^{\frac{d(n+2)(n-1)}{2}} |\mathrm{Disc}(k)|^{-\frac{n-1}{2}} \cdot \max\left\{ 1, \frac{|\mathrm{Disc}(k)|}{H^{2d}} \right\}.
			\]
		Subsituting in the value of $H$, the lemma follows.
	\end{proof}
	\begin{remark}
		When $\lambda \gg_{n,d} X^{\frac{1}{2d(n-1)}} |\mathrm{Disc}(k)|^{-\frac{1}{2d(n-1)}}$, then we recover (a soft form of) Theorem \ref{thm:uniform-schmidt}, as expected by Lemma \ref{lem:lambda-lower-bound}.  However, when $\lambda \gg_{n,d} X^{\frac{1}{dn}}$, we find instead
			\[
				\#\{ K \in \mathcal{F}_{n,k}(X;G) : \lambda_{\max}(K) \gg_{n,d} X^{\frac{1}{dn}}\}
					\ll_{n,d} X^{\frac{n+2}{4} - \frac{n^2-4}{4n(dn-d-1)}} |\mathrm{Disc}(k)|^{-\frac{3n}{4}-\frac{n-2}{4(dn-d-1)}},
			\]
		representing a minor improvement.
	\end{remark}
	
	We next have the following generalization of \cite[Theorem 20]{Bhargava-vdW} that both accounts for the shape of the lattice as in Lemma \ref{lem:skew-schmidt} and applies over an arbitrary number field.

	\begin{theorem}\label{thm:skew-bhargava}
		Let $G$ be a transitive permutation group of degree $n \geq 6$, and suppose that $\mathrm{ind}(G) \geq 2$.  Let $k$ be a number field, let $X \geq |\mathrm{Disc}(k)|^n$, and let $\lambda \geq \lambda_{\max}(k)$ and $d=[k:\mathbb{Q}]$.  If $\mathrm{ind}(G)=2$, then
			\begin{align*}
				& \#\{ K \in \mathcal{F}_{n,k}(X;G) : \lambda_{\max}(K) > \lambda\} \\
					&\quad\quad\quad \ll_{n,d,\epsilon} \left(\frac{X}{\lambda^2}\right)^{\frac{d(n-2)(n+3)}{4(dn-d-1)}} X^{-\frac{1}{2}+\frac{1}{n-1}+\epsilon} |\mathrm{Disc}(k)|^{\frac{-n}{4}-\frac{(n-2)(n+3)}{4(n-1)(dn-d-1)}} \mathcal{E}_{n-1}
			\end{align*}
		where the quantity $\mathcal{E}_{n-1}$ is given explicitly in \eqref{eqn:En-1}.  If $\mathrm{ind}(G) \geq 3$, then we have
			\begin{align*}
			& \#\{ K \in \mathcal{F}_{n,k}(X;G) : \lambda_{\max}(K) > \lambda\} \\
					&\quad\quad\quad \ll_{n,d,\epsilon} \left(\frac{X}{\lambda^2}\right)^{\frac{d(n-2)(n+3)}{4(dn-d-1)}} X^{-1+\frac{1}{\mathrm{ind}(G)}+\frac{1}{n-1}+\epsilon} |\mathrm{Disc}(k)|^{\frac{n}{4}-\frac{n}{\mathrm{ind}(G)}-\frac{(n-2)(n+3)}{4(n-1)(dn-d-1)}}.
			\end{align*}
	\end{theorem}
	\begin{proof}
		We mimic the proof of Lemma \ref{lem:skew-schmidt}, but incorporating the congruence conditions on power sums arising from Lemmas \ref{lem:power-sum-congruence} and \ref{lem:some-power-sums}.  In particular, in the proof of Lemma \ref{lem:skew-schmidt}, the power sums $P_i := \mathrm{Tr}_{K/k} \alpha^i$ are integers in $\mathcal{O}_k$ of height $O_{n,d}(H^i)$, $H:= \left(\frac{X}{\lambda^2 |\mathrm{Disc}(k)|}\right)^{\frac{1}{2(d(n-1)-1)}}$, and are subject to congruence conditions $\pmod{\mathfrak{D}_i}$, where, as in the proof of Theorem \ref{thm:invariant-theory-power-sums}, 
			\[
				\mathfrak{D}_i := \prod_{\substack{\mathfrak{p} \mid \mathfrak{D}_{K/k}: \\ v_\mathfrak{p}(\mathfrak{D}_{K/k}) \geq n+1-i}} \mathfrak{p}.
			\]
		As in the proof of Theorem \ref{thm:invariant-theory-power-sums}, we have $\lambda_d(\mathfrak{D}_i) \ll_{n,d} X^{\frac{1}{d(n+1-i)}} |\mathrm{Disc}(k)|^{\frac{1}{d}-\frac{n}{d(n+1-i)}}$.  For each $3 \leq i \leq n-2$, we find that $H^i \gg_{n,d} \lambda_{\max}(\mathfrak{D}_i)$, since we may assume that $\lambda < X^{\frac{1}{dn}}$ by Lemma \ref{lem:largest-minimum}.  By our assumption that $\mathrm{ind}(G) \geq 2$, we have that $\mathfrak{D}_n = \mathfrak{D}_{n-1}$, which implies that $H^n \gg_{n,d} \lambda_{\max}(\mathfrak{D}_n)$.  Similarly, if $\mathrm{ind}(G) \geq 3$, then we also have $\mathfrak{D}_{n-1}=\mathfrak{D}_{n-2}$, which implies that $H^{n-1} \gg \lambda_{\max}(\mathfrak{D}_{n-1})$ in this case.  For the invariant $P_2$, and also $P_{n-1}$ if $\mathrm{ind}(G) = 2$, we increase the size of the box if necessary to realize the congruence conditions.  All told, this yields
			\begin{align*}
				& \#\{ K \in \mathcal{F}_{n,k}(X;G) : \lambda_{\max}(K) > \lambda\} \\
					&\quad \ll_{n,d,\epsilon} \left(\frac{X}{\lambda^2}\right)^{\frac{n+2}{4}\left(1+\frac{1}{d(n-1)-1}\right)} X^{-1+\frac{1}{\mathrm{ind}(G)}+\epsilon} |\mathrm{Disc}(k)|^{\frac{n}{4}-\frac{n}{\mathrm{ind}(G)} -\frac{n+2}{4(d(n-1)-1)}} \mathcal{E}_2 \mathcal{E}_{n-1},
			\end{align*}
		where 
			\[
				\mathcal{E}_2 := \max\left\{ 1, \frac{X^{\frac{1}{n-1}} |\mathrm{Disc}(k)|^{-\frac{1}{n-1}} }{H^{2d}}\right\},
			\]
		and $\mathcal{E}_{n-1} = \max\left\{ 1, \frac{|\mathrm{Disc}(k)|^{\frac{2-n}{2}} X^{\frac{1}{2}}}{H^{(n-1)d}}\right\}$ if $\mathrm{ind}(G)=2$ and $\mathcal{E}_{n-1}=1$ if $\mathrm{ind}(G) \geq 3$.  Finally, by way of dyadic summation and Lemma \ref{lem:lambda-lower-bound}, we may assume that $\lambda \gg_{n,d} X^{\frac{1}{2d(n-1)}} |\mathrm{Disc}(k)|^{-\frac{1}{2d(n-1)}}$.  So doing, we observe that $\mathcal{E}_2 \ll_{n,d} \frac{|\mathrm{Disc}(k)|^{-\frac{1}{n-1}} X^{\frac{1}{n-1}}}{H^{2d}}$.
		Recalling that $H = \left(\frac{X}{\lambda^2 |\mathrm{Disc}(k)|}\right)^{\frac{1}{2(d(n-1)-1)}}$, we see that the expression $\mathcal{E}_{n-1}$ is given by
			\begin{equation} \label{eqn:En-1}
				\mathcal{E}_{n-1}
					= \max\left\{ 1, \frac{ \lambda^{\frac{d(n-1)}{dn-d-1}} }{ X^{\frac{1}{2(dn-d-1)}} |\mathrm{Disc}(k)|^{\frac{dn^2-4dn+3d-n+2}{2(dn-d-1)}}} \right\} \text{ if $\mathrm{ind}(G)=2$}.
			\end{equation}
		This completes the proof.
	\end{proof}
	
	\begin{corollary}\label{cor:bhargava-number-field}
		Let $G$ be a transitive permutation group of degree $n \geq 6$ with $\mathrm{ind}(G) \geq 2$. Let $k$ be a number field, and set $d=[k:\mathbb{Q}]$.  For any $X \geq 1$, we have
			\[
				\#\mathcal{F}_{n,k}(X;G)
					\ll_{n,d,\epsilon}X^{\frac{n-2}{4} + \frac{1}{\mathrm{ind}(G)} + \epsilon} |\mathrm{Disc}(k)|^{\frac{n}{4}-\frac{n}{\mathrm{ind}(G)}} + |\mathrm{Disc}(k)|^{\frac{n^2-1}{2}}.
			\]
	\end{corollary}
	\begin{proof}
		If $K \in \mathcal{F}_{n,k}(X;G)$ is such that $\lambda_{\max}(K) \leq \lambda_{\max}(k)$, then Lemma \ref{lem:lambda-lower-bound} and Lemma \ref{lem:largest-minimum} together imply that $|\mathrm{Disc}(k)|^{\frac{1}{d}} \gg_{n,d} |\mathrm{Disc}(K)|^{\frac{1}{2d(n-1)}} |\mathrm{Disc}(k)|^{-\frac{1}{2d(n-1)}}$, and hence that $|\mathrm{Disc}(K)| \ll_{n,d} |\mathrm{Disc}(k)|^{2n-1}$.  The number of such fields $K$ may be bounded by Theorem \ref{thm:uniform-schmidt}, and is $\ll_{n,d} |\mathrm{Disc}(k)|^{\frac{n^2-1}{2}}$.  Hence, it suffices to bound those $K \in \mathcal{F}_{n,k}(X;G)$ with $\lambda_{\max}(K) > \lambda_{\max}(k)$, for which we exploit Theorem \ref{thm:skew-bhargava} with $\lambda \gg_{n,d} X^{\frac{1}{2d(n-1)}} |\mathrm{Disc}(k)|^{-\frac{1}{2d(n-1)}}$ and dyadic summation.  
		For this value of $\lambda$, we see that $\mathcal{E}_{n-1} \ll_{n,d} \max\{1,|\mathrm{Disc}(k)|^{\frac{3-n}{2}} \} \ll_{n,d} 1$.  Simplifying, we obtain the result.
	\end{proof}

\section{Algebraically independent invariants of almost simple groups}
\label{sec:invariant-theory}

	With the general bounds on number fields behind us, our goal in this section is to construct algebraically independent sets of invariants of small degree for essentially all almost simple groups $G$, in at least one primitive permutation representation.  Although the results of the previous section allow for the action of $G$ on $(\mathbb{A}^n)^r$, we focus here on the case $r=1$, i.e. in the action of $G$ on $\mathbb{A}^n$, and we refer to the invariants of $G$ in this action simply as $G$-invariants.  Moreover, we also restrict our attention to specific primitive representations of almost simple groups, namely representations we refer to as ``natural'' primitive representations:
	
	\begin{definition}\label{def:natural}
		Let $G$ be an almost simple group with socle $N$ (so $N \unlhd G \subseteq \mathrm{Aut}(N)$).  The \emph{natural} primitive representations of $G$ are:
			\begin{enumerate}[a)]
				\item if $N \simeq A_n$, with $G \subseteq S_6$ if $n=6$, the natural degree $n$ representation;
				\item if $N$ is classical, the action on $1$-dimensional singular subspaces of the natural module, that is:
					\begin{enumerate}[i)]
						\item if $N \simeq \mathrm{PSL}_m(\mathbb{F}_q)$ for some $m \geq 2$, with $G \subseteq \mathrm{ P \Gamma L}_m(\mathbb{F}_q)$ if $m \geq 3$, the representation of $G$ on the $1$-dimensional subspaces of $\mathbb{F}_q^m$;
						\item if $N \simeq \mathrm{PSp}_{2m}(\mathbb{F}_q)$ for some $m \geq 2$, with $G \subseteq \mathrm{P \Gamma L}_{2m}(\mathbb{F}_q)$ if $m=2$ and $q$ is even, the representation of $G$ on the $1$-dimensional subspaces of $\mathbb{F}_q^{2m}$;
						\item if $N \simeq \mathrm{PSU}_m(\mathbb{F}_q)$ for some $m \geq 3$, the representation of $G$ on the $1$-dimensional subspaces of $\mathbb{F}_{q^2}^m$ that are totally singular with respect to the defining unitary form on $\mathbb{F}_{q^2}^m$;
						\item if $N \simeq \mathrm{P\Omega}_{2m+1}(\mathbb{F}_q)$ with $q$ odd and $m \geq 3$, the representation of $G$ on the $1$-dimensional subspaces of $\mathbb{F}_q^{2m+1}$ that are totally singular with respect to the defining quadratic form on $\mathbb{F}_q^{2m+1}$;
						\item if $N \simeq \mathrm{P\Omega}^+_{2m}(\mathbb{F}_q)$ or $N \simeq \mathrm{P\Omega}^-_{2m}(\mathbb{F}_q)$ for some $m \geq 4$, with $G \subseteq \mathrm{P\Gamma L}_{2m}(\mathbb{F}_q)$ if $N \simeq \mathrm{P\Omega}^+_8(\mathbb{F}_q)$, the representation of $G$ on $1$-dimensional subspaces of $\mathbb{F}_q^{2m}$ that are totally singular with respect to the defining quadratic form on $\mathbb{F}_q^{2m}$;
					\end{enumerate}
				\item if $N \simeq \mathrm{PSU}_4(\mathbb{F}_q)$, the representation of $G$ on the $2$-dimensional totally singular subspaces of $\mathbb{F}_{q^2}^4$; and
				\item if $N$ is either exceptional and sporadic, any elemental primitive representation.
			\end{enumerate}
	\end{definition}

	Our motivation for this terminology is that the natural representations are typically, but not always, the primitive representations of minimal degree; see \cite[Theorem 5.2.2]{KleidmanLiebeck}, but noting the correction in \cite{VasilevMazurov} for groups of type $\mathrm{P \Omega}_{2m}^+(\mathbb{F}_3)$.  In particular, we have included the $2$-dimensional action of $\mathrm{PSU}_4(\mathbb{F}_q)$ because it is the minimal degree representation of such groups.  More practically, we will show in Section \ref{subsec:almost-simple-not-natural} that, for the purposes of bounding $\mathcal{F}_{n,k}(X;G)$ when $G$ is almost simple, it suffices to understand the problem for the natural representations.  Additionally, the natural representations of alternating and classical groups are ``standard,'' e.g. in the sense of \cite[Definition 1]{Burness-JLMS}.
	
	Note that some almost simple groups have more than one natural primitive representation (e.g., $A_6$ and $S_6$ have natural representations in degrees $6$ and $10$ by means of the exceptional isomorphism $A_6 \simeq \mathrm{PSL}_2(\mathbb{F}_9)$, and unitary groups of the form $\mathrm{PSU}_4(\mathbb{F}_q)$ have two natural geometric representations), while others have no natural primitive representation (e.g., the automorphism groups of $\mathrm{PSL}_n(\mathbb{F}_q)$ for $n \geq 3$, $\mathrm{P \Omega}_8^+(\mathbb{F}_q)$, and, for $q$ even, $\mathrm{PSp}_4(\mathbb{F}_q)$).  We discuss the non-natural primitive representations of almost simple groups in Section \ref{subsec:almost-simple-not-natural}.
	
	With this definition in hand, we now recall the statement of Theorem \ref{thm:minimal-invariants-intro}.

	\begin{theorem}\label{thm:invariant-size}
		There is an absolute constant $C$ such that for any almost simple group $G$ in a natural primitive representation, there is a set of $n$ algebraically independent $G$-invariants $\{f_1,\dots,f_n\}$ with $\max \{ \deg f_1, \dots, \deg f_n\} \leq C \frac{\log |G|}{\log n}$, where $n = \deg G$.
	\end{theorem}
	
	As indicated in the introduction, we expect Theorem \ref{thm:invariant-size} to be essentially optimal, apart from the value of the constant $C$.  
	If $G$ is alternating, then the claim is essentially trivial, as in this case $\log |G| / \log n \ll n$ and we may take the $f_i$ to be the standard elementary symmetric polynomials.  In the case that $G$ is exceptional or sporadic, we have that $\log |G| / \log n \ll 1$ \cite[Proposition 2]{LiebeckSaxl}, and Corollary \ref{cor:easy-invariants} below shows that there is always a set of invariants with degree at most $28$.  We expect that this value can be improved, and in the case of Mathieu groups, we have experimentally found a set of invariants of degree at most $10$ (see \S\ref{subsec:mathieu-invariants}).  Finally, for the classical groups, we construct two sets of invariants, the first of which works over all finite fields, the second only for linear and symplectic groups over finite fields that are sufficiently large relative to the rank of the group.  We expect, perhaps naively, that these second sets of invariants are close to optimal, at least as the size of the finite field tends to infinity while the rank of the group remains constant.  We suspect that substantial optimizations may be possible for smaller finite fields.
		
	We begin in Theorem \ref{thm:stabilizer-invariants} by providing a general way to find algebraically independent invariants of permutation groups (primitive or otherwise) that will already be enough to handle the case that $G$ is exceptional or sporadic.  For classical groups, the first set of invariants we provide follows directly from this theorem while the second requires a variant of the argument, stated generally in Theorem \ref{thm:stabilizer-invariants-variant}. We discuss the case of Mathieu groups in \S\ref{subsec:mathieu-invariants}.  Finally, we provide the proof of Theorem \ref{thm:invariant-size} in \S\ref{subsec:invariant-bound-proof}.

\subsection{Invariants associated with stabilizers}
	\label{subsec:stabilizer-invariants}
		
	Throughout this section, we fix a transitive permutation group $G$ of degree $n$, not necessarily primitive.  For any subset $\Sigma \subseteq \{1,\dots,n\}$, let $\mathrm{Stab}_G \Sigma$ be the stabilizer of $\Sigma$ as a set.  The general result is then the following.
	
	\begin{theorem}\label{thm:stabilizer-invariants}
		Let $G$ be a transitive permutation group of degree $n$, and suppose there are $s$ disjoint subsets $\Sigma_1,\dots,\Sigma_s$ for which $\mathrm{Stab}_G \Sigma_1 \cap \dots \cap \mathrm{Stab}_G \Sigma_s$ is trivial.  Let $t = |\Sigma_1| + \dots + |\Sigma_s|$ and $w = 1+\sum_{i=1}^s (s+2-i) |\Sigma_i|$.  Then there is an algebraically independent set of $G$-invariants $\{f_1,\dots,f_n\}$ with degrees satisfying:
			\begin{itemize}
				\item $\deg f_i = i$ for $1 \leq i \leq t$, and
				\item $\deg f_i = w$ for $t+1 \leq i \leq n$.
			\end{itemize}
	\end{theorem}
	\begin{remark}
		We have stated Theorem \ref{thm:stabilizer-invariants} (and also Theorem \ref{thm:stabilizer-invariants-variant} below) in a manner consistent with an explicit construction of invariants provided in the proof.  In practice, combining these theorems with Lemma \ref{lem:replace-by-power-sums} will yield invariants with slightly smaller degrees.
	\end{remark}
	
	Before presenting the proof, we provide two simple examples for imprimitive groups $G$.
	
	\begin{example}
		Suppose $G$ acts regularly, so that the stabilizer of any point is trivial.  The hypotheses of Theorem \ref{thm:stabilizer-invariants} are thus satisfied with $s=1$ and $\Sigma_1 = \{1\}$, and we find $t=1$ and $w=3$.  It then follows that there is a set of $n$ algebraically independent invariants with degrees $\deg f_1 = 1$ and $\deg f_i = 3$ for each $2 \leq i \leq n$.
	\end{example}
	
	\begin{example}
		Suppose $G = \mathrm{GL}_m(\mathbb{F}_q)$, acting in its degree $n=q^m-1$ representation on $\mathbb{F}_q^m \setminus \{0\}$, and suppose that $q\geq m$ and $G \ne \mathrm{GL}_2(\mathbb{F}_2)$.  Let $v \in \mathbb{F}_q^m$ be a vector whose coordinates are distinct elements of $\mathbb{F}_q$, non-zero if $m=2$, let $\Sigma_1 = \{ v \}$, and let $\Sigma_2$ be the standard basis for $\mathbb{F}_q^m$.  Then $\mathrm{Stab}_G \Sigma_2$ consists of the set of permutation matrices over $\mathbb{F}_q$, from which follows that $\mathrm{Stab}_G \Sigma_1 \cap \mathrm{Stab}_G \Sigma_2$ is trivial.  Thus, the hypotheses of the theorem are satisfied with $t=m+1$ and $w = 2m + 4$.  It therefore follows that there is an independent set of invariants with $\deg f_i = i$ for $i \leq m+1$ and $\deg f_i = 2m+4$ for $m+2 \leq i \leq n$.
	\end{example}
	
	\begin{proof}[Proof of Theorem \ref{thm:stabilizer-invariants}]
	
		Let $\Sigma_1,\dots, \Sigma_s$ be as in the statement of Theorem \ref{thm:stabilizer-invariants}.  After reordering the set $\{1,\dots,n\}$ if necessary, we may assume that $\Sigma_1 = \{1,\dots,i_1\}$ for some $i_1$, $\Sigma_2 = \{i_1+1,\dots,i_2\}$ for some $i_2>i_1$, etc.; thus, $\Sigma_1 \cup \dots \cup \Sigma_s = \{1,\dots,i_s\}$, with $i_s = t$.
		
		Now, let $G$ act on $\mathbb{Z}[x_1,\dots,x_n]$ in the natural way, and let $\mathbb{Z}[x_1,\dots,x_n]^G$ denote the set of $G$-invariants.  Given $f \in \mathbb{Z}[x_1,\dots,x_n]$, we let $f^G := \sum_{g \in G} f^g$, which is necessarily $G$-invariant; for convenience, when $f$ is a monomial, we renormalize $f^G$ so that it is monic.  We construct the set of $n$ invariant polynomials $f_i$ associated with the sets $\Sigma_1,\dots,\Sigma_s$ explicitly as follows.
			\begin{itemize}
				\item Let $f_1 := x_1^G = x_1 + \dots + x_n$.
				\item More generally, for $2 \leq i \leq t$, let $f_i := (x_1 \dots x_i)^{G}$.
				\item For $i > t$, let $f_i := (P_0 x_i)^G$, where 
					\begin{equation} \label{eqn:P0}
						P_0 := \prod_{j=1}^s \prod_{i \in \Sigma_j} x_i^{s+2-j}.
					\end{equation}
			\end{itemize}
		(For example, if $G$ acts regularly, then $P_0 = x_1^2$ and the set of invariant polynomials will be $\{(x_1)^G, (x_1^2x_2)^G, \dots, (x_1^2x_n)^G\}.$)
		
		By Lemma \ref{lem:independence-criterion}, to show the set $\{f_1,\dots,f_n\}$ is algebraically independent, it suffices to show that the $n\times n$ matrix $D:=\mathbf{D}(f_1,\dots,f_n) := \left( \frac{\partial f_i}{\partial x_j}\right)_{1\leq i,j\leq n}$ has non-zero determinant.  To do so, we regard the rows of $D$ as being indexed by the $f_i$ and the columns by the $x_j$.  We have
			\begin{equation} \label{eqn:determinant-expansion}
				\det D = \sum_{\sigma \in S_n} \mathrm{sgn}(\sigma) \prod_{i=1}^n \frac{ \partial f_i}{\partial x_{\sigma(i)}}.
			\end{equation}
		We now claim that there is an ordering on the monomials of $\mathbb{Z}[x_1,\dots,x_n]$ such that there is a unique permutation $\sigma$ (in fact, the identity) such that the corresponding term in \eqref{eqn:determinant-expansion} has a maximal monomial with respect to this ordering among all those appearing for any $\sigma$.  Consequently, this monomial cannot be canceled by any other, and $\det D \ne 0$.
		
		The ordering we use is by degree in $x_1$, then by degree in $x_2$, and so on.  For the polynomials $f_i$ with $2 \leq i \leq t$, the partial derivative $\frac{\partial f_i}{\partial x_i}$ has the monomial $x_1 \dots x_{i-1}$ that is larger with respect to this ordering than any monomial appearing in $\frac{\partial f_i}{\partial x_j}$ for $j < i$.  There may be $j > i$ for which the partial derivative $\frac{\partial f_i}{\partial x_j}$ also possesses the monomial $x_1\dots x_{i-1}$, but we note that these partial derivatives lie in the main upper triangular portion of the matrix $D$.
		
		For the polynomials $f_i$ with $i > t$, the partial derivative $\frac{\partial f_i}{\partial x_i}$ has the monomial $P_0$ as defined in \eqref{eqn:P0}.  Any partial derivative with respect to $x_j$ for $j \leq t$ will have smaller total degree in $x_1,\dots,x_{t}$, and thus contains only monomials smaller than $P_0$.  Additionally, no partial derivative with respect to $x_j$ for $j > t$, $j \ne i$, may contain the monomial $P_0$: for this to happen, there must be some $g \in G$ such that $(P_0 x_i)^g = P_0 x_j$, from which follows that $P_0^g = P_0$.  However, by construction, $\mathrm{Stab}_G P_0 = \mathrm{Stab}_G \Sigma_1 \cap \dots \cap \mathrm{Stab}_G \Sigma_s$, which is trivial.  Moreover, under this ordering, $P_0$ is greater than any of its nontrivial permutations under the action of $\mathrm{Sym}\{1,\dots,t\}$, and it follows that the monomial $P_0$ appearing in $\frac{\partial f_i}{\partial x_i}$ is the greatest of any in its row.  It follows that the term in \eqref{eqn:determinant-expansion} corresponding to $\sigma = \mathrm{id}$ has a maximal monomial, that no other $\sigma$ has this monomial, and that $\det D \ne 0$.
	\end{proof}
	
	We now recall the notion of a base of a permutation group $G$ from Definition \ref{def:base}, and, following that definition, we denote by $\mathrm{base}(G)$ the minimal order of a base.  Combined with Theorem \ref{thm:stabilizer-invariants}, we therefore deduce the simple corollary:
	
	\begin{corollary}\label{cor:base-invariants}
		Given any permutation group $G$ of degree $n$, there is a set $\{f_1,\dots,f_n\}$ of algebraically independent $G$-invariants with $\max\{\deg f_1, \dots, \deg f_n\} \leq (\mathrm{base}(G)+1)(\mathrm{base}(G)+2)/2$.
	\end{corollary}
	\begin{proof}
		Take the $\Sigma_i$ in Theorem \ref{thm:stabilizer-invariants} to be the elements of a base with order $\mathrm{base}(G)$.
	\end{proof}
	
	Combining Corollary \ref{cor:base-invariants} with Theorem \ref{thm:pyber's-conjecture}, we see for any primitive group $G$ that there is always a full set of algebraically independent invariants with maximum degree being $O( ( \log |G| / \log n)^2)$.  This falls a little short of obtaining Theorem \ref{thm:invariant-size}, but it does allow us to restrict our attention to classical groups in their natural subspace actions, as these are the groups for which the ratio $\log |G| / \log n$ may be arbitrarily large.  We make this precise in the following corollary that follows from Corollary \ref{cor:base-invariants} and known results on $\mathrm{base}(G)$.  
		
	\begin{corollary}\label{cor:easy-invariants}
		Suppose $G$ is an almost simple group of either exceptional or sporadic type, in a primitive representation of degree $n$.  Then there is a set $\{f_1,\dots,f_n\}$ of algebraically independent $G$-invariants satisfying $\max\{\deg f_1, \dots, \deg f_n\} \leq 28$.
	\end{corollary}
	\begin{proof}
		From \cite{Base-Exceptional}, we have that $\mathrm{base}(G) \leq 6$ for any exceptional group $G$.  Thus, the claim in this case follows from Corollary \ref{cor:base-invariants}.  From \cite{Base-Sporadic}, we also have the bound $\mathrm{base}(G) \leq 6$ for every sporadic group $G$ except for the Mathieu group $M_{24}$ of degree $24$, which satisfies $\mathrm{base}(G)=7$.  However, in this case, the elementary symmetric polynomials of degree at most $24$ satisfy the conclusion of the corollary.
	\end{proof}
	
	For convenience, we also record a fact that will be useful in the proof of Theorem \ref{thm:solvable-intro}.
	
	\begin{corollary}\label{cor:solvable-invariants}
		If $G$ is a primitive permutation group of degree $n$ that is solvable, then $\mathrm{base}(G) \leq 4$ and there is a set $\{f_1,\dots,f_n\}$ of algebraically independent $G$-invariants satisfying $\max\{\deg f_1, \dots, \deg f_n\} \leq 15$.
	\end{corollary}
	\begin{proof}
		The bound $\mathrm{base}(G) \leq 4$ is from \cite{Base-Solvable}, and implies the remainder of the claim.
	\end{proof}
	
	Finally, we provide a variant of Theorem \ref{thm:stabilizer-invariants} that requires slightly stronger hypotheses but thereby yields a stronger result.  We will particularly take advantage of this for linear and symplectic groups whose underlying natural module $\mathbb{F}_q^m$ has $q$ large in terms of $m$.
	
	\begin{theorem}\label{thm:stabilizer-invariants-variant}
		As in Theorem \ref{thm:stabilizer-invariants}, let $G$ be a finite permutation group of degree $n$, and suppose there are $s$ disjoint subsets $\Sigma_1,\dots,\Sigma_s$ for which $\mathrm{Stab}_G \Sigma_1 \cap \dots \cap \mathrm{Stab}_G \Sigma_s$ is trivial.  Let $t = |\Sigma_1| + \dots + |\Sigma_s|$ and $w^\prime = 1+\sum_{i=1}^s (s+1-i) |\Sigma_i|$.  (Note that $w^\prime = w - t$, where $w$ is as in Theorem \ref{thm:stabilizer-invariants}.)  
		Additionally suppose that either:
			\begin{enumerate}[a)]
				\item for every subset $\Sigma \subseteq \Sigma_s$ of order $|\Sigma_s|-1$, $\mathrm{Stab}_G \Sigma_1 \cap \dots \cap \mathrm{Stab}_G \Sigma_{s-1} \cap \mathrm{Stab}_G \Sigma$ is trivial, in which case set $e:=|\Sigma_s|^2-|\Sigma_s|$; or
				\item for every subset $\Sigma \subseteq \Sigma_s$ of order $|\Sigma_s|-1$, $\mathrm{Stab}_G \Sigma_1 \cap \dots \cap \mathrm{Stab}_G \Sigma_{s-1} \cap \bigcap_{i \in \Sigma} \mathrm{Stab}_G \{i\}$ is trivial, in which case set $e:=|\Sigma_s|\cdot |\Sigma_s|!$.
			\end{enumerate}
		Then there is an algebraically independent set of $G$-invariants $\{f_1,\dots,f_n\}$ with degrees satisfying:
			\begin{itemize}
				\item $\deg f_i = i$ for $1 \leq i \leq t$,
				\item $\deg f_i = w^\prime$ for $t+1 \leq i \leq n - e$, and
				\item $\deg f_i \leq w^\prime + t$ for the remaining $n - e + 1 \leq i \leq n$.
			\end{itemize}
	\end{theorem}
	\begin{proof}
		We adapt the proof of Theorem \ref{thm:stabilizer-invariants}.  As in that proof, we relabel to assume that $\Sigma_1 = \{1,\dots,i_1\}$, $\Sigma_2 = \{i_1+1,\dots,i_2\}$, etc.  We again define $f_i = (x_1\dots x_i)^G$ for $i \leq t$, and we let $P_0 = \prod_{j=1}^s \prod_{i \in \Sigma_j} x_i^{s+2-j}$ be as in the proof of Theorem \ref{thm:stabilizer-invariants}.  Departing from that proof, we now define
			\[
				P_1 = \prod_{j=1}^s \prod_{i \in \Sigma_j} x_i^{s+1-j} = \frac{P_0}{x_1\dots x_t},
			\]
		and consider the invariants $f_i = (P_1 x_i)^G$ for $t +1 \leq i \leq n$.  The partial derivative $\frac{\partial f_i}{\partial x_i}$ contains the monomial $P_1$, which is necessarily maximal among monomials appearing in partial derivatives $\frac{\partial f_i}{\partial x_j}$.  However, it is no longer guaranteed that the partial derivative $\frac{\partial f_i}{\partial x_i}$ is the only partial derivative containing $P_1$.  For those $i$ for which it is not the only partial derivative containing $P_1$, we replace the invariant $f_i$ by the invariant $(P_0 x_i)^G$ from Theorem \ref{thm:stabilizer-invariants}.  Thus, after relabeling, our goal is to show that we must make this replacement at most $e$ times.
		
		Suppose that $\frac{\partial f_i}{\partial x_j}$ contains the monomial $P_1$ for some $j \ne i$.  Observe that we must have $j \geq t+1$ and that $(P_1 x_i)^G = (P_1 x_j)^G$ in this case.  In particular, there must be some nontrivial $g \in G$ such that $(P_1 x_i)^g = P_1 x_j$.  By degree considerations, we must have that $g \in \mathrm{Stab}_G \Sigma_1 \cap \dots \cap \mathrm{Stab}_G \Sigma_{s-1}$.  Next, since $\mathrm{Stab}_G P_1$ is trivial, we cannot have $P_1^g = P_1$, hence it must be that $g(i) \in \Sigma_s$ and $g^{-1}( \Sigma_s \setminus \{g(i)\}) \subseteq \Sigma_s$.  
		
		We now split based on whether we are in case a) or b).  In case a), we now claim that for any $a,b \in \Sigma_s$, one may recover $g$ and $i$ from the assumption that $g^{-1}( \Sigma_s \setminus \{b\}) = \Sigma_s \setminus \{a\}$ and $g(i) = b$.  It evidently suffices to recover $g$, since $i = g^{-1}(b)$.  Thus, suppose $h \in \mathrm{Stab}_G \Sigma_1 \cap \dots \cap \mathrm{Stab}_G \Sigma_{s-1}$ is such that $h^{-1}( \Sigma_s \setminus \{b\}) = \Sigma_s \setminus \{a\}$.  It follows that the element $h^{-1} g$ preserves the set $\Sigma_s \setminus \{a\}$, hence must be trivial by the assumptions of the theorem.  In particular, we must have $h=g$.  Moreover, if $a=b$, then $g$ itself preserves the set $\Sigma_s \setminus \{a\}$, and hence must be trivial.  It follows that there are at most $|\Sigma_s|(|\Sigma_s|-1)$ possible $i$'s for which we need to replace $f_i$ by $(P_0 x_i)^G$ as indicated above.
		
		In case b), we proceed essentially as in case a), except noting that if $g$ and $h$ define the same map from $\Sigma_s \setminus \{a\}$ to $\Sigma_s \setminus \{b\}$, then our hypotheses imply that $gh^{-1}$ must be trivial.  There are $|\Sigma_s| \cdot |\Sigma_s|!$ choices for $a$, $b$, and a map $\Sigma_s \setminus \{a\} \to \Sigma_s \setminus \{b\}$, and this yields the result.
	\end{proof}
	
	Lastly, while have made use of results on base sizes of primitive permutation groups, the constructions used in Theorems \ref{thm:stabilizer-invariants} and \ref{thm:stabilizer-invariants-variant} are perhaps more closely linked to the notion of the ``distinguishing number'' of a permutation group $G$, denoted $d(G)$, which is defined to be the smallest number $d$ of parts $\Sigma_1,\dots,\Sigma_{d}$ in a partition of  $\{1,\dots,n\}$ such that $\mathrm{Stab}_G \Sigma_1 \cap \dots \cap \mathrm{Stab}_G \Sigma_d$ is trivial.  In particular, for any sets $\Sigma_1, \dots, \Sigma_s$ for which either Theorem \ref{thm:stabilizer-invariants} or Theorem \ref{thm:stabilizer-invariants-variant} applies, we must have $s \geq d(G)-1$ (as follows on observing that $\Sigma_1,\dots,\Sigma_s,\{1,\dots,n\}\setminus (\Sigma_1 \cup \dots \cup \Sigma_s)$ is a distinguishing partition as above).  The distinguishing number of a group is often studied alongside its base (see, e.g., \cite[Theorem 1.2]{DuyanHalasiMaroti-Pyber} in the paper resolving Pyber's conjecture).  The parameter $t$ (and also the parameters $w$ and $w^\prime$) can be related to distinguishing partitions as well -- we have $t = n - \max\{ |\Sigma_1|,\dots,|\Sigma_d|\}$ -- but does not appear to have been studied in the literature.  We thank Tim Burness for raising this connection to our attention.
	
\subsection{Invariants associated with classical groups}
	\label{subsec:classical-invariants}
	
	In this section, we construct invariants of classical groups $G$ in their natural actions.  That is, we consider here linear groups (i.e., those almost simple groups with socle $\mathrm{PSL}_m(\mathbb{F}_q)$, for any $m \geq 2$ and any prime power $q$), symplectic groups (i.e, those with socle $\mathrm{PSp}_{2m}(\mathbb{F}_q)$ for $m \geq 2$), unitary groups (i.e, those with socle $\mathrm{PSU}_m(\mathbb{F}_q)$ for $m \geq 3$), and orthogonal groups (i.e., those with socle $\mathrm{P\Omega}_{2m}^+(\mathbb{F}_q)$ or $\mathrm{P\Omega}_{2m}^-(\mathbb{F}_q)$ for $m \geq 4$, or $\mathrm{P\Omega}_{2m+1}(\mathbb{F}_q)$ for $q$ odd with $m \geq 3$), following the notation and conventions from \cite{KleidmanLiebeck}.  The mechanics of how we treat each of these types is largely the same, but with enough specific differences that we treat each of these types separately.  For all types, we construct a set of invariants using Theorem \ref{thm:stabilizer-invariants} that will yield an efficient set of invariants for groups over every finite field $\mathbb{F}_q$.  For linear and symplectic groups, we also construct a set of invariants using Theorem \ref{thm:stabilizer-invariants-variant} that will yield improvements for $q$ sufficiently large.  Also, as suggested in the remark following Theorem \ref{thm:stabilizer-invariants}, we invoke Lemma \ref{lem:replace-by-power-sums} to leverage a small improvement.  We now proceed according to the specific socle type of the group $G$.
	
\subsubsection{Linear groups}

	We begin with a straightforward application of Theorem \ref{thm:stabilizer-invariants} and Lemma \ref{lem:replace-by-power-sums}.
	
	\begin{lemma}\label{lem:linear-invariants-all}
		Let $G$ be a transitive subgroup of $\mathrm{P\Gamma L}_m(\mathbb{F}_q)$ viewed as a permutation group of degree $n = (q^m-1)/(q-1)$ via its action on $\mathbb{P}^{m-1}(\mathbb{F}_q)$, i.e. in its action on $1$-dimensional subspaces of $\mathbb{F}_q^m$.  Then there is a set $\{f_1,\dots,f_n\}$ of algebraically independent $G$-invariants with degrees:
			\begin{itemize}
				\item $\deg f_i = i$, for $1 \leq i \leq 5m+5$, and
				\item $\deg f_i \leq 5m+5$ for $5m+6 \leq i \leq n$.
			\end{itemize}
		Moreover, if $G \subseteq \mathrm{PGL}_m(\mathbb{F}_q)$, then the invariants $f_{5m+1},\dots,f_n$ may be replaced by invariants of degree at most $5m$.
	\end{lemma}
	\begin{proof}
		Let $e_1,\dots,e_m$ denote a basis for $\mathbb{F}_q^m$.  Additionally, if $\mathbb{F}_q$ is a nontrivial extension of its prime subfield $\mathbb{F}_p$, let $\alpha$ denote a primitive element of $\mathbb{F}_q$ for which $\sigma(\alpha) \ne -\alpha$ for each nontrivial $\sigma \in \mathrm{Gal}(\mathbb{F}_q/\mathbb{F}_p)$.  Note that such an $\alpha$ always exists if $\mathbb{F}_q/\mathbb{F}_p$ is nontrivial, and in the case that $\mathbb{F}_q = \mathbb{F}_p$, we set $\alpha=1$.  For any $v \in \mathbb{F}_q^m \setminus \{0\}$, let $\overline{v}$ denote its image in $\mathbb{P}^{m-1}(\mathbb{F}_q)$.  Let $\Sigma_1 = \{\overline{e_1}\}$, $\Sigma_2 = \{\overline{e_1 + \frac{1}{\alpha} e_2}\}$, $\Sigma_3 = \{\overline{e_2},\dots,\overline{e_m}\}$, and $\Sigma_4 = \{\overline{e_1 + \alpha e_2}, \overline{e_2 + \alpha e_3}, \dots, \overline{e_{m-1} + \alpha e_m}\}$.  The lemma will thus follow from Theorem \ref{thm:stabilizer-invariants} and Lemma \ref{lem:replace-by-power-sums} provided we show that $\cap_{i=1}^4 \mathrm{Stab}_G \Sigma_i$ is trivial.
		
		Passing to the non-projective group $\mathrm{\Gamma L}_m(\mathbb{F}_q)$, 
		any $\gamma \in \mathrm{\Gamma L}_m(\mathbb{F}_q)$ may be written as $\gamma = g \sigma$, where $g \in \mathrm{GL}_m(\mathbb{F}_q)$ and $\sigma \in \mathrm{Gal}(\mathbb{F}_q/\mathbb{F}_p)$ acts on the coordinates with respect to the basis $\{e_1,\dots,e_m\}$.
		We therefore wish to show that if $\gamma$ stabilizes $\Sigma_1,\dots,\Sigma_4$, then $\sigma$ must be trivial and $g$ must be a scalar matrix.
				  
		If $\gamma$ stabilizes $\Sigma_1$ and $\Sigma_3$, then $g$ must be a monomial matrix with respect to the basis $\{e_1,\dots,e_m\}$.  It next follows that if $\gamma$ also stabilizes $\Sigma_4$, the matrix $g$ must be diagonal; in fact, $g$ must be a scalar matrix times $\mathrm{diag}(1, \frac{\alpha}{\sigma(\alpha)}, \dots, (\frac{\alpha}{\sigma(\alpha)})^{m-1})$.  Finally, if $\gamma$ stabilizes $\Sigma_2$, by the assumption that $\sigma(\alpha) \ne -\alpha$ for each nontrivial $\sigma$, we conclude that $\sigma$ is trivial and $g$ is scalar.  If $G \subseteq \mathrm{PGL}_m(\mathbb{F}_q)$, then the set $\Sigma_2$ may be omitted, and this yields the second claim.
	\end{proof}
	
	We next show that if $q$ is sufficiently large compared to $m$, then there is a set of $G$-invariants with smaller degree.
	
	\begin{lemma}\label{lem:linear-invariants-large}
		Let $G$ be a transitive subgroup of $\mathrm{P\Gamma L}_m(\mathbb{F}_q)$ in its degree $n = (q^m-1)/(q-1)$ action on $\mathbb{P}^{m-1}(\mathbb{F}_q)$.  Assume that $\binom{q-1}{m}(1-2q^{-\frac{m-1}{2}}) > (m+1 - \frac{1}{m!})(q^{m-1}+q-2) [ \mathbb{F}_q : \mathbb{F}_p]$ (so in particular also that $\frac{q}{[\mathbb{F}_q:\mathbb{F}_p]} \geq (m+1)!$), where $\mathbb{F}_p$ denotes the prime subfield of $\mathbb{F}_q$.  Then there is a set of $n$ algebraically independent $G$-invariants $\{f_1,\dots,f_n\}$ with degrees:
			\begin{itemize}
				\item $\deg f_i = i$ for $1 \leq i \leq m+2$,
				\item $\deg f_i = m+4$ for $m+3 \leq i \leq n - (m+1)\cdot(m+1)!$, and
				\item $\deg f_i \leq 2m+6$ for $n - (m+1)\cdot (m+1)! +1 \leq i \leq n$.
			\end{itemize}
	\end{lemma}
	\begin{proof}
		We ultimately appeal to Theorem \ref{thm:stabilizer-invariants-variant}, but we begin with some auxiliary considerations.  
		As in the proof of Lemma \ref{lem:linear-invariants-all}, any element $\gamma \in \mathrm{\Gamma L}_m(\mathbb{F}_q)$ may be expressed in the form $\gamma = g \sigma$, where $g \in \mathrm{GL}_m(\mathbb{F}_q)$ and $\sigma \in \mathrm{Gal}(\mathbb{F}_q/\mathbb{F}_p)$, with $\sigma$ acting naturally on coordinates with respect to a choice of basis $\{e_1,\dots,e_m\}$.  
		Let $\Sigma_2 = \{ \overline{e_1},\dots, \overline{e_m}, \overline{e_1+\dots+e_m}\}$.  Observe that the set $\Sigma_2$ is in general position, i.e. for any subset $\Sigma \subseteq \Sigma_2$ of order $m$, the minimal subspace of $\mathbb{F}_q^m$ containing each line in $\Sigma$ is just $\mathbb{F}_q^m$.  Additionally observe that $\Sigma_2$ is stabilized by $\mathrm{Gal}(\mathbb{F}_q/\mathbb{F}_p)$.  We now claim it is possible to choose $v_0 \in \mathbb{F}_q^m$ so that: 
			\begin{enumerate}[i.]
				\item for each $\sigma \in \mathrm{Gal}(\mathbb{F}_q/\mathbb{F}_p)$, $\mathrm{Stab}_{\mathrm{GL}_m(\mathbb{F}_q)} \overline{\sigma v_0} \cap \mathrm{Stab}_{\mathrm{GL}_m(\mathbb{F}_q)} \Sigma_2$ consists only of scalar matrices; 
				\item the set $\Sigma_2 \cup \{ \overline{\sigma v_0}\}$ is in general position for each $\sigma \in \mathrm{Gal}(\mathbb{F}_q/\mathbb{F}_p)$; and 
				\item the lines $\overline{\sigma v_0}$ for $\sigma \in \mathrm{Gal}(\mathbb{F}_q/\mathbb{F}_p)$ are all distinct.
			\end{enumerate}
		If so, we claim that if we let $\Sigma_1 = \{\overline{v_0}\}$, then $\Sigma_1$, $\Sigma_2$ satisfy the conditions of case \emph{b)} of Theorem \ref{thm:stabilizer-invariants-variant}.  In particular, condition i. implies that any element $\gamma \in \mathrm{Stab}_{\mathrm{\Gamma L}_m(\mathbb{F}_q)} \Sigma_1 \cap \mathrm{Stab}_{\mathrm{\Gamma L}_m(\mathbb{F}_q)} \Sigma_2$, when expressed in the form $\gamma = g \sigma$, must have $g$ a scalar matrix.  Condition iii. then implies that we must additionally have $\sigma$ trivial, so $\mathrm{Stab}_G \Sigma_1 \cap \mathrm{Stab}_G \Sigma_2$ must be trivial.  Similarly, since any element of $\mathrm{GL}_m(\mathbb{F}_q)$ fixing $m+1$ lines in general position must be scalar, we see that conditions ii. and iii. together imply that the condition in case \emph{b)} holds.  Thus, it only remains to show that such a $v_0$ exists.
		
		The condition that $\Sigma_2 \cup \{\overline{\sigma v_0}\}$ is in general position for each $\sigma$ is equivalent to demanding that the coordinates of $v_0$ with respect to the basis $e_1,\dots,e_m$ are all distinct elements of $\mathbb{F}_q^\times$.  
		The condition that the lines $\overline{\sigma v_0}$ are all distinct, meanwhile, is equivalent to asserting that the line $\overline{v_0}$ is not the extension to $\mathbb{F}_q$ of a line defined over a proper subfield of $\mathbb{F}_q$.  If we write $q=p^f$, then it follows that there are
			\begin{align*}
				\sum_{d \mid f} \mu(f/d) \frac{q-1}{p^{d}-1} (p^d-1)\dots(p^d-m)
					& \geq (q-1)\dots(q-m)\left( 1 - \sum_{d \leq f/2}\frac{(p^d-2)\dots(p^d-m)}{(q-2)\dots(q-m)} \right) \\
					&\geq (q-1)\dots(q-m) \left( 1 - \frac{2}{q^{\frac{m-1}{2}}} \right)
			\end{align*}
		vectors $v_0 \in \mathbb{F}_q^m$ satisfying conditions ii. and iii. above.
		
		On the other hand, since any element of $\mathrm{PGL}_m(\mathbb{F}_q)$ is determined by its action on any set of $m+1$ lines in general position, we see that $\mathrm{Stab}_{\mathrm{PGL}_m(\mathbb{F}_q)} \Sigma_2$ consists of at most $(m+1)!$ elements.  Any nontrivial element of $\mathrm{PGL}_m(\mathbb{F}_q)$ can stabilize the lines generated by at most $(q^{m-1}-1) + (q-1)$ vectors in $\mathbb{F}_q^m$, corresponding to the situation that there are eigenspaces of dimension $m-1$ and $1$.  
		It follows that there are at most $[\mathbb{F}_q:\mathbb{F}_p]((m+1)!-1)(q^{m-1}+q-2)$ vectors $v_0 \in \mathbb{F}_q^m$ that do not satisfy condition i. above.
		Therefore, if $(q-1)\dots(q-m) \left( 1 - \frac{2}{q^{\frac{m-1}{2}}} \right) > [\mathbb{F}_q:\mathbb{F}_p]((m+1)!-1)(q^{m-1}+q-2)$, then there must be at least one vector $v_0$ satisfying all conditions. Simplifying slightly, we obtain the conclusion of the lemma.
	\end{proof}
	
	\begin{remark}
		Simple computations show that Lemma \ref{lem:linear-invariants-large} applies for primes powers $q$ satisfying $\frac{q}{[\mathbb{F}_q:\mathbb{F}_p]} \geq 23$ if $m=2$, $\frac{q}{[\mathbb{F}_q:\mathbb{F}_p]} \geq 31$ if $m=3$, $\frac{q}{[\mathbb{F}_q:\mathbb{F}_p]} \geq 121$ if $m=4$, and $\frac{q}{[\mathbb{F}_q:\mathbb{F}_p]} \geq 739$ if $m=5$.  
		Additionally, computations in Magma \cite{Magma} analogous to those described below in Section \ref{subsec:mathieu-invariants} (the code for which is available at \url{http://lemkeoliver.github.io/}) suggest that for the group $\mathrm{P\Gamma L}_2(\mathbb{F}_q)$, $q \ne 8,16,32$, there may always be an independent set of $q+1$ invariants with degree bounded by $6$, and for $\mathrm{P\Gamma L}_3(\mathbb{F}_q)$, there may always be an independent set of $q^2+q+1$ invariants with degree bounded by $7$.  This essentially matches the conclusion of Lemma \ref{lem:linear-invariants-large}, but without any exceptional invariants of larger degree.  
		As a result, we (perhaps naively) expect the bound $m+4$ in Lemma \ref{lem:linear-invariants-large} to be at least close to optimal in general.  
	\end{remark}

\subsubsection{Symplectic groups}

	We now turn to symplectic groups $G$, whose socle we denote $\mathrm{PSp}_{2m}(\mathbb{F}_q)$ for some $m \geq 2$ and prime power $q$.  
	The natural primitive action of $G$ is then again on $\mathbb{P}^{2m-1}(\mathbb{F}_q)$, as the restriction of the symplectic form to any one-dimensional subspace is totally singular.  Thus, such groups are subject to Lemmas \ref{lem:linear-invariants-all} and \ref{lem:linear-invariants-large}, but the following lemma shows we may do better than Lemma \ref{lem:linear-invariants-all} by exploiting the symplectic form.  

	\begin{lemma}\label{lem:symplectic-invariants-all}
		Let $G$ be an almost simple group of type $\mathrm{PSp}_{2m}(\mathbb{F}_q)$ in its degree $n = (q^{2m}-1)/(q-1)$ action on $\mathbb{P}^{2m-1}(\mathbb{F}_q)$.  Then there is a set of $n$ algebraically independent $G$-invariants $\{f_1,\dots,f_n\}$ satisfying:
			\begin{itemize}
				\item $\deg f_i = i$ for $1 \leq i \leq 7m+5$, and
				\item $\deg f_i \leq 7m+5$ for $ 7m+6 \leq i \leq n$.
			\end{itemize}
	\end{lemma}
	\begin{proof}
		We work entirely with stabilizers, the conclusion then being furnished by Theorem \ref{thm:stabilizer-invariants} and Lemma \ref{lem:replace-by-power-sums}.  There is a basis $\{e_1,\dots,e_m,f_1,\dots,f_m\}$ for $\mathbb{F}_{q}^{2m}$ such that $\langle e_i,e_j\rangle = 0 = \langle f_i,f_j \rangle$ for any $i,j$ and $\langle e_i,f_j \rangle = \delta_{ij}$, where $\langle \cdot,\cdot\rangle$ denotes the symplectic pairing and $\delta_{ij}$ is the Kronecker delta \cite[Proposition 2.4.1]{KleidmanLiebeck}.\footnote{As this proof concerns only stabilizers and no explicit construction of invariants, we prefer to keep the notational conventions of \cite{KleidmanLiebeck} that some of the basis vectors are denoted $f_i$.  This notation of course overlaps with our choice for the set of invariants, but we hope that this overlap causes minimal confusion.  To be explicit, within the scope of this proof (and those similar to it, concerning the properties of stabilizers), any use of $f_i$ will mean a basis vector.}		
		Let $\alpha \in \mathbb{F}_{q}$ be a primitive element of $\mathbb{F}_{q}$ for which $\sigma(\alpha) \ne -\alpha$ for any nontrivial $\sigma \in \mathrm{Gal}(\mathbb{F}_{q}/\mathbb{F}_p)$; if $\mathbb{F}_q = \mathbb{F}_p$, we take $\alpha=1$.  Take $\Sigma_1=\{\overline{e_1}\}$, $\Sigma_2 = \{\overline{e_1+\frac{1}{\alpha}e_2 + f_1}\}$, $\Sigma_3=\{\overline{e_1+\alpha e_2}, \dots, \overline{e_{m-1}+\alpha e_m}\}$, and $\Sigma_4 = \{\overline{e_2},\dots,\overline{e_m},\overline{f_1},\dots,\overline{f_m}\}$.		
		
		As in the proof of Lemma \ref{lem:linear-invariants-all}, we now consider an element $\gamma \in \mathrm{\Gamma Sp}_{2m}(\mathbb{F}_q)$, which may be written as $\gamma = g \sigma$ for $g \in \mathrm{GL}_{2m}(\mathbb{F}_{q})$ and $\sigma \in \mathrm{Gal}(\mathbb{F}_{q}/\mathbb{F}_p)$.  We first note that if $\gamma$ stabilizes both $\Sigma_1$ and $\Sigma_4$, then $\gamma$ must stabilize $\overline{f_1}$ since it preserves the symplectic pairing up to a scalar and a Galois action.  Consequently, if $\gamma$ additionally stabilizes $\Sigma_3$, then it must stabilize $\overline{e_1+\alpha e_2}$, and hence also stabilize $\overline{e_2}$.  Iterating, we conclude that $g$ must be a diagonal matrix.  We then observe that if $\gamma$ stabilizes $\Sigma_2$, then $\gamma$ must in fact stabilize both $\overline{e_1+\frac{1}{\alpha}e_2}$ and $\overline{e_1 + f_1}$.  Since $\gamma$ stabilizes $\overline{e_1 + \frac{1}{\alpha} e_2}$, we then find, as in the proof of Lemma \ref{lem:linear-invariants-all}, that $\sigma$ must be trivial and $g$ must be a scalar matrix times $\mathrm{diag}(1,\dots,1,\lambda,\dots,\lambda)$, where $\lambda$ is the scalar such that $\langle gu, gv \rangle = \lambda \langle u,v\rangle$ for all $u,v \in \mathbb{F}_{q}^{2m}$.  However, since $\gamma$ also stabilizes $\overline{e_1+f_1}$, we then conclude that $\lambda = 1$ and that $g$ is scalar.  The result now follows from Theorem \ref{thm:stabilizer-invariants}.
	\end{proof}
	
	For large $q$, we have the following immediate consequence of Lemma \ref{lem:linear-invariants-large}.
	
	\begin{lemma}\label{lem:symplectic-invariants-large}
		Let $G$ be an almost simple group of type $\mathrm{PSp}_{2m}(\mathbb{F}_q)$ in its degree $n = (q^{2m}-1)/(q-1)$ action on $\mathbb{P}^{2m-1}(\mathbb{F}_q)$.  Assume that $\binom{q-1}{2m}(1-2q^{-\frac{2m-1}{2}}) > (2m+1 - \frac{1}{(2m)!})(q^{2m-1}+q-2) [ \mathbb{F}_q : \mathbb{F}_p]$ (so in particular also that $\frac{q}{[\mathbb{F}_q:\mathbb{F}_p]} \geq (2m+1)!$), where $\mathbb{F}_p$ denotes the prime subfield of $\mathbb{F}_q$.  Then there is a set of $n$ algebraically independent $G$-invariants $\{f_1,\dots,f_n\}$ with degrees:
			\begin{itemize}
				\item $\deg f_i = i$ for $1 \leq i \leq 2m+2$,
				\item $\deg f_i = 2m+4$ for $2m+3 \leq i \leq n - (2m+1)\cdot(2m+1)!$, and
				\item $\deg f_i \leq 4m+6$ for $n - (2m+1)\cdot (2m+1)! +1 \leq i \leq n$.
			\end{itemize}
	\end{lemma}
	\begin{proof}
		This follows from Lemma \ref{lem:linear-invariants-large}, as $G$ is a transitive subgroup of $\mathrm{P\Gamma L}_{2m}(\mathbb{F}_q)$.
	\end{proof}

\subsubsection{Unitary groups}

	For any $m \geq 3$, let $\mathrm{PSU}_m(\mathbb{F}_q)$ be the subgroup of $\mathrm{PGL}_m(\mathbb{F}_{q^2})$ preserving a nondegenerate unitary pairing on $\mathbb{F}_{q^2}^m$.  (Recall that $\mathrm{PSU}_2(\mathbb{F}_q)\simeq \mathrm{PSL}_2(\mathbb{F}_q)$.)  We begin by considering the natural primitive subspace representation of any almost simple group of type $\mathrm{PSU}_m(\mathbb{F}_q)$ on the set of totally singular lines in $\mathbb{F}_{q^2}^m$, which has degree $n = (q^m-(-1)^m)(q^{m-1}-(-1)^{m-1})/(q^2-1)$. 
	
	We split into two cases, depending on whether the rank of the group is even or odd.  The even case is essentially the same as the symplectic case, Lemma \ref{lem:symplectic-invariants-all}.
	
	\begin{lemma}\label{lem:unitary-invariants-even}
		Let $G$ be a unitary group of type $\mathrm{PSU}_{2m}(\mathbb{F}_q)$ as described above, viewed in its degree $n = (q^{2m}-1)(q^{2m-1}+1)/(q^2-1)$ action on totally singular lines.  Then there is a set of $n$ algebraically independent $G$-invariants $\{f_1,\dots,f_n\}$ satisfying:
			\begin{itemize}
				\item $\deg f_i = i$ for $1 \leq i \leq 7m+9$, and
				\item $\deg f_i \leq 7m+9$ for $ 7m+10 \leq i \leq n$.
			\end{itemize}
	\end{lemma}
	\begin{proof}
		The proof is essentially the same as that of Lemma \ref{lem:symplectic-invariants-all}, so we will be brief.  There is a basis $\{e_1,\dots,e_m,f_1,\dots,f_m\}$ for $\mathbb{F}_{q^2}^{2m}$ such that $\langle e_i,e_j\rangle = 0 = \langle f_i,f_j \rangle$ for any $i,j$ and $\langle e_i,f_j \rangle = \delta_{ij}$, where $\langle \cdot,\cdot\rangle$ denotes the unitary pairing and $\delta_{ij}$ is the Kronecker delta \cite[Proposition 2.3.2]{KleidmanLiebeck}.  Let $\alpha \in \mathbb{F}_{q^2}$ be a primitive element of $\mathbb{F}_{q^2}$ for which $\sigma(\alpha) \ne -\alpha$ for any nontrivial $\sigma \in \mathrm{Gal}(\mathbb{F}_{q^2}/\mathbb{F}_p)$.  Take $\Sigma_1=\{\overline{e_1}\}$, $\Sigma_2 = \{\overline{e_1+\frac{1}{\alpha}e_2 }, \overline{e_2+f_1}\}$, $\Sigma_3=\{\overline{e_1+\alpha e_2}, \dots, \overline{e_{m-1}+\alpha e_m}\}$, and $\Sigma_4 = \{\overline{e_2},\dots,\overline{e_m},\overline{f_1},\dots,\overline{f_m}\}$.	The argument of Lemma \ref{lem:symplectic-invariants-all} then shows that the intersection of the stabilizers $\Sigma_1,\dots,\Sigma_4$ is trivial, and the lemma follows.	
	\end{proof}
	
	We now turn to the odd case.
	
	\begin{lemma} \label{lem:unitary-invariants-odd}
		Let $G$ be a unitary group of type $\mathrm{PSU}_{2m+1}(\mathbb{F}_q)$ in its degree $n = (q^{2m+1}+1)(q^{2m}-1)/(q^2-1)$ action on the set of totally singular lines.  If $m=1$, then $\mathrm{base}(G) \leq 5$, and if $m \geq 2$, then there is a set of $n$ algebraically independent $G$-invariants $\{f_1,\dots,f_n\}$ satisfying:
			\begin{itemize}
				\item $\deg f_i = i$ for $1 \leq i \leq 7m+15$, and
				\item $\deg f_i \leq 7m+15$ for $ 7m+16 \leq i \leq n$.
			\end{itemize}
	\end{lemma}
	\begin{proof}
		By \cite[Proposition 2.3.2]{KleidmanLiebeck}, there is a basis $\{e_1,\dots,e_m,f_1,\dots,f_m,x\}$ for $\mathbb{F}_{q^2}^{2m+1}$ satisfying $\langle e_i,e_j\rangle = 0 = \langle f_i,f_j \rangle = \langle e_i,x\rangle = \langle f_i,x\rangle$ for any $i,j$, $\langle e_i,f_j \rangle = \delta_{ij}$, and $\langle x,x \rangle = 1$.  Let $a,b \in \mathbb{F}_{q^2}^\times$ be such that $b+b^q  + a^{q+1} = 0$; these conditions ensure that the vector $w:=e_1 + bf_1 + a x$ is singular.  Suppose first that $m \geq 2$.
		
		We take $\Sigma_1 = \{\overline{e_1}\}$, $\Sigma_2 = \{\overline{w}\}$, $\Sigma_3 = \{\overline{e_1+\frac{1}{\alpha}e_2}, \overline{e_2+f_1}\}$, $\Sigma_4=\{\overline{e_1+\alpha e_2}, \dots, \overline{e_{m-1}+\alpha e_m}\}$, and $\Sigma_5 = \{\overline{e_2},\dots,\overline{e_m},\overline{f_1},\dots,\overline{f_m}\}$.
		Ignoring $\Sigma_2$ for the time being, by following the proof of Lemma \ref{lem:unitary-invariants-even}, we conclude that any $\gamma \in \mathrm{\Gamma U}_{2m+1}(\mathbb{F}_q)$ must have trivial $\sigma$ and have matrix $g \in \mathrm{GL}_{2m+1}(\mathbb{F}_{q^2})$ that is a scalar $\lambda$ when restricted to the subspace spanned by $\{e_1,\dots,e_m,f_1,\dots,f_m\}$.  It follows that $g x = \varepsilon \lambda x$ for some $\varepsilon \in \mathbb{F}_{q^2}$ satisfying $\varepsilon^{q+1} = 1$.  Thus, if $\gamma$ stabilizes $\Sigma_3$, we conclude that $\varepsilon = 1$ and that $\gamma$ is scalar, completing the proof in the case $m \geq 2$.
		
		If $m=1$, then we let $a,b \in \mathbb{F}_{q^2}^\times$ satisfy $a^{q+1}+b+b^q = 0$ as above, and let $a_2 \in \mathbb{F}_{q^2}^\times$ be one of the other $q$ elements satisfying $a_2^{q+1} + b + b^q = 0$; set $w = e_1 + bf_1 + ax$ and $w_2 = e_1 + bf_1 + a_2x$.  Additionally, let $\alpha \in \mathbb{F}_{q^2}^\times$ be an element satisfying $\alpha + \alpha^q = 0$, and set $w_3 = e_1 + \alpha f_1$.  Then the lines $\overline{e_1}, \overline{f_1}, \overline{w}, \overline{w_2}, \overline{w_3}$ are all singular, and we claim there are choices of $\alpha, a, a_2, b$ for which they comprise a base.  Regardless of these choices, if $\gamma \in \mathrm{\Gamma U}_3(\mathbb{F}_q)$ stabilizes $\overline{e_1}$, $\overline{f_1}$, $\overline{w}$, and $\overline{w_2}$, then the underlying matrix $g \in \mathrm{GU}_3(\mathbb{F}_q)$ must be diagonal.  (Note that $\overline{w-w_2} = \overline{x}$.)  If moreover $\gamma$ stabilizes $\overline{w_3}$, then $g$ must be a scalar matrix times $\mathrm{diag}(1,\frac{\alpha}{\sigma(\alpha)},\frac{a - a_2}{\sigma(a-a_2)}\}$.  If $q$ is odd, then we may choose $b=1$ and $\alpha$ to be a generator of $\mathbb{F}_{q^2}^\times$.  With these choices, we find that if $\gamma$ stabilizes $\overline{w}$, then $\sigma(\alpha)=\alpha$, and hence $\sigma$ is trivial; this implies that $g$ is scalar, as desired.  If $q$ is even, then we instead take $\alpha = 1$ and $b$ to be a generator of $\mathbb{F}_{q^2}^\times$.  We then again conclude that $\sigma$ must be trivial for $\gamma$ to stabilize $\overline{w}$, completing the proof.
	\end{proof}
	
	Finally, we consider the natural representation of $\mathrm{PSU}_4(\mathbb{F}_q)$ on totally singular subspaces of dimension $2$.
	
	\begin{lemma} \label{lem:unitary-invariants-4}
		Let $G$ be a unitary group of type $\mathrm{PSU}_4(\mathbb{F}_q)$ in its degree $n=(q^3+1)(q+1)$ action on the set of totally singular subspaces of $\mathbb{F}_{q^2}^4$ with dimension $2$.  Then there is a set of $n$ algebraically independent $G$-invariants satisfying
			\begin{itemize}
				\item $\deg f_i = i$ for $1 \leq i \leq 25$, and
				\item $\deg f_i \leq 25$ for $26 \leq i \leq n$.
			\end{itemize}
	\end{lemma}
	\begin{proof}
		Let $e_1,e_2,f_1,f_2$ and $\alpha$ be as in the proof of Lemma \ref{lem:unitary-invariants-even}.  Given two vectors $v_1,v_2 \in \mathbb{F}_{q^2}^4$, we let $\langle v_1, v_2\rangle$ be the subspace spanned by $v_1$ and $v_2$.  Now, let 
		$\Sigma_1 = \{\langle e_1,f_2\rangle\}$, 
		$\Sigma_2 = \{\langle e_1+f_2,e_2-f_1 \rangle\}$, 
		$\Sigma_3 = \{ \langle e_1 + \alpha e_2, f_1 - \frac{1}{\alpha^q} f_2\rangle \}$, and 
		$\Sigma_4 = \{ \langle e_1 + \frac{1}{\alpha} e_2, f_1 - \alpha^q f_2 \rangle\}$, and 
		$\Sigma_5 = \{\langle e_1,e_2\rangle, \langle e_2,f_1\rangle, \langle f_1,f_2\rangle \}$.  We first observe that any element of $\mathrm{Stab}_G \Sigma_1 \cap \mathrm{Stab}_G \Sigma_5$ must stabilize the sets of lines $\{\overline{e_1},\overline{f_2}\}$ and $\{\overline{e_2},\overline{f_1}\}$.  Hence, if $\gamma \in \Gamma L_4(\mathbb{F}_{q^2})$ additionally stabilizes $\Sigma_4$, we must have that the associated matrix $g$ is diagonal.  And, if it also stabilizes $\Sigma_3$, that the element $\sigma \in \mathrm{Gal}(\mathbb{F}_{q^2}/\mathbb{F}_p)$ must be trivial.  If $g$ then also stabilizes $\Sigma_2$, we see that $g$ must be scalar, completing the proof.
	\end{proof}
	
\subsubsection{Orthogonal groups}

	We now turn to orthogonal groups in their natural action on totally singular lines.  
	As might be expected, we split into three cases, according to whether the group is of type $\mathrm{P \Omega}_{2m+1}$, $\mathrm{P \Omega}_{2m}^+$, or $\mathrm{P \Omega}_{2m}^-$.  We begin with the easiest of these, $\mathrm{P \Omega}_{2m}^+$, which we treat much as we did the symplectic groups and even unitary groups.  In this case, we may assume $m \geq 4$, since groups of type $\mathrm{P \Omega}_4^+$ are not simple and groups of type $\mathrm{P \Omega}_6^+$ are isomorphic to linear groups of rank $3$.  
	
	\begin{lemma} \label{lem:orthogonal-invariants-even-plus}
		Let $G$ be an orthogonal group of type $\mathrm{P\Omega}^+_{2m}(\mathbb{F}_q)$ as described above, viewed in its degree $n = (q^m-1)(q^{m-1}+1)/(q-1)$ action on totally singular lines.  Then there is a set of $n$ algebraically independent $G$-invariants $\{f_1,\dots,f_n\}$ satisfying:
			\begin{itemize}
				\item $\deg f_i = i$ for $1 \leq i \leq 7m+9$, and
				\item $\deg f_i \leq 7m+9$ for $ 7m+10 \leq i \leq n$.
			\end{itemize}
	\end{lemma}
	\begin{proof}
		The proof is the same as that of Lemma \ref{lem:unitary-invariants-even}, as there is again a basis $\{e_1,\dots,f_m\}$ satisfying $\langle e_i,e_j\rangle = 0 = \langle f_i,f_j\rangle$ and $\langle e_i,f_j \rangle = \delta_{ij}$ \cite[Proposition 2.5.3]{KleidmanLiebeck}.
	\end{proof}
	
	For orthogonal groups of odd rank $2m+1$, we proceed very similarly to the odd unitary case.  We may assume that $q$ is odd, as when $q$ is even there is an isomorphism $\mathrm{P\Omega}_{2m+1}(\mathbb{F}_q) \simeq \mathrm{PSp}_{2m}(\mathbb{F}_q)$, and we may assume that $m \geq 3$ due to the isomorphisms $\mathrm{P\Omega}_3(\mathbb{F}_q) \simeq \mathrm{PSL}_2(\mathbb{F}_q)$ and $\mathrm{P\Omega}_5(\mathbb{F}_q) \simeq \mathrm{PSp}_4(\mathbb{F}_q)$.
	
	\begin{lemma} \label{lem:orthogonal-invariants-odd}
		Let $G$ be an orthogonal group of type $\mathrm{P\Omega}_{2m+1}(\mathbb{F}_q)$ as described above, viewed in its degree $n = (q^{2m}-1)/(q-1)$ action on totally singular lines.  Then there is a set of $n$ algebraically independent $G$-invariants $\{f_1,\dots,f_n\}$ satisfying:
			\begin{itemize}
				\item $\deg f_i = i$ for $1 \leq i \leq 7m+15$, and
				\item $\deg f_i \leq 7m+16$ for $7m+16 \leq i \leq n$.
			\end{itemize}
	\end{lemma}
	\begin{proof}
		The proof is nearly identical to that of Lemma \ref{lem:unitary-invariants-odd}, as by \cite[Proposition 2.5.3]{KleidmanLiebeck} there is again a basis $\{e_1,\dots,e_m,f_1,\dots,f_m,x\}$ for $\mathbb{F}_{q}^{2m+1}$ satisfying $\langle e_i,e_j\rangle = 0 = \langle f_i,f_j \rangle = \langle e_i,x\rangle = \langle f_i,x\rangle$ for any $i,j$ and $\langle e_i,f_j \rangle = \delta_{ij}$, and where the vector $x$ is non-singular; in fact, we may assume that $\langle x, x \rangle =1$ by \cite[Proposition 2.5.4]{KleidmanLiebeck}.  If $a,b \in \mathbb{F}_q^\times$ are such that $a^2+2b=0$, then the vector $w := e_1 + bf_1 + ax$ is singular, and we proceed as in the proof of Lemma \ref{lem:unitary-invariants-odd}.
	\end{proof}
	
	We now turn to groups of type $\mathrm{P \Omega}^-_{2m}(\mathbb{F}_q)$, where we may assume that $m \geq 4$ by means of the isomorphisms $\mathrm{P\Omega}^-_{4}(\mathbb{F}_q) \simeq \mathrm{PSL}_2(\mathbb{F}_{q^2})$ and $\mathrm{P\Omega}^-_6(\mathbb{F}_q) \simeq \mathrm{PSU}_4(\mathbb{F}_q)$.
	
	\begin{lemma} \label{lem:orthogonal-invariants-even-minus}
		Let $G$ be an orthogonal group of type $\mathrm{P\Omega}^-_{2m}(\mathbb{F}_q)$ as described above, viewed in its degree $n = (q^m+1)(q^{m-1}-1)/(q-1)$ action on totally singular lines.  Then there is a set of $n$ algebraically independent $G$-invariants $\{f_1,\dots,f_n\}$ satisfying:
			\begin{itemize}
				\item $\deg f_i = i$ for $1 \leq i \leq 7m+15$, and
				\item $\deg f_i \leq 7m+15$ for $ 7m+16 \leq i \leq n$.
			\end{itemize}
	\end{lemma}
	\begin{proof}
		By \cite[Proposition 2.5.4]{KleidmanLiebeck}, there is a basis $\{e_1,\dots,e_{m-1},f_1,\dots,f_{m-1},x,y\}$ for $\mathbb{F}_q^{2m}$, where $\langle e_i,e_j\rangle = \langle f_i,f_j\rangle = \langle e_i,x\rangle = \langle f_i,x\rangle= \langle e_i,y\rangle = \langle f_i,y\rangle = 0$, $\langle x,y \rangle = 1$, $\langle x,x \rangle=1$, and $\langle y,y\rangle = \zeta$, where $\zeta$ is such that the polynomial $t^2+t+\zeta \in \mathbb{F}_q[t]$ is irreducible.  Let $w_x = e_1 - f_1 + x$ and $w_y = e_1 - \zeta f_1 + y$, and note that both $w_x$ and $w_y$ are singular.  Let $\alpha \in \mathbb{F}_q^\times$ be primitive if $q \ne p$ and let $\alpha = 1$ otherwise.  We then let $\Sigma_1 = \{\overline{e_1}\}$, $\Sigma_2 = \{\overline{w_x}\}$, $\Sigma_3 = \{\overline{w_y}\}$, $\Sigma_4 = \{\overline{e_1+\frac{1}{\alpha}e_2},\overline{e_2+f_1}\}$, $\Sigma_5=\{\overline{e_1+\alpha e_2}, \dots, \overline{e_{m-2}+\alpha e_{m-1}}\}$, and $\Sigma_6 = \{\overline{e_2},\dots,\overline{e_{m-1}},\overline{f_1},\dots,\overline{f_{m-1}}\}$.  As in previous proofs, we then observe that the intersection of these stabilizers is trivial.
	\end{proof}

\subsection{Invariants associated with Mathieu groups}
\label{subsec:mathieu-invariants}

	For the Mathieu groups, we are able to determine exactly a minimal set of degrees of a full set of algebraically independent invariants.

	\begin{lemma}\label{lem:mathieu-invariants}
		For the Mathieu groups $M_{11}$, $M_{12}$, $M_{22}$, $M_{23}$, and $M_{24}$, along with the group $M_{22}.2$ (which acts primitively in degree $22$) there is a set of algebraically independent invariants with degrees recorded in Table \ref{table:mathieu}.  The degrees of these invariants are minimal.
	\end{lemma}
	\begin{proof}
		By a computation in Magma, as described below.  The code is available at \url{http://lemkeoliver.github.io/}.
		
		We describe the computation for $M_{11}$, the other cases being materially the same.  First, we verify that there is a set of algebraically independent invariants with the claimed degrees (the largest of which is $8$).  To do so, we compute first a full set of linearly independent invariants of degree at most $8$; there are $71$ such invariants, with $41$ having degree at most $7$.  We then compute the reduced row echelon form of the $71 \times 11$ matrix formed by the partial derivatives of these $71$ invariants evaluated at a single random $11$-tuple of integers between $-10$ and $10$.  The degrees in Table \ref{table:mathieu} correspond to the degrees of the invariants associated with the non-zero rows in this reduced row echelon form, as the associated invariants must be algebraically independent by Lemma \ref{lem:independence-criterion}.
		
		To show that there are no $M_{11}$-invariants with smaller degree, we compute in each degree the dimension of the subspace generated by polynomials in the invariants of smaller degree.
		For example, since $M_{11}$ is $4$-transtive, the invariants of degree up to $4$ are generated by the power sums of these degrees.  Polynomials in these four invariants generate a linear subspace of the degree $5$ invariants with dimension $6$.  On the other hand, the dimension of the full space of degree $5$ invariants is $8$, so there can be at most $2$ invariants of degree $5$ that are algebraically independent from each other and the power sums of degree up to $4$.  And, in fact, there are indeed two such invariants, as the computation of the previous paragraph shows. 
		Exploiting this, we similarly compute that polynomials in the invariants of degree at most $5$ generate a codimension $2$ subspace of the degree $6$ invariants, showing that there can be at most $2$ invariants in degree $6$ that are algebraically independent from each other and those of smaller degree.
		Again, such invariants exist, and iterating this process shows that there cannot be a set of $11$ algebraically independent invariants with degrees smaller than those claimed.  The analogous procedure works for each of the groups in the lemma, completing the proof.
	\end{proof}

	\begin{table}[h]
		\begin{tabular}{|l|l|l|l|} \hline
			$G$ & $\mathrm{ind}(G)$ & $\{\deg f_i\}$ & $\sum \deg f_i$ \\ \hline
			$M_{11}$ & $4$ & $\{ 1, 2, 3, 4, 5^2, 6^2, 7^2,8\}$ ($=\{1,2,3,4,5,5,6,6,7,7,8\}$) & $54$ \\
			$M_{12}$ & $4$ & $\{1,2,3,4,5,6^2,7,8^2,9^2\}$ ($=\{1,2,3,4,5,6,6,7,8,8,9,9\}$) & $68$ \\
			$M_{22}$ & $8$ & $\{1,2,3,4^2,5^4,6^7,7^6\}$ & $118$ \\
			$M_{22}.2$ & $7$ & $\{1,2,3,4^2,5^3,6^6,7^8\}$ & $121$ \\
			$M_{23}$ & $8$ & $\{1,2,3,4,5^2,6^3,7^5,8^8,9\}$ & $146$ \\
			$M_{24}$ & $8$ & $\{1,2,3,4,5,6^2,7^2,8^4,9^5,10^6\}$ & $178$ \\ \hline
		\end{tabular}
		\caption{Minimal degrees of algebraically independent invariants of Mathieu groups}\label{table:mathieu}
	\end{table}

\subsection{Proof of Theorem \ref{thm:invariant-size}}
	\label{subsec:invariant-bound-proof}

	Let $G$ be an almost simple group in a natural primitive representation.  Any almost simple group is either alternating, classical, exceptional, or sporadic.  If $G$ is alternating, the conclusion of Theorem \ref{thm:invariant-size} is satisfied by taking the invariants to be the elementary symmetric polynomials.  If $G$ is either exceptional or sporadic, the conclusion is satisfied by taking the invariants from Corollary \ref{cor:easy-invariants}.  Finally, if $G$ is classical, say of rank $m$ over $\mathbb{F}_q$, then $\log |G| / \log n \ll m$.  Consequently, the result follows by taking the invariants from Lemmas \ref{lem:linear-invariants-all}--\ref{lem:orthogonal-invariants-even-minus} as appropriate.

\section{Bounds on almost simple extensions in natural representations}
		\label{sec:natural-bounds}
	
	In this section, we provide bounds on $G$-extensions for primitive groups $G$ that are almost simple groups in a natural primitive representation.
	These bounds are nearly immediate from the results in the previous two sections, apart from the bound we provide on $A_n$- and $S_n$-extensions, which requires recalling some results from \cite{LOThorne}.
	
	We begin by recording the following simple consequences of Theorem \ref{thm:invariant-theory-full}, Theorem \ref{thm:stabilizer-invariants}, and Lemma \ref{lem:replace-by-power-sums} that allow for a streamlined, though slightly not optimal, presentation of the relevant bounds.
	
	\begin{lemma}\label{lem:simplified-invariant-bound-degree}
		Let $G$ be a transitive group of degree $n$.  Suppose that the conclusion of Theorem \ref{thm:stabilizer-invariants} holds for some $w \geq 5$.  Then for any number field $k$ and any $X \geq 1$, we have
			\[
				\#\mathcal{F}_{n,k}(X;G)
					\leq (2\pi)^{dn/2} (d+1)!^n |G|^{dn} (2dn^3)^{dnw} X^{w-\frac{1}{2}-\frac{w(w-1)}{2n}} |\mathrm{Disc}(k)|^{-\frac{n}{2}},
			\]
		where $d=[k:\mathbb{Q}]$.
	\end{lemma}
	\begin{proof}
		Note that $|\mathcal{I}|(1) \leq |G|^n$ and $\deg \mathcal{I} = nw - \binom{w}{2} \leq n^2 - \binom{n}{2}$ for the invariants constructed by Theorem \ref{thm:stabilizer-invariants} and Lemma \ref{lem:replace-by-power-sums}.  The conclusion then follows directly from Theorem \ref{thm:invariant-theory-multiplicity}.
	\end{proof}
	
	The second consequence makes use of the notion of an elemental primitive group from Definition \ref{def:basic-elemental} and yields a stronger power of $X$ but a worse dependence on the other parameters.
	
	\begin{lemma}\label{lem:simplified-invariant-bound}
		Let $G$ be an elemental primitive group of degree $n$ and suppose the conclusion of Theorem \ref{thm:stabilizer-invariants} holds for some $w \geq 5$. 
		
		a) If $G$ does not contain $A_n$ and is not one of the exceptional cases in Lemma \ref{lem:elemental-index}, then for any number field $k$ and any $X \geq 1$, we have
			\begin{equation} \label{eqn:general-elemental-bound-n/4}
				\#\mathcal{F}_{n,k}(X;G)
					\leq (\sqrt{8 \pi} \cdot |G| 8^{|G|})^{dn} (d+1)!^n (2d n^3)^{dnw} \left(\frac{8}{n} \log X\right)^{db-1} X^{w - \frac{3}{2} - \frac{\binom{w}{2}-10}{n}} |\mathrm{Disc}(k)|^{\frac{n}{2}},
			\end{equation}
		where $d=[k:\mathbb{Q}]$ and $b=b(G)$ is the number of conjugacy classes of cyclic subgroups $G_0 \subseteq G$ with $\#\mathrm{Orb}(G_0) = n - \mathrm{ind}(G)$.
		
		b) If $G$ does not contain $A_n$ but is one of the exceptional cases in Lemma \ref{lem:elemental-index}, then for any number field $k$ and any $X \geq 1$, we have
			\begin{equation} \label{eqn:general-elemental-bound-3n/14}
				\#\mathcal{F}_{n,k}(X;G) \leq (\sqrt{8 \pi} |G| 8^{|G|})^{dn} (d+1)!^n (2d n^3)^{dnw} \left(\frac{28}{3n} \log X\right)^{db-1} X^{w - \frac{3}{2} - \frac{\binom{w}{2}-\frac{40}{3}}{n}} |\mathrm{Disc}(k)|^{\frac{n}{2}}.
			\end{equation}
	\end{lemma}
	\begin{proof}
		This readily follows from Theorem \ref{thm:stabilizer-invariants}, Lemma \ref{lem:replace-by-power-sums}, and Theorem \ref{thm:invariant-theory-full}, upon noting that $|\mathcal{I}|(1) \leq |G|^n$ and $\deg \mathcal{I} = nw - \binom{w}{2} \leq n^2 - \binom{n}{2}$.  Additionally, by assuming that $X \geq |\mathrm{Disc}(k)|^n$ (as we may), we see in the first case that
			\[
				\left(\frac{X}{|\mathrm{Disc}(k)|^n}\right)^{\frac{1}{\mathrm{ind}(G)}} \leq \left(\frac{X}{|\mathrm{Disc}(k)|^n}\right)^{\frac{4}{n}} = X^{\frac{4}{n}} |\mathrm{Disc}(k)|^{-4}.
			\]
		It follows from this that the factors of the discriminant present in the maximum in the statement of Theorem \ref{thm:invariant-theory-full} may be safely ignored, which yields the first claim.  An analogous computation reveals the second.
	\end{proof}
	
	We also have the following bound exploiting Theorem \ref{thm:invariant-theory-power-sums}, which we also state in a manner ensuring that the assumptions of the inductive argument (Proposition \ref{prop:induction}) are satisfied.
	
	\begin{lemma}\label{lem:simplified-power-sum-bound}
		Let $G$ be an elemental primitive group of degree $n$, not containing $A_n$, and suppose that the conclusion of Theorem \ref{thm:stabilizer-invariants} holds for some $w \geq 5$.  Then for any number field $k$ and any $X \geq 1$, 
			\[
				\#\mathcal{F}_{n,k}(X;G)
					\ll_{n,[k:\mathbb{Q}],G,\epsilon} X^{w - \frac{3}{2} - \frac{w(w-1)}{2n} + \frac{1}{\mathrm{ind}(G)}+\epsilon} |\mathrm{Disc}(k)|^{\frac{n}{2}}.
			\]
		In particular, we have also $\#\mathcal{F}_{n,k}(X;G) \ll_{n,[k:\mathbb{Q}]} X^{w-\frac{3}{2}} |\mathrm{Disc}(k)|^{\frac{n}{2}}$ and $\#\mathcal{F}_{n,k}(X;G) \ll_{n,[k:\mathbb{Q}],G} X^{w-1} |\mathrm{Disc}(k)|^{-(w-\frac{1}{2}+\frac{1}{2n})}$. 
	\end{lemma}
	\begin{proof}
		The first claim follows directly from Theorem \ref{thm:stabilizer-invariants}, Lemma \ref{lem:replace-by-power-sums}, Theorem \ref{thm:invariant-theory-power-sums}, and Lemma \ref{lem:elemental-index}.  The second follows upon using the assumption that $X \geq |\mathrm{Disc}(k)|^n$ (since otherwise $\mathcal{F}_{n,k}(X;G)$ is empty) and choosing $\epsilon = \frac{1}{3n}$.
	\end{proof}

\subsection{Classical groups}
	\label{subsec:classical-bounds}

	We cut to the chase:
	
	\begin{theorem}\label{thm:classical-bound}
		Let $G$ be an almost simple classical group in a natural primitive representation.  Then:
			\begin{enumerate}[i)]
				\item If $G$ is of type $\mathrm{PSL}_m(\mathbb{F}_q)$, then Lemma \ref{lem:simplified-invariant-bound-degree}, Lemma \ref{lem:simplified-invariant-bound}a), and Lemma \ref{lem:simplified-power-sum-bound} all hold with $w=5m+5$.
				\item If $G$ is of type $\mathrm{PSp}_{2m}(\mathbb{F}_q)$, then Lemma \ref{lem:simplified-invariant-bound-degree}, Lemma \ref{lem:simplified-invariant-bound}a), and Lemma \ref{lem:simplified-power-sum-bound} all hold with $w=7m+5$.
				\item If $G$ is of type $\mathrm{PSU}_{2m}(\mathbb{F}_q)$ (in its $1$-dimensional action if $m=2$) or $\mathrm{P\Omega}^+_{2m}(\mathbb{F}_q)$, then Lemma \ref{lem:simplified-invariant-bound-degree}, Lemma \ref{lem:simplified-invariant-bound}a), and Lemma \ref{lem:simplified-power-sum-bound} all hold with $w=7m+9$.
				\item If $G$ is of type $\mathrm{PSU}_4(\mathbb{F}_q)$ in its $2$-dimensional natural action, then Lemma \ref{lem:simplified-invariant-bound-degree}, Lemma \ref{lem:simplified-invariant-bound}a), and Lemma \ref{lem:simplified-power-sum-bound} all hold with $w=25$.
				\item If $G$ is of type $\mathrm{PSU}_{2m+1}(\mathbb{F}_q)$ or $\mathrm{P\Omega}_{2m+1}(\mathbb{F}_q)$ with $m \geq 2$, or of type $\mathrm{P\Omega}^-_{2m}(\mathbb{F}_q)$ with $q\ne 2$, then Lemma \ref{lem:simplified-invariant-bound-degree}, Lemma \ref{lem:simplified-invariant-bound}a), and Lemma \ref{lem:simplified-power-sum-bound} all hold with $w=7m+15$.
				\item If $G$ is of type $\mathrm{PSU}_3(\mathbb{F}_q)$, then Lemma \ref{lem:simplified-invariant-bound-degree}, Lemma \ref{lem:simplified-invariant-bound}a), and Lemma \ref{lem:simplified-power-sum-bound} all hold with $w=21$.
				\item If $G$ is of type $\mathrm{P\Omega}^-_{2m}(\mathbb{F}_2)$, then Lemma \ref{lem:simplified-invariant-bound-degree}, Lemma \ref{lem:simplified-invariant-bound}b), and Lemma \ref{lem:simplified-power-sum-bound} all hold with $w=7m+15$.
			\end{enumerate}
	\end{theorem}
	\begin{proof}
		By \cite[Theorem 7]{BurnessGuralnick}, the only almost simple groups $G$ in natural primitive representations for which $\mathrm{ind}(G) < \frac{n}{4}$ are of type $\mathrm{P\Omega}_{2m}^-(\mathbb{F}_2)$.  The theorem then follows by applying Lemma \ref{lem:simplified-invariant-bound} to the invariants from Lemmas \ref{lem:linear-invariants-all}, \ref{lem:symplectic-invariants-all}, \ref{lem:unitary-invariants-even}, \ref{lem:unitary-invariants-odd}, \ref{lem:unitary-invariants-4}, \ref{lem:orthogonal-invariants-even-plus}, \ref{lem:orthogonal-invariants-odd}, and \ref{lem:orthogonal-invariants-even-minus}.
	\end{proof}
	
	We also have the expected improvement for transitive subgroups of $\mathrm{P\Gamma L}_m(\mathbb{F}_q)$ when $q$ is sufficiently large in terms of $m$.  For convenience, we state the bound coming from Theorem \ref{thm:invariant-theory-power-sums}, but of course an explicit version also follows from Theorem \ref{thm:invariant-theory-full}.
	
	\begin{theorem}\label{thm:linear-bound-large}
		Let $G$ be a transitive subgroup of $\mathrm{P\Gamma L}_m(\mathbb{F}_q)$ in its degree $n=(q^m-1)/(q-1)$ action on $\mathbb{P}^{m-1}(\mathbb{F}_q)$, and suppose that $\binom{q-1}{m}(1-2q^{-\frac{m-1}{2}}) > (m+1 - \frac{1}{m!})(q^{m-1}+q-2) [ \mathbb{F}_q : \mathbb{F}_p]$ (so in particular also that $\frac{q}{[\mathbb{F}_q:\mathbb{F}_p]} \geq (m+1)!$), where $\mathbb{F}_p$ denotes the prime subfield of $\mathbb{F}_q$.  Then for any number field $k$ and any $X \geq 1$, 
			\[
				\#\mathcal{F}_{n,k}(X;G)
					\ll_{q,m,[k:\mathbb{Q}],\epsilon} X^{m + \frac{5}{2} + \frac{(m+1)\cdot(m+2)! - \binom{m+4}{2}}{n}+ \frac{1}{\mathrm{ind}(G)} + \epsilon} |\mathrm{Disc}(k)|^{\frac{n}{2}}.
			\]
	\end{theorem}
	\begin{proof}
		This follows from Theorem \ref{thm:invariant-theory-power-sums}, Lemma \ref{lem:linear-invariants-large}, Lemma \ref{lem:replace-by-power-sums}, and Lemma \ref{lem:elemental-index}.
	\end{proof}

\subsection{Alternating and symmetric groups}
	\label{subsec:alternating-symmetric-bounds}
	
	For alternating and symmetric groups, we still apply bounds coming from invariant theory, but (using an idea of Ellenberg and Venkatesh \cite{EV} as developed by Lemke Oliver and Thorne \cite{LOThorne}) not simply in the action of $G$ on $\mathbb{A}^n$; instead, it is in this case we consider $G$-invariants in its diagonal action on $(\mathbb{A}^n)^r$ for $r \geq 1$.  
	
	We begin by recalling the following from Lemke Oliver and Thorne \cite{LOThorne}.
	
	\begin{lemma}\label{lem:sn-invariants-r=2}
		Let $n \geq 2$, and let $w$ be the least integer for which $\binom{w+2}{2} \geq 2n+1$.  Then there is a set of $2n$ algebraically independent invariants $\{f_1,\dots,f_{2n}\}$ of $S_n$ in its diagonal action on $(\mathbb{A}^n)^2$ satisfying
			\[
				\sum_{i=1}^{2n} \deg f_i
					= 2nw - \frac{w(w-1)(w+4)}{6}.
			\]
		These invariants each satisfy $|f_i|(1)=n$.
	\end{lemma}
	\begin{proof}
		This follows from \cite[Lemma 3.1]{LOThorne}.  (See also the computation for the $r=2$ case of \cite[Theorem 1.2]{LOThorne}.)
	\end{proof}
	
	\begin{lemma}\label{lem:sn-invariants-large}
		Let $n \geq 2$, and let $w$ and $r\geq 3$ be integers satisfying $\binom{w+r-1}{r-1} \geq nr$, with $(w,r,n) \ne (3,5,7),(4,5,14)$.  Then there is a set of $nr$ algebraically independent invariants $\{f_1,\dots,f_{nr}\}$ of $S_n$ in its diagonal action on $(\mathbb{A}^n)^r$ with each $\deg f_i = w$ and $|f_i|(1)=n$.
	\end{lemma}
	\begin{proof}
		This follows from \cite[Lemma 3.3]{LOThorne}, which in turn relies on a theorem of Alexander and Hirschowitz \cite{AlexanderHirschowitz}.
	\end{proof}
	
	As a consequence of Theorem \ref{thm:invariant-theory-multiplicity}, we then find:
	
	\begin{theorem}\label{thm:alternating-symmetric-bound}
		Let $n \geq	2$.  Let $k$ be a number field, and set $d = [k:\mathbb{Q}]$. 
		
		1) If $w$ is the least integer for which $\binom{w+2}{2} \geq 2n+1$, then for any $X \geq 1$
			\[
				\#\mathcal{F}_{n,k}(X)
					\leq (\sqrt{2 \pi}\cdot  n)^{2dn} (d+1)!^{2n} \left(2dn \left(\deg \mathcal{I} + \binom{n}{2}\right)\right)^{d \cdot \deg \mathcal{I}} X^{2w - \frac{w(w-1)(w+4)}{6n} - 1} |\mathrm{Disc}(k)|^{-n},
			\]
		where $\deg \mathcal{I} = 2nw - \frac{w(w-1)(w+4)}{6}$. 
		
		2) If $w$ and $r\geq 3$ are integers satisfying $\binom{w+r-1}{r-1} \geq nr$, with $(w,r,n) \ne (3,5,7),(4,5,14)$, then
			\[
				\#\mathcal{F}_{n,k}(X)
					\leq (\sqrt{2 \pi}\cdot  n)^{dnr} (d+1)!^{nr} \left(2dn \left( nrw + \binom{n}{2} \right)\right)^{dnrw} X^{rw - \frac{r}{2}} |\mathrm{Disc}(k)|^{-nr/2}.
			\]
		3) There are absolute positive constants $c$, $c^\prime$, and $c^{\prime \prime}$ such that
			\[
				\#\mathcal{F}_{n,k}(X)
					\leq (2 dn^3)^{c^{\prime\prime} dn (\log n)^2} X^{ c (\log n)^2} |\mathrm{Disc}(k)|^{-c^\prime n \log n}.
			\]
		In fact, $c^\prime = 0.159$ and $c^{\prime\prime} = 1.847$ are admissible for any $n \geq 2$.  We may take $c=1.487$ if $n \geq 3$, and $c = 1/(\log 2)^2$ if $n \geq 2$.
	\end{theorem}
	\begin{proof}
		The first two claims follow directly by applying Theorem \ref{thm:invariant-theory-multiplicity} to the invariants provided by Lemmas \ref{lem:sn-invariants-r=2} and \ref{lem:sn-invariants-large}, respectively.  The third follows as in the proof of \cite[Theorem 1.1]{LOThorne}, but incoporating the minor improvements represented by the first two claims above into the computation.
	\end{proof}
	
	Theorem \ref{thm:alternating-symmetric-bound} both extends \cite[Theorem 1.2]{LOThorne} to bound extensions of arbitrary number fields $k$ and incorporates a small improvement over that result even when $k=\mathbb{Q}$.  By a comparison with \cite[Theorem 2]{BSW} when $k=\mathbb{Q}$ and the Schmidt bound (Corollary \ref{cor:schmidt-total}) when $k \ne \mathbb{Q}$, we see that Theorem \ref{thm:alternating-symmetric-bound} yields the best known bound on $\mathcal{F}_{n,k}(X)$ for every $n \geq 86$.  
	In fact, by also incorporating Lemma \ref{lem:skew-schmidt}, we obtain an improvement on the best known bounds in many smaller degrees as well: 
	
	\begin{theorem}\label{thm:small-degree}
		For every integer $n$ satisfying $20 \leq n \leq 85$, any number field $k$, and any $X \geq |\mathrm{Disc}(k)|^{2n-1}$, there are positive constants $\delta_{d,n}$ and $\gamma_{d,n}$ depending only on $d:=[k:\mathbb{Q}]$ and $n$ such that
			\[
				\#\mathcal{F}_{n,k}(X) \ll_{n,d} X^{\frac{n+2}{4} - \delta_{d,n}} |\mathrm{Disc}(k)|^{-\frac{3n}{4}-\gamma_{d,n}}.
			\]
		For given $d$, the constants $\delta_{d,n}$ can be taken to form an increasing sequence in $n$, with $\delta_{d,20} = \frac{33}{6384d+82}$ and $\delta_{d,85} = \frac{74733}{297024d+3772}$ being admissible. 
		For $d = 1$, the values $\delta_{1,20} = 0.005$, $\delta_{1,23} = 0.025$, $\delta_{1,30} = 0.062$, $\delta_{1,50} = 0.144$, $\delta_{1,70} = 0.207$, and $\delta_{1,85} = 0.248$ are admissible.
	\end{theorem}
	\begin{proof}
		As indicated, we compare the bounds resulting from Theorem \ref{thm:invariant-theory-multiplicity} applied to the invariants coming from Lemma \ref{lem:sn-invariants-r=2} with Lemma \ref{lem:skew-schmidt}.  As Lemma \ref{lem:skew-schmidt} only applies to primitive extensions, we first note that Corollary \ref{cor:schmidt-total} and Proposition \ref{prop:induction} together imply that the number of imprimitive extensions in $\mathcal{F}_{n,k}(X)$ is at most $O_{d,n}(X^{\frac{n+4}{8}} |\mathrm{Disc}(k)|^{-\frac{n+6}{4}})$, which is stronger than the claimed bound, so we may restrict attention to primitive extensions.
		
		Suppose $\lambda \geq X^{\frac{1}{2d(n-1)}} |\mathrm{Disc}(k)|^{-\frac{1}{2d(n-1)}}$ is a parameter to be specified later.  Then Lemma \ref{lem:largest-minimum} and our assumption that $X \geq |\mathrm{Disc}(k)|^{2n-1}$ imply that $\lambda \geq \lambda_{\max}(k)$.  We may therefore apply Lemma \ref{lem:skew-schmidt} to find
			\begin{equation} \label{eqn:skew-bound-1}
				\#\{ K \in \mathcal{F}_{n,k}^{\mathrm{prim}}(X) : \lambda_{\max}(K) > \lambda\}
					\ll_{d,n} \left(\frac{X}{\lambda^2}\right)^{\frac{n+2}{4}\left(1+\frac{1}{d(n-1)-1}\right)} |\mathrm{Disc}(k)|^{-\frac{3n}{4}-\frac{n+2}{4(d(n-1)-1)}},
			\end{equation}
		since our assumption that $X \geq |\mathrm{Disc}(k)|^{2n-1}$ implies that the maximum in the conclusion of Lemma \ref{lem:skew-schmidt} is $1$.  On the other hand, by Theorem \ref{thm:invariant-theory-multiplicity} applied to the invariants from Lemma \ref{lem:sn-invariants-r=2}, we find
			\begin{equation} \label{eqn:skew-bound-2}
				\#\{ K \in \mathcal{F}_{n,k}^{\mathrm{prim}}(X) : \lambda_{\max}(K) \leq \lambda\}
					\ll_{d,n} \lambda^{2dnw - \frac{dw(w-1)(w+4)}{6} -2dn}X |\mathrm{Disc}(k)|^{-n}.
			\end{equation}
		Comparing the complementary bounds \eqref{eqn:skew-bound-1} and \eqref{eqn:skew-bound-2} when $\lambda = X^{\frac{1}{2d(n-1)}} |\mathrm{Disc}(k)|^{-\frac{1}{2d(n-1)}}$, we see that \eqref{eqn:skew-bound-2} yields a stronger inequality, while \eqref{eqn:skew-bound-1} matches the conclusion of Corollary \ref{cor:schmidt-total}.  This implies that the optimal value of $\lambda$ is greater than $X^{\frac{1}{2d(n-1)}} |\mathrm{Disc}(k)|^{-\frac{1}{2d(n-1)}}$ and that some form of the theorem is true.  Computing the optimal value of $\lambda$ in Magma, say, proves the theorem as stated, with each $\delta_{d,n}$ and $\gamma_{d,n}$ for fixed n given by an explicit rational function in $d$.  The code verifying this computation, and the claims regarding the properties of $\delta_{d,n}$ and $\gamma_{d,n}$, is available at \url{http://lemkeoliver.github.io/}.
	\end{proof}
	
	\begin{remark}
		When $k=\mathbb{Q}$, the main result of \cite{BSW} shows that $\#\mathcal{F}_{n,\mathbb{Q}}(X) \ll_{n,\epsilon} X^{\frac{n+2}{4} - \frac{1-2^{-2g}}{2n-2}+\epsilon}$ for any $\epsilon > 0$, where $ g= \lfloor \frac{n-1}{2} \rfloor$, and this was previously the best known bound for $6 \leq n \leq 94$.  As remarked above, Theorem \ref{thm:alternating-symmetric-bound} improves on this for $n \geq 86$.  We therefore see that Theorem \ref{thm:small-degree} provides the best known bound on $\#\mathcal{F}_{n,\mathbb{Q}}(X)$ for $23 \leq n \leq 85$, while \cite[Theorem 2]{BSW} remains the best known for $6 \leq n \leq 22$.  Over number fields $k \ne \mathbb{Q}$, Theorem \ref{thm:small-degree} provides the best bounds for $n \geq 20$, while the Schmidt bound (or Corollary \ref{cor:schmidt-total}, if parameters other than $X$ are relevant) provide the best bounds for $6 \leq n \leq 19$.
	\end{remark}
	
\subsection{Exceptional and sporadic groups} 
	\label{subsec:exceptional-sporadic-bounds}
	
	For exceptional and sporadic groups, we have the following immediate consequence of Corollary \ref{cor:easy-invariants} and Lemma \ref{lem:simplified-invariant-bound}.
	
	\begin{theorem}\label{thm:exceptional-sporadic-bound}
		Let $G$ be an exceptional or sporadic almost simple group in a natural primitive representation.  Then \eqref{eqn:general-elemental-bound-n/4} and Lemma \ref{lem:simplified-power-sum-bound} hold with $w=28$.
	\end{theorem}
	\begin{proof}
		This follows from Corollary \ref{cor:easy-invariants} and Lemma \ref{lem:simplified-invariant-bound}.
	\end{proof}
	
	Stronger bounds than Theorem \ref{thm:exceptional-sporadic-bound} hold for many exceptional and sporadic groups $G$, as follows immediately from Corollary \ref{cor:base-invariants} and the results from \cite{Base-Exceptional} and \cite{Base-Sporadic}, respectively, which show that the inequality $\mathrm{base}(G) \leq 6$ behind Corollary \ref{cor:easy-invariants} is frequently not sharp.  Additionally, for the Mathieu groups, we show that by exploiting the minimal invariants provided by Lemma \ref{lem:mathieu-invariants}, it is possible to provide stronger bounds than those given by Corollary \ref{cor:bhargava-number-field}.  We write these bounds in a manner indicating the saving over Corollary \ref{cor:bhargava-number-field}.
		
	\begin{theorem} \label{thm:mathieu-bound}
		Let $k$ be a number field, and set $d = [k:\mathbb{Q}]$.  Then:
		\begin{enumerate}[i)]
			\item For any $X \geq |\mathrm{Disc}(k)|^{21}$, we have 
				\[
					\#\mathcal{F}_{11,k}(X;M_{11})
						\ll_{[k:\mathbb{Q}],\epsilon} X^{\frac{5}{2} - \frac{189}{2150d+100} + \epsilon} |\mathrm{Disc}(k)|^{\frac{4221}{8600d+400}}.
				\]
			In particular, $\#\mathcal{F}_{11,\mathbb{Q}}(X;M_{11}) \ll_\epsilon X^{\frac{302}{125}+\epsilon} = X^{\frac{5}{2}-\frac{21}{250}+\epsilon}$ for every $X \geq 1$.
			
			\item For any $X \geq |\mathrm{Disc}(k)|^{23}$, we have
				\[
					\#\mathcal{F}_{12,k}(X;M_{12})
						\ll_{[k:\mathbb{Q}],\epsilon} X^{\frac{11}{4}-\frac{375}{6776d+209}+\epsilon} |\mathrm{Disc}(k)|^{\frac{2850}{6776d+209}}.
				\]
			In particular, $\#\mathcal{F}_{12,\mathbb{Q}}(X;M_{12}) \ll_\epsilon X^{\frac{15067}{5588}+\epsilon} = X^{\frac{11}{4}-\frac{75}{1397}+\epsilon}$ for every $X \geq 1$.
			
			\item For any $X \geq 1$, we have
				\begin{align*}
					&\#\mathcal{F}_{22,k}(X;M_{22}) 
						 \ll_{[k:\mathbb{Q}],\epsilon} X^{\frac{351}{88}+\epsilon} |\mathrm{Disc}(k)|^{11} = X^{\frac{41}{8}-\frac{25}{22} + \epsilon} |\mathrm{Disc}(k)|^{11}, \\
					&\#\mathcal{F}_{22,k}(X;M_{22}.2) 
						 \ll_{[k:\mathbb{Q}],\epsilon} X^{\frac{29}{7}+\epsilon} |\mathrm{Disc}(k)|^{11} = X^{\frac{36}{7}-1+\epsilon} |\mathrm{Disc}(k)|^{11}, \\
					&\#\mathcal{F}_{23,k}(X;M_{23}) 
						 \ll_{[k:\mathbb{Q}],\epsilon} X^{\frac{915}{184}+\epsilon} |\mathrm{Disc}(k)|^{\frac{23}{2}} = X^{\frac{43}{8}-\frac{37}{92}+\epsilon} |\mathrm{Disc}(k)|^{\frac{23}{2}}, \text{ and }\\
					&\#\mathcal{F}_{24,k}(X;M_{24}) 
						 \ll_{[k:\mathbb{Q}],\epsilon} X^{\frac{145}{24}+\epsilon} |\mathrm{Disc}(k)|^{12} = X^{\frac{45}{8}-\frac{5}{12}+\epsilon} |\mathrm{Disc}(k)|^{12}.\\
				\end{align*}
		\end{enumerate}
	\end{theorem}
	\begin{proof}
		Since the groups $M_{22}$ and $M_{22}.2$ are $3$-transitive, $M_{23}$ is $4$-transitive, and $M_{24}$ is $5$-transitive, we may assume in these cases that the invariants of degree at most $3$ are power sums.  Case \emph{iii)} then follows immediately from Theorem \ref{thm:invariant-theory-power-sums} and Lemma \ref{lem:mathieu-invariants}.  The groups $M_{11}$ and $M_{12}$ are $4$- and $5$-transitive, respectively, so we may apply Theorem \ref{thm:invariant-theory-power-sums} in this case too.  The conclusion fails to be stronger than Corollary \ref{cor:bhargava-number-field}, however, so in these cases we proceed analogously to the proof of Theorem \ref{thm:small-degree} by leveraging Theorem \ref{thm:skew-bhargava} against Theorem \ref{thm:invariant-theory-power-sums}.
		
		For $M_{11}$, we use Theorem \ref{thm:skew-bhargava} to find for any $X \geq |\mathrm{Disc}(k)|^{21}$ and $\lambda \geq X^{\frac{1}{20d}} |\mathrm{Disc}(k)|^{-\frac{1}{20d}}$ that
			\[
				\#\{ K \in \mathcal{F}_{11,k}(X;M_{11}) : \lambda_{\max}(K) > \lambda\}
					\ll_{d,\epsilon} X^{\frac{5}{2}+\frac{63}{200d-20}+\epsilon} \lambda^{-\frac{63d}{10d-1}} |\mathrm{Disc}(k)|^{-\frac{63}{200d-20}}.
			\]
		On the other hand, appealing to Theorem \ref{thm:invariant-theory-power-sums} with $m=5$ and the invariants from Lemma \ref{lem:mathieu-invariants}, we find
			\[
				\#\{ K \in \mathcal{F}_{11,k}(X;M_{11}) : \lambda_{\max}(K) \leq \lambda\}
					\ll_{d,\epsilon} X^{-\frac{1}{4}+\epsilon} \lambda^{43d} |\mathrm{Disc}(k)|^{\frac{11}{2}}.
			\]
		Optimizing by taking $\lambda = X^{\frac{275d+4}{4300d^2+200d}} |\mathrm{Disc}(k)|^{\frac{-1100d+47}{8600d^2+400d}}$, we obtain the claimed bound on $\#\mathcal{F}_{11,k}(X;M_{11})$.
		
		Analogously, for $M_{12}$, we find for any $X \geq |\mathrm{Disc}(k)|^{23}$ and any $\lambda \geq X^{\frac{1}{22d}} |\mathrm{Disc}(k)|^{-\frac{1}{22d}}$ that
			\[
				\#\{ K \in \mathcal{F}_{12,k}(X;M_{12}) : \lambda_{\max}(K) > \lambda\}
					\ll_{d,\epsilon} X^{\frac{11}{4}+\frac{75}{242d+22}+\epsilon} \lambda^{-\frac{75d}{11d-1}}|\mathrm{Disc}(k)|^{-\frac{75}{242d+22}}
			\]
		and
			\[
				\#\{ K \in \mathcal{F}_{12,k}(X;M_{12}) : \lambda_{\max}(K) \leq \lambda\}
					\ll_{d,\epsilon} X^{-\frac{1}{4}+\epsilon} \lambda^{56d} |\mathrm{Disc}(k)|^{\frac{13}{2}}.
			\]
		Optimizing with $\lambda = X^{\frac{726d+9}{13552d^2+418d}} |\mathrm{Disc}(k)|^{\frac{-1452d+57}{13552d^2+418d}}$, the result follows.
	\end{proof}

\section{Bounds on general primitive extensions}	
	\label{sec:non-minimal}

	In this section, we provide upper bounds on the set $\mathcal{F}_{n,k}(X;G)$ for all primitive groups $G$, regardless of whether they are in a natural primitive representation.  In particular, our main goal in this section is to prove Theorem \ref{thm:general-bound-intro} in the case when $G$ is primitive:
	
	\begin{theorem}\label{thm:general-bound-primitive}
		There is an absolute constant $C$ such that if $G$ is a primitive group (say, of degree $n$), then $\#\mathcal{F}_{n,k}(X;G) \ll_{n,[k:\mathbb{Q}]} X^{C \cdot \widetilde{\mathrm{rk}}(G)} |\mathrm{Disc}(k)|^{-b(G)}$, where $b(G)>1$ depends only on $G$ and $\widetilde{\mathrm{rk}}(G)$ is defined by:
			\begin{itemize}
				\item $\widetilde{\mathrm{rk}}(G) = (\log n)^2$ if $G = A_n$ or $S_n$, in the natural degree $n$ action;
				\item $\widetilde{\mathrm{rk}}(G) = m$ if $G$ is an almost simple classical group of rank $m$ over a finite field;
				\item $\widetilde{\mathrm{rk}}(G) = 1$ if $G$ is an exceptional or sporadic almost simple group, or if $G$ is solvable;
				\item $\widetilde{\mathrm{rk}}(G) = 1 + \frac{m}{p}$ if $G$ is a non-solvable affine group of type $\mathbb{F}_p^m$; and
				\item $\widetilde{\mathrm{rk}}(G) = \frac{(\log n)^3}{\sqrt{n}}$ if $G$ is any other primitive group.
			\end{itemize}
	\end{theorem}
	
	Following the breakdown suggested by the O'Nan--Scott theorem (Theorem \ref{thm:onan-scott}), we proceed according to whether $G$ is almost simple, affine, non-elemental (which includes both product types and twisted wreath types), or diagonal.  In contrast to the previous section, where some care was placed in obtaining the explicit bounds provided, we prefer here to provide less optimized but more general results that lead to Theorem \ref{thm:general-bound-primitive}.  Our main objective in doing so is to show that bounds on almost simple groups in natural representations lead to comparably strong (or frequently stronger) bounds on all other primitive extensions.
	
	\subsection{Representation swapping}
	
	We treat many of the cases in Theorem \ref{thm:general-bound-primitive} by means of a technique we refer to as ``representation swapping'' (though this has appeared in earlier works, among them \cite[Proof of Proposition 1.3]{EV} and \cite[Proof of Theorem 24]{Bhargava-vdW}).  We begin with the following definition.
	
	\begin{definition}\label{def:swap-ratio}
		Let $G$ be a finite group, let $\pi_1$ and $\pi_2$ be permutation representations of $G$, say with degrees $n_1$ and $n_2$, respectively, and suppose that $\pi_1$ is faithful.  We define the \emph{swap ratio} between $\pi_1$ and $\pi_2$, denoted $\mathrm{swap}(\pi_1,\pi_2)$, to be
			\[
				\mathrm{swap}(\pi_1,\pi_2)
					:= \sup_{ \substack{ g \in G \\ g \ne 1}} \frac{n_2 - \#\mathrm{Orb}(\pi_2(g))}{n_1 - \#\mathrm{Orb}(\pi_1(g))},
			\]
		where, for any $g \in G$ and any permutation representation $\pi$, $\#\mathrm{Orb}(\pi(g))$ denotes the number of orbits of the cyclic subgroup generated by $g$ under the permutation representation $\pi$.
	\end{definition}
	
	We then have:
	
	\begin{lemma}[``Representation swapping'']\label{lem:representation-swapping}
		Let $G$ be a finite group, and let $\pi_1$ and $\pi_2$ be faithful, transitive permutation representations of $G$, say with degrees $n_1$ and $n_2$.  Define $G_1 = \pi_1(G)$ and $G_2 = \pi_2(G)$.  Then for any $X \geq 1$ and any number field $k$,
			\[
				\#\mathcal{F}_{n_1,k}(X;G_1)
					\leq \frac{n_1}{n_2} \#\mathcal{F}_{n_2,k}( 4^{dn_2^2}|G|^{dn_2} X^{\mathrm{swap}(\pi_1,\pi_2)} |\mathrm{Disc}(k)|^{n_2 - \mathrm{swap}(\pi_1,\pi_2)\cdot n_1}; G_2),
			\]
		where $d = [k:\mathbb{Q}]$.
	\end{lemma}
	\begin{proof}
		Suppose $K_1 \in \mathcal{F}_{n_1,k}(X;G_1)$.  Let $\widetilde{K_1}$ be the normal closure of $K_1$ over $k$, 
		and let $K_2$ be the subfield of $\widetilde{K_1}$ corresponding to $\pi_2$ (i.e., the subfield fixed by the subgroup $\mathrm{Stab}_{G_2}\{1\}$).  Since $\pi_2$ is faithful, the normal closure $\widetilde{K_2}$ of $K_2$ over $k$ satisfies $\widetilde{K_2} = \widetilde{K_1}$.  There are $n_1$ different $G_1$-extensions inside $\widetilde{K_2}$ (namely, the conjugates of $K_1$), while there are $n_2$ different $G_2$ extensions giving rise to the same normal closure (namely, the conjugates of $K_2$).  Hence, if $Y$ is such that $|\mathrm{Disc}(K_2)| \leq Y$ for any $K_2$ as above, then we find $\#\mathcal{F}_{n_1,k}(X;G_1) \leq \frac{n_1}{n_2} \#\mathcal{F}_{n_2,k}(Y;G_2)$.  Thus, it suffices to find an admissible value of $Y$, or, nearly equivalently, to bound $|\mathrm{Disc}(K_2)|$ in terms of $|\mathrm{Disc}(K_1)|$.
		
		For any prime $\mathfrak{p}$ of $k$, it follows from \eqref{eqn:tame-wild} that
			\[
				\frac{v_\mathfrak{p}(\mathfrak{D}_{K_2/k}^\mathrm{tame})}{v_\mathfrak{p}(\mathfrak{D}_{K_1/k}^\mathrm{tame})}
					= \frac{n_2 - \#\mathrm{Orb}(\pi_2(G_0))}{n_1 - \#\mathrm{Orb}(\pi_1(G_0))},
			\]
		where $G_0$ denotes the inertia subgroup of $\mathrm{Gal}(\widetilde{K_1}/k) \simeq G$ at $\mathfrak{p}$.  Hence, if $\mathfrak{p}$ does not divide $|G|$, then $v_\mathfrak{p}(\mathfrak{D}_{K_2/k}) \leq \mathrm{swap}(\pi_1,\pi_2) v_\mathfrak{p}(\mathfrak{D}_{K_1/k})$.  If we define
			\begin{equation} \label{eqn:swap-constant}
				c(\pi_1,\pi_2) 
					:=\sup_{\substack{ G_0 \subseteq G \\ \text{ $p$-by-cyclic}}} \frac{n_2 - \#\mathrm{Orb}(\pi_2(G_0)}{ n_1 - \#\mathrm{Orb}(\pi_1(G_0)))} -\mathrm{swap}(\pi_1,\pi_2),
			\end{equation}
		it follows that $|\mathfrak{D}_{K_2/k}^\mathrm{tame}| \leq |G|^{[k:\mathbb{Q}] \cdot c(\pi_1,\pi_2)} |\mathfrak{D}_{K_1/k}^{\mathrm{tame}}|^{\mathrm{swap}(\pi_1,\pi_2)}$ since $\prod_{\mathfrak{p} \mid |G|} |\mathfrak{p}| \leq |G|^{[k:\mathbb{Q}]}$.  By Lemma \ref{lem:wild-bound}, $|\mathfrak{D}_{K_2/k}^\mathrm{wild}| \leq 4^{[k:\mathbb{Q}]n_2^2}$, and we conclude that \[|\mathfrak{D}_{K_2/k}| \leq |G|^{[k:\mathbb{Q}] \cdot c(\pi_1,\pi_2)} 4^{[k:\mathbb{Q}]n_2^2} |\mathfrak{D}_{K_1/k}|^{\mathrm{swap}(\pi_1,\pi_2)}.\]
		We note that trivially we have $c(\pi_1,\pi_2) < n_2$, and upon recalling that $|\mathrm{Disc}(K_i)| = |\mathfrak{D}_{K_i/k}| \cdot |\mathrm{Disc}(k)|^{n_i}$ for $i=1,2$, the lemma follows.
	\end{proof}
	
	\begin{remark}
		There is a potential loss in the dependence on $n$, $d$, and $|G|$ in Lemma \ref{lem:representation-swapping} due to our essentially trivial handling of the wild part of the discriminant.  For situations in which this is relevant, one could additionally consider a ``wild swap ratio'' defined to be the supremum over $p$-subgroups $H \subseteq G$ of the ratio $\frac{n_2 - \#\mathrm{Orb}(\pi_2(H))}{n_1 - \#\mathrm{Orb}(\pi_1(H))}$ to account for the additional factors in \eqref{eqn:tame-wild}.  This would afford a slight improvement to Lemma \ref{lem:representation-swapping} that is not necessary for our purposes.
	\end{remark}
	
	We record a nearly trivial, but still frequently useful, bound on the swap ratio.
	 
	\begin{lemma}\label{lem:swap-ratio-bound}
		Let $G$ be a finite group, $\pi_1$ and $\pi_2$ permutation representations of $G$, assume $\pi_1$ is faithful, and let $G_1 = \pi_1(G)$ and $G_2 = \pi_2(G)$.  Then
			\[
				\mathrm{swap}(\pi_1,\pi_2)
					< \frac{n_2}{\mathrm{ind}(G_1)}.
			\]
		In particular, if $G_1$ is an elemental primitive group, then $\mathrm{swap}(\pi_1,\pi_2) < \frac{14}{3} \frac{n_2}{n_1}$, and if $G_1$ is not one of the exceptional types from Lemma \ref{lem:elemental-index}, then $\mathrm{swap}(\pi_1,\pi_2) < 4 \frac{n_2}{n_1}$.
	\end{lemma}
	\begin{proof}
		We find 
			\[
				\mathrm{swap}(\pi_1,\pi_2)
					\leq \frac{\sup_{1 \ne g \in G} n_2 - \#\mathrm{Orb}(\pi_2(g))}{\inf_{1 \ne g \in G} n_1 - \#\mathrm{Orb}(\pi_1(g))}
					< \frac{n_2}{\mathrm{ind}(G_1)}.
			\]
		The remaining claim follows from Lemma \ref{lem:elemental-index}.
	\end{proof}
		
	We begin with a demonstration of the utility of representation swapping to bounding non-elemental extensions.

\subsection{Non-elemental primitive groups}
	
	Recall from Definition \ref{def:basic-elemental} that a primitive permutation group $G$ of degree $n$ is said to be \emph{non-elemental} of type $(m,\ell,r)$ if $n=\binom{m}{k}^r$, there is a permutation representation $\pi_2 \colon G \hookrightarrow S_m \wr S_r$, and $G$ is permutation isomorphic to the action of $\pi(G)$ on $\Omega^r$, where $\Omega$ is the set of $k$-element subsets of $\{1,\dots,m\}$, and where $m \geq 3$, $\ell< m/2$ and at least one of $r$ and $\ell$ is greater than $1$.
	
	For such groups, we have a strong bound on the swap ratio between the primitive representation and the (often imprimitive) representation inside $S_m \wr S_r$ due to Bhargava \cite{Bhargava-vdW}:
	
	\begin{lemma}\label{lem:non-elemental-swap}
		Let $G$ be a finite group, and suppose that $\pi_1$ is a primitive permutation representation of $G$ of degree $n = \binom{m}{\ell}^r$ such that $\pi_1(G)$ is non-elemental of type $(m,\ell,r)$, and let $\pi_2\colon G \hookrightarrow S_m \wr S_r$ be the corresponding representation as above.  Then 
			\[
				\mathrm{swap}(\pi_1,\pi_2) < \frac{3rm}{n} \leq \frac{6}{\sqrt{n}}.
			\]
	\end{lemma}
	\begin{proof}
		The first inequality is \cite[Theorem 16]{Bhargava-vdW} in different terms.  The second is essentially an explicit form of \cite[Corollary 17]{Bhargava-vdW}, and follows upon noting that either $r\geq 2$ (in which case $mr \leq rn^{1/r} \leq 2\sqrt{n}$) or $r=1$ and $\ell\geq 2$ (in which case $mr \leq \sqrt{2n}+1 < 2\sqrt{n}$).
	\end{proof}
	
	Using this, we obtain a version of \cite[Theorem 23]{Bhargava-vdW} that holds over number fields.
	
	\begin{corollary}\label{cor:non-elemental-bound}
		Let $G$ be a non-elemental primitive group of degree $n$ and type $(m,\ell,r)$.  Then there are constants $c,c^\prime > 0$ so that for any number field $k$ and any $X \geq 1$, there holds
			\[
				\#\mathcal{F}_{n,k}(X;G)
					\ll_{n,[k:\mathbb{Q}]} X^{\frac{3cmr}{n} (\log mr)^2} |\mathrm{Disc}(k)|^{-2cmr (\log mr)^2 - c^\prime mr \log mr},
			\]
		and in particular $\#\mathcal{F}_{n,k}(X;G) \ll_{n,[k:\mathbb{Q}]} X^{\frac{6c (\log n)^2}{\sqrt{n}}} |\mathrm{Disc}(k)|^{-2cmr (\log mr)^2 - c^\prime mr \log mr}$.  The values $c=1.487$ and $c^\prime = 0.159$ are admissible.
	\end{corollary}
	\begin{proof}
		Let $\pi_1$ be the permutation representation corresponding to $G$, let $\pi_2$ correspond to its degree $mr$ representation, and let $G_2$ be the image of $\pi_2$.  Then by Lemma \ref{lem:representation-swapping} and Lemma \ref{lem:non-elemental-swap}, we obtain
			\[
				\#\mathcal{F}_{n,k}(X;G)
					\leq \frac{n}{mr} \#\mathcal{F}_{mr,k}( 4^{dm^2r^2} |G|^{dmr} X^{\frac{3mr}{n}} |\mathrm{Disc}(k)|^{-2rm}; G_2).
			\]
		The result then follows immediately from Theorem \ref{thm:alternating-symmetric-bound}, part \emph{3)}.
	\end{proof}

\subsection{Primitive groups with diagonal action}
	
	Following \cite[\S 4.5]{DixonMortimer}, a primitive group $G$ of degree $n$ is of diagonal type if there is a nonabelian simple group $T$ and an integer $m \geq 2$ such that the socle $N$ of $G$ (i.e., the subgroup generated by the minimal normal subgroups) is isomorphic to $T^m$, $n=|T|^{m-1}$, and $G$ is a subgroup of the extension of the wreath product $T \wr S_m$ by $\mathrm{Out}(T)$ (the outer automorphism group of $T$) satisfying certain conditions.  Moreover, we have in this case that if $m \geq 3$, then $N$ is the unique minimal normal subgroup of $G$ \cite[Corollary 4.3B]{DixonMortimer}.  We will again apply representation swapping to bound $G$-extensions, in a manner very analogous to our treatment of non-elemental groups.	
	
	\begin{lemma} \label{lem:diagonal-strategy}
		Let $G$ be a finite group with a primitive permutation representation $\pi_1$ of degree $n$ and diagonal type as above.  Then there is an absolute constant $c>0$ and a faithful permutation representation $\pi_2$ of $G$ satisfying $\deg \pi_2 \leq c n^{1/2} \log n$ and $\mathrm{swap}(\pi_1,\pi_2) \leq \frac{4c \log n}{\sqrt{n}}$.
	\end{lemma}
	\begin{proof}
		Since $G$ has a primitive representation of diagonal type, its socle $N$ is isomorphic to $T^m$ for some $m \geq 2$ and $T$ as above.  
		Let $G_0 = G \cap (T \wr S_m)$, and observe that $G/G_0$ is a subgroup of $\mathrm{Out}(T)$.  Moreover, we either have that $G_0 = T^2$ or that $G_0 = T \wr S$ for some primitive group $S$ of degree $m$ by \cite[Theorem 4.5A]{DixonMortimer}.  We consider these two cases separately.
		
		In the first case, if $H$ denotes the largest maximal subgroup of $T$, let $\pi_2$ be the direct sum of the permutation representations of $G$ corresponding to $H \times T$ and $T \times H$.  We then have that $\deg \pi_2 \leq 2|\mathrm{Out}(T)| [T:H]$.  As a consequence of the classification of finite simple groups, we have that $|\mathrm{Out}(T)| \ll \log |T| = \log n$ with an absolute implied constant.  Also from the classification, we have that $[T:H] \leq |T|^{1/2} = n^{1/2}$ \cite[Comment after 5.2.7]{KleidmanLiebeck}.  Hence, $\deg \pi_2 \ll n^{1/2} \log n$, as claimed.  The kernel of $\pi_2$ must be a normal subgroup contained in both $H \times T$ and $T \times H$, so must be trivial.  Thus, $\pi_2$ is faithful, completing the proof in this case.
		
		In the second case, where $G_0 = T \wr S$ for a primitive group $S$ of degree $m$, we proceed largely similarly.  Let $S_0 = \mathrm{Stab}_S\{1\}$, let $T_0 \simeq T$ be the corresponding direct summand of $N$, and let $G_1 \subseteq G_0$ be its normalizer, i.e. $G_1 = N_{G_0}(T_0)$.  Observe that $G_1 = T_0 \times T \wr S_0$ and $[G_0 : G_1] = m$.  Finally, if $H_0 \subseteq T_0$ is the largest maximal subgroup of $T_0$, let $G_2 = H_0 \times T \wr S_0$.  Let $\pi_2$ be the permutation representation of $G$ corresponding to $G_2$.
		
		We compute that $\deg \pi_2 \leq |\mathrm{Out}(T)| \cdot m \cdot [T_0:H_0]$.  As with the first case, we now see that $|\mathrm{Out}(T)| \ll \log |T| \ll \frac{\log n}{m}$ with absolute implied constants.  We also have $[T_0:H_0] \leq |T|^{1/2} \leq \sqrt{n}$ as in the first case.  The claim about $\deg \pi_2$ then follows, as does the claim about $\mathrm{swap}(\pi_1,\pi_2)$ upon appealing to Lemma \ref{lem:swap-ratio-bound}.  To see that $\pi_2$ is faithful, we observe that the core of $G_2$ in $G$, $\mathrm{Core}_G(G_2) = \cap_{g \in G} G_2^g$, does not contain $N$ because $G_2$ does not.  But $N$ is the unique minimal normal subgroup of $G$, hence $\mathrm{Core}_G(G_2) = 1$, and $\pi_2$ is faithful.
	\end{proof}
	
	\begin{corollary}\label{cor:diagonal-bound}
		There are absolute constants $c,c^\prime >0$ such that for any primitive group $G$ of diagonal type and degree $n$, any number field $k$, and any $X \geq 1$, there holds
			\[
				\#\mathcal{F}_{n,k}(X;G)
					\ll_{n,[k:\mathbb{Q}]} X^{\frac{c (\log n)^3}{\sqrt{n}}} |\mathrm{Disc}(k)|^{-c^\prime \sqrt{n} (\log n)^2} 
			\]
	\end{corollary}
	\begin{proof}
		This follows from Lemma \ref{lem:representation-swapping}, Lemma \ref{lem:diagonal-strategy}, and Theorem \ref{thm:alternating-symmetric-bound} in the case that the representation $\pi_2$ provided by Lemma \ref{lem:diagonal-strategy} is transitive.  If $\pi_2$ is not transitive, then it has two orbits, and in particular corresponds to an \'etale algebra expressible as the direct sum of two fields of degree $\frac{1}{2}\deg \pi_2$.  Multiplying together the number of choices for each such field yields the claim in this case too.
	\end{proof}
	
\subsection{Almost simple groups}
	\label{subsec:almost-simple-not-natural}
	
	We now turn to bounding almost simple extensions in their non-natural representations, since for the natural representations Theorem \ref{thm:general-bound-primitive} follows from the results in Section \ref{sec:natural-bounds}.  Since Corollary \ref{cor:non-elemental-bound} resolves Theorem \ref{thm:general-bound-primitive} for non-elemental primitive groups, we may assume going forward that all primitive groups under consideration are elemental.

	We first note that if $G$ is an almost simple group that is either exceptional or sporadic, then the conclusion of Theorem \ref{thm:general-bound-primitive} follows from Theorem \ref{thm:exceptional-sporadic-bound}.  Thus, we may assume that $G$ is either alternating or classical.  If $G$ is in a natural primitive representation, then the conclusion of Theorem \ref{thm:general-bound-primitive} follows from Theorems \ref{thm:classical-bound} and \ref{thm:alternating-symmetric-bound}.  Thus, in the case that $G$ is almost simple, it remains to consider the case that $G$ is an alternating group in a non-natural representation or a classical group in a non-natural representation.
	
	We begin by dispatching of non-natural representations of alternating and symmetric groups.
	
	\begin{lemma}\label{lem:alternating-swap}
		Suppose $G$ is almost simple of type $A_m$ for some $m \geq 15$.  Suppose $\pi_1$ is a non-natural, elemental primitive permutation representation of $G$ with degree $n$.  Let $\pi_m$ denote the natural degree $m$ representation of $G$.  Then $m \leq \frac{\log n}{\log 2}$ and
			\[
				\mathrm{swap}(\pi_1,\pi_m) < \frac{ 6\log n}{n}.
			\]
	\end{lemma}
	\begin{proof}
		Let $G_1 = \pi_1(G)$, and let $H = \mathrm{Stab}_{G_1}\{1\}$.  It follows from the O'Nan--Scott theorem (for example in the form proved by Liebeck, Praeger, and Saxl \cite{LiebeckPraegerSaxl}), that $H$ must be a primitive subgroup of $S_m$, since the actions of $G$ on imprimitive subgroups and on intransitive subgroups are non-elemental.  By \cite[Corollary 1.2]{Maroti}, we have $|H| < 3^m$, and hence $n=[G:H] > \frac{m!}{2\cdot 3^m}$.  By our assumption that $m \geq 15$, we find that $\frac{m!}{2\cdot3^m} \geq 2^m$.  Thus, $m \leq \frac{\log n}{\log 2}$ as claimed, with the remaining conclusion following from Lemma \ref{lem:swap-ratio-bound} and Lemma \ref{lem:elemental-index}.
	\end{proof}
	
	\begin{corollary}\label{cor:alternating-non-natural}
		There are constants $c,c^\prime>0$ such that for any almost simple group $G$ of type $A_m$ for some $m \geq 15$, in a non-natural, elemental primitive representation of degree $n$, and for any number field $k$ and any $X \geq 1$, there holds
			\[
				\#\mathcal{F}_{n,k}(X;G)
					\ll_{n,[k:\mathbb{Q}]} X^{\frac{c (\log n) (\log\log n)^2}{n}} |\mathrm{Disc}(k)|^{-c^\prime (\log n)(\log\log n)}.
			\]
	\end{corollary}
	\begin{proof}
		This is immediate from Lemma \ref{lem:alternating-swap} and Theorem \ref{thm:alternating-symmetric-bound}.
	\end{proof}
	
	We now turn to studying classical groups in non-natural representations.  For this, we have the following.
	
	\begin{lemma}\label{lem:classical-swap}
		Let $G$ be an almost simple classical group with natural module $V$, and let $\pi$ be any elemental primitive representation of $G$.  If $G \subseteq \mathrm{P \Gamma L}(V)$, let $\pi_{\mathrm{nat}}$ be the natural primitive representation of $G$ with minimal degree, and if not, let $\pi_{\mathrm{nat}}$ be the coset representation of $G$ on the stabilizer subgroup of the natural representation of $G \cap \mathrm{P \Gamma L}(V)$.  Then there is an absolute constant $C>0$ so that $\mathrm{swap}(\pi,\pi_{\mathrm{nat}}) \leq C$.
	\end{lemma}
	\begin{proof}
		Suppose first that $G \subseteq \mathrm{P \Gamma L}(V)$.  By \cite[Theorem 5.2.2]{KleidmanLiebeck} and \cite{VasilevMazurov}, $\pi_{\mathrm{nat}}$ is the primitive representation of $G$ with minimal degree unless $G$ is of type $\mathrm{PSp}_{2m}(\mathbb{F}_2)$ with $m \geq 3$, $\mathrm{P \Omega}_{2m+1}(\mathbb{F}_3)$, $\mathrm{P\Omega}^+_{2m}(\mathbb{F}_2)$, $\mathrm{PSU}_3(\mathbb{F}_q)$, $\mathrm{PSU}_n(\mathbb{F}_2)$, or $\mathrm{P\Omega}_{2m}^+(\mathbb{F}_3)$.  However, for each of these groups, we see that the primitive representation of minimal degree $\pi_{\min}$ satisfies $\frac{\mathrm{deg} \pi_{\mathrm{nat}} }{\mathrm{deg} \pi_{\min}} \ll 1$ for an absolute implied constant.  If $G \not\subseteq \mathrm{P \Gamma L}(V)$, then $[G:G \cap \mathrm{P \Gamma L}(V)] \leq 6$, and the result follows analogously.
	\end{proof}
	
	\begin{corollary} \label{cor:classical-bound}
		There are absolute constants $c,c^\prime$ so that for any almost simple classical group $G$ with natural module $\mathbb{F}_q^m$ in a primitive representation of degree $n$, any number field $k$, and any $X \geq 1$, there holds
			\[
				\#\mathcal{F}_{n,k}(X;G)
					\ll_{n,[k:\mathbb{Q}]} X^{c m} |\mathrm{Disc}(k)|^{-c^\prime n}.
			\]
	\end{corollary}
	\begin{proof}
		This is immediate from Lemma \ref{lem:classical-swap}, Lemma \ref{lem:representation-swapping}, and Theorem \ref{thm:classical-bound}.
	\end{proof}
	
	\begin{remark}
		Using the excellent monographs \cite{BrayHoltRoneyDougal,KleidmanLiebeck} along with work of Liebeck \cite{Liebeck}, it is possible to show that for any positive $\delta<1$ and any almost simple classical group, outside of some explicit families of primitive representations $\pi$ of degree $n$, the conclusion of Lemma \ref{lem:classical-swap} can be strengthened to show that $\mathrm{swap}(\pi,\pi_{\mathrm{nat}}) \ll n^{-\delta}$.  For such representations, the conclusion of Corollary \ref{cor:classical-bound} would be improved to $\#\mathcal{F}_{n,k}(X;\pi(G)) \ll_{n,[k:\mathbb{Q}]} X^{c \frac{\log n}{n^\delta}} |\mathrm{Disc}(k)|^{-c^\prime n^{1-\delta}} $ for some absolute constants $c,c^\prime$.  
	\end{remark}

\subsection{Affine groups}

	We now consider the case that $G$ is an elemental primitive group of affine type.  Thus, in this case, $G = G_0 \ltimes \mathbb{F}_p^m$ for an irreducible subgroup $G_0 \subseteq \mathrm{GL}_m(\mathbb{F}_p)$ (i.e., a subgroup that preserves no nontrivial subspaces of $\mathbb{F}_p^m$).  Setting $N$ to be the normal subgroup $\mathbb{F}_p^m$, the primitive representation of $G$ is in its degree $p^m$ action on $N$ and $G_0$ is the stabilizer of a point.  Moreover, the subgroup $N$ is the unique minimal normal subgroup of $G$.  (For more on primitive groups of affine type, see, for example, \cite[\S 4.7]{DixonMortimer}.)
	
	To handle such groups, we proceed as follows.  First, let $\Lambda = G_0 \cap Z(\mathrm{GL}_m(\mathbb{F}_p))$ consist of the scalar matrices in $G_0$.  Let $W \subseteq N$ be a subgroup of index $p$, which we may regard as a subspace of $\mathbb{F}_p^m$ of codimension $1$.  Let $H_0$ be the subgroup of $G_0$ stabilizing $W$, that is, $H_0 = \mathrm{Stab}_{G_0} W = N_{G_0}(W)$, and observe that $H_0$ contains $\Lambda$.  Define $\widetilde{H_0} := H_0 \ltimes N$ and $H := H_0 \ltimes W$.  The key lemma enabling our approach in this case is then the following.
	
	\begin{lemma}\label{lem:affine-strategy}
		With notation as above:
			\begin{enumerate}[i)]
				\item The subgroup $H$ has index $p$ inside $\widetilde{H_0}$.
				\item The degree $p$ coset action of $\widetilde{H_0}$ on $H$ is isomoprhic to a subgroup of $\mathbb{F}_p^\times \ltimes \mathbb{F}_p =: \mathrm{AGL}_1(\mathbb{F}_p)$.
				\item The coset representation of $G$ on $H$ is faithful.
			\end{enumerate}
	\end{lemma}
	\begin{proof}
		The first claim is immediate from the definitions of $H$ and $\widetilde{H_0}$.  The second follows by observing that the cosets of $\widetilde{H_0}/H$ are represented by $(e,0),(e,v),\dots,(e,(p-1)v)$ for any $v \in N \setminus W$ and where $e \in H_0$ is the identity. For the third, we observe that since $H$ does not contain $N$, the core $\mathrm{Core}_G(H) := \cap_{g \in G} H^g$ cannot contain $N$.  But since $N$ is the unique minimal normal subgroup of $G$ and $\mathrm{Core}_G(H)$ is normal, we see that $\mathrm{Core}_G(H) = 1$, and hence the coset action of $G$ on $H$ is faithful as claimed.
	\end{proof}
	
	Thus, to bound the number of extensions $K \in \mathcal{F}_{p^m,k}(X;G)$, we shall instead (where $\widetilde{K}$ denotes the normal closure of $K$ over $k$): bound the number of possible extensions $\widetilde{K}^{\widetilde{H_0}} =: F_0$, using that $F_0$ is a $PG_0 := G_0 / \Lambda$-extension of $k$, where $PG_0 \subseteq \mathrm{PGL}_m(\mathbb{F}_p)$ acts on subspaces of $\mathbb{F}_p^m$ of codimension $1$; and bound the number of possible extensions $F := \widetilde{K}^{H}$, using that $F$ is (at most) an $\mathrm{AGL}_1(\mathbb{F}_p)$-extension of $F_0$.  By Lemma \ref{lem:affine-strategy}, the extension $F/k$ determines $\widetilde{K}$ and hence $K$.
	
	Before executing this strategy, we record a classical fact that will be of use to us.
	
	\begin{lemma} \label{lem:dimension-codimension-1}
		The two permutation actions of $\mathrm{P \Gamma L}_m(\mathbb{F}_q)$ on subspaces of dimension $1$ and on subspaces of codimension $1$ are (permutation) isomorphic.  
	\end{lemma}
	\begin{proof}
		This fact is classical, and follows from the fact that any element of $\mathrm{GL}_m(\mathbb{F}_q)$ is conjugate to its transpose.
	\end{proof}
	
	\begin{lemma}\label{lem:affine-stabilizer-step}
		Let $G = G_0 \ltimes \mathbb{F}_p^m$ be an elemental affine group, let $\widetilde{H_0}$ be as above, and let $k$ be a number field.  Then, for any $X \geq 1$, as $K$ varies over elements of $\mathcal{F}_{p^m,k}(X;G)$, the number of different extensions $F_0 := \widetilde{K}^{\widetilde{H_0}}$ appearing is
			\[
				\ll_{p,m,[k:\mathbb{Q}]} X^{\frac{40m}{p}} |\mathrm{Disc}(k)|^{-\frac{7}{2} \frac{p^m-1}{p-1}}.
			\]
	\end{lemma}
	\begin{proof}
		Suppose first that $PG_0 = G_0 /\Lambda$ acts transitively on codimension $1$ subspaces of $\mathbb{F}_p^m$.  In this case, we see that $[F_0 : k] = [G_0 : H_0] = (p^m-1)/(p-1)$.  If we let $\pi_{p^m}$ be the representation of $G$ corresponding to $K$ and $\pi_0$ the non-faithful representation corresponding to $F_0$, then by Lemma \ref{lem:swap-ratio-bound} and Lemma \ref{lem:elemental-index}, we have $\mathrm{swap}(\pi_{p^m},\pi_0) \leq \frac{4 (p^m-1)}{p^m(p-1)} \leq \frac{8}{p}$.
		Following the proof of Lemma \ref{lem:representation-swapping}, we then see that 
			\begin{equation}\label{eqn:affine-disc-F0}
				|\mathrm{Disc}(F_0)| 
					\ll_{p,m,[k:\mathbb{Q}]} X^{\frac{8}{p}} |\mathrm{Disc}(k)|^{-3 \frac{p^m-1}{p-1}}.
			\end{equation}
		By Lemma \ref{lem:linear-invariants-all} and Lemma \ref{lem:dimension-codimension-1}, there is a full set of algebraically independent $PG_0$-invariants of degree at most $5m$.  By Theorem \ref{thm:invariant-theory-soft-general}, we therefore find that the number of choices for $F_0$ in this case is at most
			\[
				\ll_{p,m,[k:\mathbb{Q}]} X^{\frac{40m}{p}} |\mathrm{Disc}(k)|^{-\frac{7}{2} \frac{p^m-1}{p-1}},
			\]
		which is sufficient.
		
		Now, suppose that $PG_0$ does not act transitively on the codimension $1$ subspaces.  If there are $r$ orbits, say $\Omega_1,\dots,\Omega_r$, then we let $\omega_i \in \Omega_i$ be arbitrary and set $H_i = \mathrm{Stab}_{G_0} \omega_i$ for each $i \leq r$.  Thus, $H_0$ is conjugate to some $H_i$.  Setting $\widetilde{H_i} = H_i \ltimes N$ and $F_i = \widetilde{K}^{\widetilde{H_i}}$ for each $i \leq r$, we see that $\oplus_{i=1}^r F_i$ is a degree $\frac{p^m-1}{p-1}$ \'etale algebra over $k$ whose Galois group is permutation isomorphic to the action of $PG_0$ on subspaces of $\mathbb{F}_p^m$ with codimension $1$.  By Lemmas \ref{lem:linear-invariants-all} and \ref{lem:dimension-codimension-1}, there is a full set of algebraically independent invariants with degree $5m$ with respect to this action.  Now, by Lemma \ref{lem:swap-ratio-bound} and the idea behind the proof of Lemma \ref{lem:representation-swapping}, we see that $|\mathrm{Disc}(F_i)| \ll_{p,m,[k:\mathbb{Q}]} X^{\frac{4[F_i:k]}{p^m}} |\mathrm{Disc}(k)|^{-3[F_i:k]}$.  Hence, by Lemma \ref{lem:largest-minimum}, $\lambda_{\max}(F_i) \ll_{p,m,[k:\mathbb{Q}]} X^{\frac{4}{d \cdot p^m}} |\mathrm{Disc}(k)|^{-\frac{3}{d}}$, where $d = [k:\mathbb{Q}]$.  Following the proof of Theorem \ref{thm:invariant-theory-simple}, we then see that the number of possible \'etale algebras $\oplus_{i=1}^r F_i$ appearing in this manner is again
			\[
				\ll_{p,m,[k:\mathbb{Q}]} X^{\frac{40m}{p}} |\mathrm{Disc}(k)|^{-\frac{7}{2} \frac{p^m-1}{p-1}}.
			\]
		Since the \'etale algebra $\oplus_{i=1}^r F_i$ determines each of its summands, one of which is isomorphic to $F_0$, the result follows.
	\end{proof}
	
	\begin{lemma}\label{lem:affine-subspace-step}
		Let $G = G_0 \ltimes \mathbb{F}_p^m$ be an elemental affine group, let $\widetilde{H_0}$ and $H$ be as above, and let $k$ be a number field.  Then, for any $X \geq 1$, as $K$ varies over elements of $\mathcal{F}_{p^m,k}(X;G)$ with a specified choice of $F_0 = \widetilde{K}^{\widetilde{H}_0}$, the number of fields $F = \widetilde{K}^H$ appearing is $\ll_{p,m,[k:\mathbb{Q}]} X^{44} |\mathrm{Disc}(k)|^{-33 p^m}$.
	\end{lemma}
	\begin{proof}
		As follows from Lemma \ref{lem:affine-strategy}, the extension $F/F_0$ is a degree $p$ extension whose Galois group is a transitive subgroup of $\mathrm{AGL}_1(\mathbb{F}_p)$.  Call this subgroup $A$.  Since $\mathrm{AGL}_1(\mathbb{F}_p)$ has a base of size $2$, so does $A$, and we see from Corollary \ref{cor:base-invariants} and Theorem \ref{thm:invariant-theory-soft-general} that for any $T \geq 1$,
			\[
				\#\mathcal{F}_{p,F_0}(T;A)  
					\ll_{p,m,[k:\mathbb{Q}]} T^{\frac{11}{2}} |\mathrm{Disc}(F_0)|^{-\frac{p}{2}}
					\ll_{p,m,[k:\mathbb{Q}]} T^{\frac{11}{2}}.
			\]
		Since $[F:k] = p \cdot [F_0:k] \leq 2p^m$, by Lemmas \ref{lem:swap-ratio-bound} and \ref{lem:elemental-index}, we see that $\mathrm{swap}(\pi_{p^m},\pi_H) \leq 8$, where $\pi_{p^m}$ is the primitive degree $p^m$ representation of $G$ and where $\pi_H$ is the permutation representation corresponding to $H$.  Hence, we find
			\[
				|\mathrm{Disc}(F)| \ll_{p,m,[k:\mathbb{Q}]} X^8 |\mathrm{Disc}(k)|^{-6 p^m}.
			\]
		Taking this as $T$ above, we therefore find the number of choices for $F$ is
			\[
				\ll_{p,m,[k:\mathbb{Q}]} X^{44} |\mathrm{Disc}(k)|^{-33 p^m},
			\]
		as claimed.
	\end{proof}
	
	\begin{corollary} \label{cor:affine-bound}
		Let $k$ be a number field, and let $G$ be an elemental primitive affine group of type $\mathbb{F}_p^m$.  Then for any $X \geq 1$, we have
			\[
				\#\mathcal{F}_{p^m,k}(X;G)
					\ll_{p,m,[k:\mathbb{Q}]} X^{44 + \frac{40m}{p}} |\mathrm{Disc}(k)|^{-33p^m - \frac{7}{2}\frac{p^m-1}{p-1}}.
			\]
	\end{corollary}
	\begin{proof}
		This follows from Lemmas \ref{lem:affine-strategy}, \ref{lem:affine-stabilizer-step}, and \ref{lem:affine-subspace-step}.
	\end{proof}
	
	Finally, we also note that any primitive group that is solvable must be of affine type, but may be treated differently than other affine groups.  For these, we have the following easy consequence of Corollary \ref{cor:solvable-invariants} and Lemma \ref{lem:simplified-power-sum-bound}.
	
	\begin{corollary} \label{cor:solvable-bound}
		Let $G$ be a primitive solvable group of degree $n$.  Then for any number field $k$ and any $X \geq 1$, we have
			\[
				\#\mathcal{F}_{n,k}(X;G)
					\ll_{n,[k:\mathbb{Q}]} X^{14} |\mathrm{Disc}(k)|^{-\frac{29}{2}-\frac{1}{2n}}.
			\]
	\end{corollary}
	
\subsection{Proof of Theorem \ref{thm:general-bound-primitive}}

	This follows upon combining Theorem \ref{thm:classical-bound}, Theorem \ref{thm:alternating-symmetric-bound}, and Theorem \ref{thm:exceptional-sporadic-bound} (which treat almost simple groups in their natural representations), Corollaries \ref{cor:alternating-non-natural} and \ref{cor:classical-bound} (which treat almost simple groups not in a natural representation), Corollary \ref{cor:non-elemental-bound} (which treats non-elemental representations, in particular those of product or twisted wreath type), Corollary \ref{cor:diagonal-bound} (which treats diagonal primitive groups), Corollary \ref{cor:solvable-bound} (which treats solvable affine groups), and Corollary \ref{cor:affine-bound} (which treats all other affine groups).

\section{Proofs of theorems from the introduction}
	\label{sec:proofs-intro}
	
\subsection{Proof of Theorem \ref{thm:lie-type-intro}: Bounds on Lie-type extensions}
	This follows from Corollary \ref{cor:classical-bound} if $G$ is classical, and Theorem \ref{thm:exceptional-sporadic-bound} if $G$ is exceptional.
	
\subsection{Proof of Theorem \ref{thm:primitive-bound-improvement}: Bounds on primitive extensions}  This follows from Theorem \ref{thm:general-bound-primitive}.

\subsection{Proof of Corollary \ref{cor:vdw-improvement}: Bounds on non-$S_n,A_n$ polynomials}  This follows from \cite[Theorem 2]{Bhargava-vdW}, Theorem \ref{thm:general-bound-primitive}, Lemma \ref{lem:elemental-index}, and work of Widmer \cite{Widmer}.
	
\subsection{Proof of Theorem \ref{thm:solvable-intro}: Bounds on solvable extensions}
	
	We prove Theorem \ref{thm:solvable-intro} in the following explicit form.
	
	\begin{theorem}\label{thm:solvable-bound}
		Let $G$ be any transitive permutation group of degree $n$ that is solvable.  Then for any number field $k$ and any $X > 1$, we have
			\[
				\#\mathcal{F}_{n,k}(X;G)
					\ll_{n,[k:\mathbb{Q}]} X^{14} |\mathrm{Disc}(k)|^{-\frac{29}{2}}.
			\]
	\end{theorem}
	\begin{proof}
		Let $(G_1,\dots,G_m)$ be a tower type for $G$ where each $G_i$ is primitive of degree $n_i$.  Each $G_i$ is realizable as a section of $G$, and thus must also be solvable.  Combining Corollary \ref{cor:solvable-bound} with the inductive argument (Proposition \ref{prop:induction}) then yields the claim.
	\end{proof}
	
\subsection{Proof of Theorem \ref{thm:monster-intro}: Bounds on monstrous extensions}  
	From \cite{Base-Sporadic}, we have that $\mathrm{base}(\mathbb{M}) = 3$.  Thus, Theorem \ref{thm:monster-intro} follows from Corollary \ref{cor:base-invariants} and Lemma \ref{lem:simplified-power-sum-bound}, in fact in the form $\#\mathcal{F}_{n,k}(X;\mathbb{M}) \ll_{n,[k:\mathbb{Q}]} X^{\frac{17}{2}} |\mathrm{Disc}(k)|^{n/2}$.
	
\subsection{Proof of Theorem \ref{thm:general-bound-intro}: General bounds} 
	Let $G$ be a transitive permutation group with primitive tower type $(G_1,\dots,G_m)$.  We essentially wish to combine Theorem \ref{thm:general-bound-primitive} with Proposition \ref{prop:induction} to provide a bound on $\#\mathcal{F}_{n,k}(X;G)$.  To ensure that the hypotheses of Proposition \ref{prop:induction} are satisfied (in particular, the inequality $a_i < b_i + \frac{1}{2} + \frac{1}{2n_i}$), we note that if $G_i$ has $n_i\geq 2$, $k$ is a number field, and $X \geq |\mathrm{Disc}(k)|^{n_i}$, then, by Theorem \ref{thm:general-bound-primitive},
		\[
			\#\mathcal{F}_{n_i,k}(X;G_i)
				\ll_{n,[k:\mathbb{Q}]} X^{C \cdot \widetilde{\mathrm{rk}}(G_i)} |\mathrm{Disc}(k)|^{-b(G_i)}
				\leq X^{2C \cdot \widetilde{\mathrm{rk}}(G_i)} |\mathrm{Disc}(k)|^{-b(G_i)-n_iC \cdot \widetilde{\mathrm{rk}}(G_i)},
		\]
	where $C$ is any absolute constant for which Theorem \ref{thm:general-bound-primitive} holds.  The same inequality trivially holds if $X < |\mathrm{Disc}(k)|^{n_i}$, since in that case the left-hand side is $0$.  Since $b(G_i)>1$ and $n_i\geq 2$, we observe that $b(G_i)+n_iC \cdot \widetilde{\mathrm{rk}}(G_i) \geq 2C \cdot \widetilde{\mathrm{rk}}(G_i) + \frac{1}{2} + \frac{1}{2n_i}$.  Thus, each $G_i$ satisfies the hypotheses of Proposition \ref{prop:induction} with $a_i = 2C \cdot \widetilde{\mathrm{rk}}(G_i)$ and $b_i = b(G_i) + n_i C \cdot \widetilde{\mathrm{rk}}(G_i)$.  We thus conclude that
		\[
			\#\mathcal{F}_{n,k}(X;G)
				\ll_{n,[k:\mathbb{Q}]} X^{2C \cdot \widetilde{\mathrm{rk}}(G)},
		\]
	which gives Theorem \ref{thm:general-bound-intro}.  This moreover shows that if $C$ is any admissible constant in Theorem \ref{thm:general-bound-primitive}, then $2C$ is admissible in Theorem \ref{thm:general-bound-intro}.
	
\subsection{Proof of Corollary \ref{cor:all-degree-bound}: Bounds on all number fields}
	Recall that we are considering the set $\mathcal{F}(X) = \cup_{n \geq 2} \mathcal{F}_n(X)$ for $X \geq 10^6$.  First, taking $\sigma = 1.5$ in \cite[Equation (3.6)]{Odlyzko}, we see that the set $\mathcal{F}_n(X)$ is empty if $n > \log X$.  Moreover, comparing the bound from Theorem \ref{thm:degree-n-bound-intro} at consecutive integers $n$ and $n+1$, we see that it more than doubles in the latter.  Consequently, we see that $\#\mathcal{F}(X)$ may be bounded at worst by twice the bound from Theorem \ref{thm:degree-n-bound-intro} evaluated at $n = \log X$.  Simplifying the resulting expression yields the corollary.

\bibliographystyle{alpha}
\bibliography{references}
	
\end{document}